\documentclass{scrartcl}
\RequirePackage{amsthm,amsmath,amsfonts,amssymb}
\RequirePackage[numbers]{natbib}
\RequirePackage[colorlinks,citecolor=blue,urlcolor=blue,linkcolor=blue]{hyperref}
\RequirePackage{graphicx}
\usepackage[utf8]{inputenc}
\usepackage{lmodern} 
\usepackage{thmtools}
\usepackage{mathtools}
\usepackage{enumitem} 
\usepackage{dsfont} 
\usepackage{ifthen}
\usepackage{supertabular} 
\usepackage{booktabs}
\usepackage{longtable}
\usepackage{array}
\usepackage{colortbl}
\usepackage{pdflscape}
\usepackage{makecell}
\usepackage{xcolor}
\renewcommand{\subset}{\subseteq}
\newcommand{\lcb}{\left\lbrace} 
\newcommand{\rcb}{\right\rbrace} 
\newcommand{\cb}[1]{\lcb #1 \rcb} 
\newcommand{\cbOf}[1]{\mathopen{}\lcb #1 \rcb\mathclose{}} 
\newcommand{\lb}{\left(} 
\newcommand{\rb}{\right)} 
\newcommand{\br}[1]{\lb #1 \rb} 
\newcommand{\brOf}[1]{\!\br{#1}} 
\newcommand{\abs}[1]{\left| #1 \right|} 
\newcommand*{\E}{\mathbb{E}} 
\let\Pr\relax
\newcommand*{\Pr}{\mathbb{P}} 
\newcommand{\sizedMid}[2]{#1 \, \kern-\nulldelimiterspace\mathopen{}\left| \vphantom{#1}\,#2\right.\mathclose{}\kern-\nulldelimiterspace}

\newcommand{\PrOf}[1]{\Pr\mathopen{}\lb #1 \rb\mathclose{}}
\newcommand{\Prof}[1]{\Pr(#1)}

\DeclareMathOperator{\diam}{\mathsf{diam}}

\DeclareMathOperator{\ball}{\mathrm{B}}

\DeclarePairedDelimiterX\Set[1]{\lbrace}{\rbrace}%
{  #1 }
\newcommand{\Ex}{\E\expectarg}
\DeclarePairedDelimiterX{\expectarg}[1]{[}{]}{%
	\ifnum\currentgrouptype=16 \else\begingroup\fi
	\activatebar#1
	\ifnum\currentgrouptype=16 \else\endgroup\fi
}
\newcommand{\innermid}{\nonscript\;\delimsize\vert\nonscript\;}
\newcommand{\activatebar}{%
	\begingroup\lccode`\~=`\|
	\lowercase{\endgroup\let~}\innermid 
	\mathcode`|=\string"8000
}
\newcommand*{\mc}[1]{\mathcal{#1}}
\newcommand*{\mb}[1]{\mathbb{#1}}

\newcommand*{\ms}[1]{\mathsf{#1}}
\newcommand*{\mo}[1]{\mathbf{#1}}
\newcommand*{\mf}[1]{\mathfrak{#1}}
\newcommand{\N}{\mathbb{N}}

\newcommand{\R}{\mathbb{R}}

\newcommand{\Z}{\mathbb{Z}}
\newcommand{\transpose}{\!^\top\!}
\newcommand{\tr}{\transpose}
\newcommand{\pr}{^\prime}

\newcommand{\Exp}{\mathsf{Exp}}

\newcommand{\dl}{\mathrm{d}}
\newcommand{\ind}{\mathds{1}}
\newcommand{\norm}{|\cdot|}
\newcommand{\normof}[1]{|#1|}
\newcommand{\normOf}[1]{\left|#1\right|}
\newcommand{\equationFullstop}{\, .}
\newcommand{\eqfs}{\equationFullstop}
\newcommand{\equationComma}{\, ,}
\newcommand{\eqcm}{\equationComma}

\DeclareMathOperator*{\argmin}{arg\,min}

\newcommand{\ol}[2]{\overline{#1,\!#2}}

\newcommand{\old}[2]{d\brOf{#1,#2}}

\theoremstyle{plain}
\newtheorem{theorem}{Theorem}
\newtheorem{lemma}{Lemma}
\newtheorem{corollary}{Corollary}
\newtheorem{proposition}{Proposition}
\newtheorem{definition}{Definition}
\theoremstyle{remark}
\newtheorem{remark}{Remark}
\newtheorem{assumptions}{Assumptions}
\begin{document}
\title{Nonparametric Regression in Nonstandard Spaces}
\author{Christof Schötz}
\maketitle
\begin{abstract}
A nonparametric regression setting is considered with a real-valued covariate and responses from a metric space.
One may approach this setting via Fréchet regression, where the value of the regression function at each point is estimated via a Fréchet mean calculated from an estimated objective function. A second approach is geodesic regression, which builds upon fitting geodesics to observations by a least squares method. These approaches are applied to transform two of the most important nonparametric regression estimators in statistics to the metric setting -- the local linear regression estimator and the orthogonal series projection estimator.
The resulting procedures consist of known estimators as well as new methods. We investigate their rates of convergence in a general setting and compare their performance in a simulation study on the sphere. 
\end{abstract}
\tableofcontents
\section{Introduction}
Our goal is to estimate an unknown function $[0,1]\to\mc Q, t\mapsto m_t$, which is not of a simple parametric form, where $(\mc Q, d)$ is a general metric space. To this end, we have access to independent data $(x_i, y_i)_{i=1,\dots,n}$. We assume that the covariates are fixed as $x_i = \frac in$, and $y_i$ is a random variable with values in $\mc Q$ such that its \textit{Fréchet mean} is equal to $m_{x_i}$, i.e., $m_{x_i} = \argmin_{q\in\mc Q}\Ex{d(y_i, q)^2}$. We consider $\mc Q$ to be nonstandard, i.e., a metric space that is not isometric to a convex subset of a separable Hilbert space. Examples of nonstandard spaces are Riemannian manifolds, like the hypersphere $\mb S^k$, Hadamard spaces, like the space of phylogenetic trees \cite{billera01}, or Wasserstein spaces \cite{agueh11} in dimension greater than one. 

The literature on statistical analysis in nonstandard spaces is vast. We refer the reader to \cite{huckemann20a} for an overview and only present a small glimpse here. The \textit{Fréchet mean} \cite{frechet48} or \textit{barycenter} $m \in \argmin_{q\in\mc Q}\Ex{d(Y, q)^2}$ of a random variable $Y$ with values in the metric space $\mc Q$ lies at the heart of most analysis in nonstandard spaces. It can be viewed as a generalization of the Euclidean mean as $\Ex{X} = \argmin_{q\in\R^k}\Ex{\abs{X - q}^2}$ for a $\R^k$-valued random variable $X$ with $\Ex{\abs{X}^2} < \infty$. In Alexandrov spaces, the sample Fréchet mean is shown to attain the parametric rate of convergence under certain conditions \cite{gouic19}. In Hadamard spaces, the theory of Fréchet means \cite{sturm03} and algorithms for their calculation \cite{bacak14} are well described.
The Fréchet mean has been studied on Riemannian manifolds, e.g., \cite{bhattacharya03}. In this setting, \cite{eltzner19} (among others) show a central limit theorem. Nonparametric regression with metric target values is developed, e.g., in \cite{davis10, hein09, petersen16}. \cite{lin19} present a regression technique with regularization by total variation. \cite{steinke10} discuss nonparametric regression techniques between Riemannian manifolds. Specifically in the Riemannian manifold of symmetric positive-definite matrices, \cite{yuan12} develop a version of a local polynomial regression estimators, where higher order polynomials in this space are defined using parallel transport. Based on the notion of geodesics, \cite{fletcher13} introduces an analog of linear regression in symmetric Riemannian manifolds. These results are generalized and extended in \cite{cornea17}. 
\subsection{Model}\label{ssec:settings}
Let $(\mc Q, d)$ be a metric space. For $t \in [0, 1]$, let $Y_t$ be a $\mc Q$-valued random variable with finite second moment, i.e., $\Ex{d(Y_t, q)^2} < \infty$ for all $t\in[0,1]$ and $q\in\mc Q$. Let the regression function $m \colon [0,1] \to \mc Q$ be a minimizer $m_t \in \argmin_{q\in \mc Q} \Ex{d(Y_t, q)^2}$. 
Later, we will define certain smoothness conditions on $t\mapsto m_t$ (and on the change of the distribution of $Y_t$) to restrict the class of possible functions. We will consider nonparametric estimators which have access to following data: 
Let $x_i := \frac in$ and let $(y_i)_{i=1,\dots, n}$ be independent random variables with values in $\mc Q$ such that $y_i$ has the same distribution as $Y_{x_i}$. 

This model will be considered for two classes of metric spaces $\mc Q$: bounded metric spaces and Hadamard space. If a metric space $(\mc Q, d)$ fulfills $\sup_{q,p\in\mc Q} d(q,p) < \infty$, then it is called \textit{bounded}. This requirement simplifies the assumptions that require integrals of distances to be finite. Hadamard spaces are geodesic metric spaces of nonpositive curvature. Formally, a metric space $(\mc Q, d)$ is Hadamard if and only if it is complete, nonempty, and for all $q,p\in\mc Q$, there is $z \in \mc Q$ such that $d(y,z)^2 \leq \frac12 d(y,q)^2 + \frac12 d(y,p)^2- \frac14 d(q,p)^2$ for all $y\in\mc Q$. Hilbert spaces and complete simply-connected Riemannian manifolds of nonpositive sectional curvature are Hadamard, but also spaces without smooth structure like metric trees \cite[Proposition 3.4]{sturm03} or the space of phylogenetic trees \cite{billera01}.

To show the applicability in practice, the results are applied to the hyperspheres $\mb S^k$ and simulations are executed on the sphere $\mb S^2$.
\subsection{Two Approaches}
To construct an estimator for $t \mapsto m_t$, one may try to adapt a known Euclidean estimator to the new scenario. Two prominent approaches to this task are Fréchet regression \cite{petersen16} and geodesic regression \cite{fletcher13}.

\textbf{Fréchet Regression.}
The regression function $m_t$ is the Fréchet mean of $Y_t$, i.e., the minimizer of $\Ex{d(Y_t, q)^2}$ over $q\in\mc Q$. In Fréchet regression, we estimate the function $t \mapsto \Ex{d(Y_t, q)^2}$ for every fixed $q\in\mc Q$ by an Euclidean estimator $t \mapsto \hat F_t(q)$ using the data $(x_i, z_{q,i})_{i=1,\dots,n}\subset [0,1] \times \R$ with $z_{q,i} := d(y_i, q)^2$. In this step, we may use one of the standard nonparametric regression estimators for certain classes of functions $[0,1]\to\R$. Then $\hat F_t(q)$ is minimized over $q\in\mc Q$ for a fixed $t$ to obtain the estimator $\hat m_t$. 

\textbf{Geodesic Regression.}
Assume our metric space $\mc Q$ is equipped with an exponential map $\Exp \colon \Theta  \to \mc Q$, where $\Theta \subset \ms T\mc Q \subset \mc Q \times \R^k$ is a subset of the tangent bundle of $\mc Q$. A geodesic starting in point $p\in\mc Q$ and continuing in the direction $v \in \ms T_p\mc Q$ of the tangent space $\ms T_p\mc Q = \cb{u \in \R^k \colon (p, u) \in \ms T\mc Q}$ of $\mc Q$ at $p$ can be described as a function $\R \to\mc Q,\, x \mapsto \Exp(p, xv)$ with $(p,v) \in \ms T\mc Q$. In geodesic regression with covariates $x_i \in \R$, we minimize the empirical squared error 
\begin{equation}
	\sum_{i=1}^n d(y_i, \Exp(p, x_i v))^2
\end{equation}
 over $(p,v)\in \Theta$ to find the best fitting geodesic. All forms of geodesic regression built on this criterion or a modification of it. For example, we can extend it to multivariate regression
\begin{equation}
	\sum_{i=1}^n d\brOf{y_i, \Exp\brOf{p, \sum_{j=1}^J x_{i,j} v_j}}^2\eqcm
\end{equation}
where $x_{i}\in\R^J$ and $v_1, \dots, v_J \in \ms T_p\mc Q$ or more general feature regression
\begin{equation}
	\sum_{i=1}^n d\brOf{y_i, \Exp\brOf{p, \sum_{j=1}^J \psi_j(x_i) v_j}}^2\eqcm
\end{equation}
where $x_i \in\mc X$ for an arbitrary space of covariates $\mc X$ and features $\psi_j \colon \mc X \to \R$.
Furthermore, we may introduce weights $w_{i,t}$, e.g., $w_{i,t} = K((x_i-t)/h)$ for a kernel $K$ and a bandwidth $h >0$ to localize the procedure, and obtain (here for one-dimensional covariates)
\begin{equation}
	(\hat m_t, \hat {\dot m}_t) = \argmin_{(p, v) \in \Theta}\sum_{i=1}^n w_{i,t} d(y_i, \Exp(p, x_i v))^2
	\eqfs
\end{equation}
In this paper, we do not require the existence of an exponential map in the sense of Riemannian geometry. Instead, $\Exp\colon \Theta \to \mc Q, \Theta \subset \mc Q \times \R^k$ is required to fulfill certain distance bounds as described in our results on geodesic regression.
\subsection{Contribution}
We compare the two approaches of geodesic (\texttt{Geo}) and Frechet (\texttt{Fre}) regression on two regression estimators, namely local linear regression (\texttt{Loc}) and the orthogonal series estimator (\texttt{Ort}). This makes four estimation procedures, which we refer to as \texttt{LocGeo}, \texttt{LocFre}, \texttt{OrtGeo}, and \texttt{OrtFre}. 
For the resulting estimators, which we denote as $\hat m_t$, our goal is to show explicit finite sample bounds of the mean integrated squared error (MISE) of the form $\int_0^1\Ex{ d(m_t, \hat m_t)^2} \dl t  \leq C n^{-\alpha}$ for constants $C,\alpha > 0$. We are not interested in optimal universal constants, but the dependence on further parameters, like a moment bound, is to be explicit.
For \texttt{LocGeo}, \texttt{LocFre}, and \texttt{OrtFre} we find $\int_0^1\Ex{ d(m_t, \hat m_t)^2} \dl t  \leq C n^{-\frac{2\beta}{2\beta+1}}$, where $\beta>0$ is a smoothness parameter. Regarding the smoothness condition, we consider different models for different estimators. In particular, $\beta$ has a somewhat different meaning for each estimator. Nonetheless, the results are comparable and the optimal nonparametric rate of convergence $n^{-\frac{2\beta}{2\beta+1}}$ is shown to hold in these three cases.
\begin{itemize}
\item \texttt{LocFre}  (section \ref{sec:locfre}):
\cite{petersen16} introduce local constant (Nadaraya--Watson) and local linear Fréchet regression for general bounded metric spaces. For the local linear estimator, they show $d(\hat m_t, m_t) \in \mo O_\Pr(n^{-\frac25})$ and a more general version of this result, see Corollary 1 in their article. We show, for a general local polynomial Fréchet estimator of order $\ell\in\N_0$, the point-wise error bound $\Ex*{d(m_{t}, \hat m_{t})^2} \leq C n^{-\frac{2\beta}{2\beta+1}}$ for a constant $C > 0$ and a smoothness parameter $\beta > 0$, $\lfloor\beta\rfloor = \ell$, which implies the same rate for the MISE, \autoref{cor:locfre:bounded}, \autoref{cor:locfre:hadamard}. Our results are slightly more general with conditions slightly less demanding. Furthermore, bounds in expectation for finite $n$ are stronger than in $\mo O_\Pr$ and are needed to make the error bound of this estimator comparable to the others. As \cite{petersen16}, we demand a smoothness condition not directly on $t \mapsto m_t$, but on the change of the probability density of $Y_t$ in $t$.
\item \texttt{OrtFre} (section \ref{sec:trifre}):
We apply the approach of Fréchet regression to the orthogonal series projection estimator and arrive at a new estimator, \texttt{OrtFre}. For the trigonometric series as instance of an orthogonal series, we show $\Ex{\int_0^1 d(m_t, \hat m_t)^2 \dl t} \leq C n^{-\frac{2\beta}{2\beta + 1}}$ for a smoothness parameter $\beta \geq 1$ and a constant $C > 0$, \autoref{cor:trifre:bounded}, \autoref{cor:trifre:hadamard}. As for \texttt{LocFre} the smoothness condition is a requirement on the change of the density of $Y_t$ in $t$.
\item \texttt{LocGeo} (section \ref{sec:locgeo}):
We apply the approach of geodesic regression to the classical local linear estimator to obtain \texttt{LocGeo}. A local polynomial regression estimator of arbitrary order in the Riemannian manifold of symmetric positive definite matrices was already introduced in \cite{yuan12}. In contrast, the results here are restricted to a first order expansion, but they are applicable to a wide range of metric spaces. We show a point-wise error bound of $\Ex{d(m_t,\hat m_t)^2} \leq C n^{-\frac{2\beta}{2\beta+1}}$ for all $t\in[0,1]$, a smoothness parameter $\beta\in(1,2]$, and a constant $C>0$, which implies the same bound on the mean integrated squared error, \autoref{cor:locgeo:bounded}, \autoref{cor:locgeo:hadamard}. For this result, we assume a smoothness condition, which generalizes the Hölder condition that is common for local linear estimators. It demands that the true function $t\mapsto m_t$ can be locally approximated at $t$ by a geodesic up to an error of order $|x-t|^\beta$ for $x$ close to $t$.
\end{itemize}
In section \ref{sec:trigeo}, we discuss a construction of an \texttt{OrtGeo} estimator: We apply the geodesic regression approach to the orthogonal series projection estimator. We do not show optimal rates of convergence, and argue that this estimator may be sub-optimal as the properties that make it appealing in Euclidean spaces are lost in nonstandard spaces. Nonetheless, we include an estimator with the trigonometric series as the chosen orthogonal series in our simulation study.

Our goal is to make all theorems as general as reasonably possible. This manifests in quite abstract statements. To get a gist of the meaning of the abstract objects, we apply the general theorems on the hypersphere: \autoref{cor:locfre:sphere}, \autoref{cor:trifre:sphere}, and \autoref{cor:locgeo:sphere}. These corollaries illustrate our results and show that they are indeed applicable to explicit and interesting nonstandard spaces. Furthermore, abstract assumptions of the general theorems are justified by showing that they are fulfilled on the hyperspheres. 

The sphere is also the metric space used in our simulation study, section \ref{sec:simu}. To fulfill a variance inequality, which is an assumption for all our results, we introduce a new family of distributions on the sphere, the \textit{contracted uniform distributions}. All estimators are implemented using the statistical programming language R \cite{r}. The resulting package is freely available at \url{https://github.com/ChristofSch/spheregr}. Our experiments confirm and illustrate the theoretical findings. 

The proofs of all results can be found in the appendix \ref{sec:proofs}.
They partially built upon techniques developed in \cite{schoetz19}. The major tools to prove results in this setting are empirical process theory with chaining, e.g. \cite{vaart96} or \cite{talagrand14}, and a technique called \textit{slicing} or \textit{peeling}, e.g., \cite{geer00}. The proofs for local regression techniques partially follow the Euclidean version in \cite[section 1.6]{tsybakov08}, for trigonometric regression we build upon \cite[section 1.7]{tsybakov08}.
\subsection{Notation and Conventions}
Assumptions are named in small caps, e.g., \textsc{Moment}. The names of the presented methods are set in a typewriter font, e.g., \texttt{LocFre}.

We use a lower case $c$ for universal constants $c>0$. If the value depends on a variable, we indicate this by an index, e.g., $c_\kappa$ is a constant that depends only on $\kappa$. We do not specify the values of such constants. They are silently understood to take an appropriate value. Furthermore, the value may vary between two occurrences of such a constant.

A capital $C$ indicates a constant that has further meaning, which is usually described by a three letter index, e.g., we may require a moment condition $\Ex{d(Y_t, m_t)^2}\leq C_{\ms {Mom}}$ for all $t$ to be fulfilled. For simplicity, we assume these constants to be $\geq 1$, so that, e.g., $C_{\ms{Abc}}^2 + C_{\ms{Abc}} C_{\ms{Xyz}} \leq c C_{\ms{Abc}}^2 C_{\ms{Xyz}}$.

There is a silently underlying probability space  $(\Omega, \Sigma_{\Omega}, \Pr)$. If a random variable, say $Y$, has values in a set, say $\mc Y$, that set is silently understood to be a measurable space $(\mc Y, \Sigma_{\mc Y})$ and the random variable is a measurable map $Y \colon (\Omega, \Sigma_{\Omega}) \to (\mc Y, \Sigma_{\mc Y})$.

In each section, the estimator of the regression function at $t$ is denoted as $\hat m_t$. It depends on $n$ and potentially on further parameters like a bandwidth $h$, which will not be indicated in the notation but should be clear in the context.

For a vector $v \in \R^k$, we denote its Euclidean norm by $|v|$. For $\beta\in\R$, let $\lfloor \beta \rfloor$ be the largest integer strictly smaller than $\beta$.
Let $(\mc Q, d)$ be a metric space. To shorten the notation, we sometimes write $\ol qp$ instead of $d(q,p)$ for $q,p\in\mc Q$. Define the ball $\ball(o, d, \delta) := \{q\in\mc Q \colon \ol qo < \delta\}$ and the diameter $\diam(\mc Q, d) := \sup_{q,p\in\mc Q} \ol qp$. 

For the theorems below, we need a quantification of the entropy of the metric space $\mc Q$. To this end, we use Talagrands's $\gamma_2$ \cite{talagrand14} as defined below.
\begin{definition}\label{def:entropy}
\mbox{ }
\begin{enumerate}[label=(\roman*)]
\item 
	Given a set $\mc Q$, an \emph{admissible sequence} is an increasing
	sequence $(\mc A_k)_{k\in\N_0}$ of partitions of $\mc Q$ such that $\mc A_0 = \{\mc Q$\} and the cardinality of $\mc A_k$ is bounded as $\#\mc A_k \leq 2^{2^k}$ for $k\geq 1$.
	
	By an increasing sequence of partitions we mean that every set of $\mc A_{k+1}$ is
	contained in a set of $\mc A_k$. We denote by $A_k(q)$ the unique
	element of $\mc A_k$ which contains $q\in\mc Q$.
\item
	Let $(\mc Q, d)$ be a pseudo-metric space, i.e., $d$ is symmetric, fulfills the triangle inequality, and $d(q,q) = 0$ for all $q\in\mc Q$.
	Define 
	\begin{equation}
		\gamma_2(\mc Q, d) := \inf \sup_{q\in\mc Q}\sum_{k=0}^\infty 2^{\frac k2} \diam(A_k(q), d)\eqcm
	\end{equation}
	where the infimum is taken over all admissible sequences in $\mc Q$.
\end{enumerate}
\end{definition}
\subsection{Common Assumptions}
Following assumption are made for all results on rates of convergence of regression estimators in this article. They 
are conditions needed to bound the rate of convergence when estimating Fréchet means -- even without considering covariates, see \cite[Theorem 1]{schoetz19}.
\begin{assumptions}\label{ass:intro}\mbox{ }
\begin{itemize}
\item 
	\textsc{VarIneq}: There is $C_{\ms{Vlo}}\in[1,\infty)$ such that $C_{\ms{Vlo}}^{-1}\,\ol{q}{m_t}^2 \leq \Ex{d(Y_t, q)^2 - d(Y_t, m_t)^2}$ for all $q\in\mc Q$ and $t\in[0,1]$.
\item
	\textsc{Entropy}: 
	There are $C_{\ms{Ent}} \in [1,\infty)$ and $\alpha \in [1,2)$ such that 
	\begin{equation}
		\gamma_2(\mc B, d) \leq C_{\ms{Ent}}\max(\diam(\mc B, d), \diam(\mc B, d)^\alpha)
	\end{equation}
	for all $\mc B \subset \mc Q$.
\item 
	\textsc{Moment}:
	There are $\kappa > \frac{2}{2-\alpha}$ and $C_{\ms{Mom}} \in [1,\infty)$ such that $\Ex{d(Y_t, m_t)^\kappa}^{\frac1\kappa} \leq C_{\ms{Mom}}$ for all $t \in [0, 1]$.
\end{itemize}
\end{assumptions}
\begin{remark}\label{rem:lintro:assu}\mbox{ }
\begin{itemize}
\item \textsc{VarIneq}:
	This condition is also called variance inequality and is well-known in the context of Fréchet means in Alexandrov spaces, \cite{sturm03, ohta12, gouic19}.
	\textsc{VarIneq} is a condition on the noise distribution and the geometry of the metric space. It can be viewed as a quantitative version of the condition of unique Fréchet means $m_t$ of $Y_t$. The variance inequality not only ensures uniqueness of $m_t$, it also requires the objective function $\Ex{\ol{Y_t}{q}^2}$ to grow quadratically in the distance of a test point $q$ to the minimizer $m_t$. Intuitively, this is fulfilled when the noise distribution is not too similar to a distribution that has nonunique Fréchet means. 
	
	\textsc{VarIneq} is always true in Hadamard spaces \cite[Proposition 4.4]{sturm03}, which are geodesic metric spaces with nonpositive curvature and include the Euclidean spaces. For a variance inequality in spaces of nonnegative curvature, see \cite[Theorem 3.3]{ahidar20}. Furthermore, \autoref{prp:contracted} below shows an explicit construction of distributions  fulfilling \textsc{VarIneq}. We use this in section \ref{sec:simu} to construct a distribution for our simulations on the sphere.
\item \textsc{Entropy}:
	This condition can be viewed as a quantitative version of the requirement that balls in $\mc Q$ are totally bounded.

	We use Talagrand's $\gamma_2$ to formulate the entropy condition. Let $\mc B \subset \mc Q$. It holds
	\begin{equation}
		\gamma_2(\mc B, d) \leq \int_0^{\infty}  \sqrt{\log\brOf{N(\mc B, d, r)}} \dl r\eqcm
	\end{equation}
	where the integral is called \textit{entropy integral} and
	\begin{equation}
		N(\mc B, d, r) = \min\cbOf{k\in\N \,\bigg\vert\, \exists q_1,\dots,q_k\in\mc Q\colon \mc B \subset \bigcup_{j=1}^k \ball(q_j, d, r)}
	\end{equation}
	is the \textit{covering number}. Thus, we can use bounds on the entropy integral to fulfill \textsc{Entropy}, which is more common in the statistics literature. In some circumstances $\gamma_2$ is strictly lower than the entropy integral \cite[Exercise 4.3.11]{talagrand14}. One can further weaken the entropy condition as done in \cite{ahidar20} and \cite{schoetz19}, potentially at the cost of worse rates of convergence.
	
	In the Euclidean space $\R^k$, \textsc{Entropy} holds with $\alpha=1$ and $C_{\ms{Ent}} = 2\sqrt{k}$.
	If $\diam(\mc Q, d) < \infty$, one can choose $\alpha=1$ without loss of generality as the ratio between $\diam(\mc B, d)$ and $\diam(\mc B, d)^\alpha$ is bounded by the constant $\diam(\mc Q, d)^{\alpha-1}$.
	
	Next we consider an example in which $\alpha>1$ is needed. Take countably infinitely many intervals of length 1 and glue them together such that they form an infinite binary tree. This space with its intrinsic distance $d$ is an example of a metric tree and a Hadamard space \cite[Proposition 3.4]{sturm03}. A subset $\mc B$ in this space with diameter $2R$ has at most $3^{R+1}$ branches and all branches together have at most length $R3^{R+1}$. Thus, $N(\mc B, d, r) \leq c R\exp(c R)/r$ and we can calculate the bound $\gamma_2(\mc B, d) \leq c \max(R, R^{\frac32})$.		
\item \textsc{Moment}:
	This condition can be described as a moment condition. In Euclidean spaces $\mc Q = \R^k$, $d = \norm$, this is equivalent to $\Ex{\abs{Y_t - \Ex{Y_t}}^\kappa} < C_{\ms{Mom}}^\kappa$. Note that, due to the triangle inequality, $\Ex{d(Y_t, m_t)^\kappa} < \infty$ if and only if $\Ex{d(Y_t, q)^\kappa} < \infty$ for any $q\in\mc Q$ or, equivalently, for all $q \in \mc Q$. 
\end{itemize}
\end{remark}
\begin{proposition}[{\cite[section 5]{ohta12}}]\label{prp:contracted}
	Let $(\mc Q, d)$ be a proper Alexandrov space of nonnegative curvature.
	Let $Z_1$ be a random variable with values $\mc Q$ such that $\Ex{d(Z_1,q)^2} < \infty$ for all $q\in\mc Q$. Let $m \in \argmin_{q\in\mc Q}\Ex{\ol {Z_1}q^2}$ be any Fréchet mean of $Z_1$. For $a\in[0,1)$, let $Z_a := \gamma_{m \to Z}(a)$, where, for $z\in\mc Q$, $\gamma_{m \to z}$ is a geodesic with $\gamma_{m \to z}(0)=m$, $\gamma_{m \to z}(1)=z$. Then
	\begin{equation}
		(1-a)\ol qm^2 \leq \Ex{\ol{Z_a}q^2-\ol{Z_a}m^2}
	\end{equation}
	for all $a\in[0,1]$.
\end{proposition}
%
\section{Local Fréchet Regression}\label{sec:locfre}
We use the principles of Fréchet regression on local polynomial regression. This yields local polynomial Fréchet regression, \texttt{LocFre}, which was introduced (in the local constant and local linear forms) in \cite{petersen16}. 

Let $K \colon \R \to\R$ be a function, the kernel. For $\ell\in \N_0$, $h > 0$, and $x,t\in[0,1]$ define
\begin{align}
\Psi(x) &:= \br{\frac{x^k}{k!}}_{k=0,\dots,\ell}\eqcm\\
B_{n,t} &:= \frac1{nh}\sum_{i=1}^n \Psi\brOf{\frac{x_i-t}{h}}\Psi\brOf{\frac{x_i-t}{h}}\tr K\brOf{\frac{x_i-t}{h}}\eqcm\\
w_{i,t} &:= \frac1{nh}\Psi(0)\tr B_{n,t}^{-1} \Psi\brOf{\frac{x_i-t}{h}} K\brOf{\frac{x_i-t}{h}}
\eqcm
\end{align}
whenever $B_{n,t}$ is invertible. Note that $w_{i,t}$ depends on $n, (x_j)_{j=1,\dots,n}$ and in particular on $h$, which is not indicated in the notation.
A local polynomial Fréchet estimator of order $\ell$ is any element
\begin{equation}
\hat m_{t} \in \argmin_{q\in\mc Q} \sum_{i=1}^n w_{i,t} d(y_i, q)^2\eqfs
\end{equation}

For denoting a smoothness condition required for this estimator to achieve the nonparametric rate of convergence, we need to refer to the \textit{Hölder class} $\Sigma(\beta, L)$ for $\beta, L > 0$. It is defined as the set of $\lfloor \beta\rfloor$-times continuously differentiable functions $f\colon [0,1]\to\R$ with $|f^{(\lfloor \beta\rfloor)}(t)-f^{(\lfloor \beta\rfloor)}(x)| \leq L\abs{x-t}^{\beta-\lfloor \beta\rfloor}$ for all $x,t\in [0,1]$.
\begin{assumptions}\label{ass:locfre}\mbox{ }
\begin{itemize}
\item \textsc{Kernel}:
	There are $C_{\ms{Kmi}}, C_{\ms{Kma}} \in [1,\infty)$ such that 
	\begin{equation}
		C_{\ms{Kmi}}^{-1} \ind_{[-\frac12,\frac12]}(x) \leq K(x) \leq C_{\ms{Kma}} \ind_{[-1,1]}(x)
	\end{equation}
	for all $x\in\R$.
\item \textsc{HölderSmoothDensity}:
The function $[0,1] \to \mc Q,\, t\mapsto m_t$ is continuous. Let $C_{\ms{Len}} \in [1,\infty)$ such that $\sup_{s,t\in[0,1]} d(m_s, m_t) \leq C_{\ms{Len}}$. 
	Let $\mu$ be a probability measure on $\mc Q$. Let $C_{\ms{Int}} \in [1,\infty)$ such that $\int \ol{y}{m_0}^2 \mu (\dl y) \leq C_{\ms{Int}}$. 
	Let $y\to\rho(y|t)$ be the $\mu$-density of $Y_t$.
	Let $\beta>0$ with $\ell = \lfloor\beta\rfloor$. For $\mu$-almost all $y\in\mc Q$, there is $L(y)\geq0$ such that $t \mapsto  \rho(y|t) \in \Sigma(\beta, L(y))$. Furthermore, there is a constant $C_{\ms{SmD}} \in [1,\infty)$, $\int L(y)^2 \dl \mu(y) \leq C_{\ms{SmD}}^2$.
\end{itemize}
\end{assumptions}
\textsc{Kernel} and a smoothness condition are classical requirements for a local polynomial estimators to obtain an optimal error bound \cite[Proposition 1.13]{tsybakov08}.
\begin{remark}\label{rem:locfre:assu}\mbox{ }
\begin{itemize}
\item \textsc{Kernel}:
	This is a typical condition on kernels for local kernel regression, see also \cite[Lemma 1.5]{tsybakov08}. It is fulfilled, e.g., by the rectangular kernel $\ind_{[-\frac12, \frac12]}(x)$ or the Epanechnikov kernel $\frac34(1-x^2)\ind_{[-1,1]}(x)$. \textsc{Kernel} likely could be weakened to allow for a greater variety of kernels, e.g., higher order kernels.
\item \textsc{HölderSmoothDensity}:
	If the noise distribution has a $\mu$-density and this density is smooth enough, \textsc{HölderSmoothDensity} can be interpreted as a smoothness condition on $t\mapsto m_t$:
	In a Euclidean space $\mc Q = \R^k$ with a location model $\rho(y|t) = f(\normof{y-m_t}^2)$ for a smooth function $f \colon [0,\infty) \to [0,\infty)$, we have $\partial_t \rho(y|t) = -2(y-m_t)\tr\dot m_t f\pr(\normof{y-m_t}^2)$, where $\dot m_t \in \R^k$ is the derivative of $x\mapsto m_x$ at $t$. If $f\pr$ is smooth enough and bounded, the smoothness of $\partial_t \rho(y|t)$ is dominated by the smoothness of $m_t$. Informally, the density should be as least as smooth as the regression function, to view this condition as a typical smoothness assumption on the regression function. It is likely an artifact of the proof that we require the error density to be smooth. 
\end{itemize}
\end{remark}
\begin{theorem}[\texttt{LocFre} Bounded]\label{cor:locfre:bounded}
	Let $(\mc Q, d)$ be a bounded metric space. Let $\beta > 0$ with $\ell = \lfloor\beta\rfloor$. Let $\hat m_{t}$ be the local polynomial estimator of order $\ell$ with $h\geq \frac cn$ and $n\geq c$.
	Assume \textsc{VarIneq}, \textsc{Entropy} with $\alpha=1$, \textsc{HölderSmoothDensity}, \textsc{Kernel}.
	Then 
	\begin{equation}
		\Ex*{\ol{m_t}{\hat m_t}^2}  
		\leq 
		C_1 h^{2\beta} + C_2 (nh)^{-1}
		\eqcm
	\end{equation}
	where 
	\begin{align*}
		C_1 &= c C_{\ms{Vlo}}^2C_{\ms{Ker}}^2 C_{\ms{SmD}}^2 \diam(\mc Q, d)^2\eqcm\\
		C_2 &= c C_{\ms{Vlo}}^2 C_{\ms{Ent}}^2 C_{\ms{Ker}}^2 \diam(\mc Q, d)^2\eqfs
	\end{align*}
\end{theorem}
\begin{theorem}[\texttt{LocFre} Hadamard]\label{cor:locfre:hadamard}
	Let $(\mc Q, d)$ be a Hadamard space. Let $\beta > 0$ with $\ell = \lfloor\beta\rfloor$. Let $\hat m_{t}$ be the local polynomial estimator of order $\ell$  with $c \geq h\geq \frac cn$ and $n\geq c$. Assume  \textsc{Moment}, \textsc{Entropy}, \textsc{HölderSmoothDensity}, \textsc{Kernel}.
	Then, for all $t\in[0,1]$, 
	\begin{equation}
		\Ex*{\ol{m_t}{\hat m_t}^2} 
		\leq 
		C_1 h^{2\beta}+ C_2(nh)^{-1}
		\eqcm
	\end{equation}
	where 
	\begin{align*}
		C_1 &= c_{\alpha,\kappa} \br{C_{\ms{Kmi}}^2C_{\ms{Kma}}^2 C_{\ms{SmD}} C_{\ms{Mom}} C_{\ms{Len}} C_{\ms{Int}}}^{\frac{2}{2-\alpha}}
		\eqcm\\
		C_2 &= c_{\alpha,\kappa} \br{C_{\ms{Mom}} C_{\ms{Ent}}C_{\ms{Kmi}}^2C_{\ms{Kma}}^2}^{\frac{2}{2-\alpha}}
		\eqfs
	\end{align*}
\end{theorem}
The two theorems are derived from a more general result in the appendix, \autoref{thm:locfre}.
We obtain the classical error bound for local polynomial estimators with a bias term $h^{2\beta}$ and a variance term $(nh)^{-1}$. 
If we set $h = n^{-\frac{1}{2\beta+1}}$, in both cases, we obtain the classical nonparametric rate of convergence $\Ex{\ol{m_t}{\hat m_t}^2} \leq C n^{-\frac{2\beta}{2\beta+1}}$. By integrating the inequality, we obtain the same bound for the MISE $\Ex{\int_0^1 \ol{m_t}{\hat m_t}^2 \dl t}$.
\begin{remark}
	\autoref{cor:locfre:hadamard} applied to the real line $(\mc Q, d) = (\R, \norm)$ yields almost the same result as the standard result for Euclidean local polynomial regression \cite[Proposition 1.13]{tsybakov08}. Aside from different constants, we require a finite moment of order $\kappa>2$ instead of $\kappa=2$ and the error density needs to change smoothly, see point \textsc{HölderSmoothDensity} in \autoref{rem:locfre:assu}. It seems remarkable that the results are so close as we have to do without an inner product and without vector space structure in the space of responses.
\end{remark}
\section{Orthogonal Series Fréchet Regression}\label{sec:trifre}

Let $(\psi_j)_{j\in\N}$ be a sequence of functions that form an orthonormal base in $\mb L^2[0,1]$, in particular, 
\begin{equation}
\int_0^1\psi_j(x)\psi_{\tilde j}(x) \dl x = \delta_{j{\tilde j}}
\end{equation}
for all $\tilde j,j \in \N$, where $\delta_{j{\tilde j}}$ is the Kronecker delta. 
Let $N \in \N$. Define $\Psi_N := (\psi_j)_{j=1,\dots,N}$.

Assume the matrix $B_{n} := \frac1n\sum_{i=1}^n\Psi_N(x_i)\Psi_N(x_i)\tr$ is invertible. The orthogonal series Fréchet regression estimator is 
\begin{equation}
	\hat m_t \in \argmin_{q\in\mc Q} \Psi_N(t)\tr B_{n}^{-1} \frac1n\sum_{i=1}^n\Psi_N(x_i) d(y_i, q)^2
	\eqfs
\end{equation}
For an explicit estimator, we have to choose an explicit orthogonal series. Because of its appealing theoretical properties among other things, the trigonometric series is a common choice. 
Let $(\psi_j)_{j\in\N}$ be the trigonometric basis of $\mb L^2[0,1]$, i.e., for $x\in[0,1]$, $j\in\N$,
\begin{align}
	\psi_1(x) &= 1\eqcm&
	\psi_{2j}(x) &= \sqrt{2} \cos(2\pi j x)\eqcm&
	\psi_{2j+1}(x) &= \sqrt{2} \sin(2\pi j x)\eqfs
\end{align}
The trigonometric basis is orthonormal. Furthermore, 
\begin{equation}
	\frac1n \sum_{i=1}^n \psi_j(x_i)\psi_{\tilde j}(x_i) = \delta_{j{\tilde j}}
\end{equation}
for $j,\tilde j\in \cb{1, \dots, n-1}$, see \cite[Lemma 1.7]{tsybakov08}. Thus, $B_{n}$ is the identity matrix if $N < n$ and the estimator simplifies to 
\begin{equation}
	\hat m_t \in \argmin_{q\in\mc Q} \Psi_N(t)\tr \frac1n\sum_{i=1}^n\Psi_N(x_i) d(y_i, q)^2
	\eqfs
\end{equation}
The appropriate smoothness class connected to the trigonometric basis $(\psi_j)_{j\in\N}$ is the \textit{periodic Sobolev class} $W^{\ms{per}}(\beta,L)$, see \cite[Definition 1.11]{tsybakov08}. 
A function $f(x) = \sum_{j=1}^\infty \vartheta_j \psi_j(x)$ belongs to $W^{\ms{per}}(\beta,L)$ if and only if the sequence $\vartheta=(\vartheta_j)_{j\in\N}$, $\vartheta_j = \int_0^1 f(x) \psi_j(x) \dl x$, of the Fourier coefficients of $f$ belongs to the ellipsoid $\mc E(\beta, L)$, which is defined as 
\begin{equation}
\mc  E(\beta, L) = \cb{\vartheta\in\ell^2(\R) \colon \sum_{j=1}^\infty \vartheta_j^2 a_j^{-2} \leq L^2}
\eqcm
\end{equation}
where $a_{2j+1} = a_{2j} = (2j)^{-\beta}$, see \cite[Proposition 1.14]{tsybakov08}.
\begin{assumptions}\label{ass:trifre}\mbox{ }
\begin{itemize}
\item \textsc{SobolevSmoothDensity}: 
	The function $[0,1] \to \mc Q,\, t\mapsto m_t$ is continuous. Let $C_{\ms{Len}} \in [1,\infty)$ such that $\sup_{s,t\in[0,1]} d(m_s, m_t) \leq C_{\ms{Len}}$. 
	Let $\mu$ be a probability measure on $\mc Q$. Let $C_{\ms{Int}} \in [1,\infty)$ such that $\int \ol{y}{m_0}^2 \mu (\dl y) \leq C_{\ms{Int}}$. 
	For all $t\in[0,1]$, the random variable $Y_t$ has a density $y\mapsto \rho(y|t)$ with respect to $\mu$. 
	Let $\beta\geq 1$. For $\mu$-almost all $y\in\mc Y$, there is $L(y)\geq0$ such that $t \mapsto \rho(y|t) \in W^{\ms{per}}(\beta, L(y))$. Furthermore, there is $C_{\ms {SmD}}\in[1,\infty)$ such that $\int L(y)^2 \dl \mu(y) \leq C_{\ms {SmD}}^2$. 
\end{itemize}
\end{assumptions}
\begin{remark}\label{rem:trifre:assu}\mbox{ }
\begin{itemize}
\item \textsc{SobolevSmoothDensity}: This condition is analogous to \textsc{Höl\-der\-Smooth\-Den\-sity} with the Hölder smoothness class replaced by the Sobolev smoothness class. Again, this condition can be interpreted as a smoothness condition on $t\mapsto m_t$ if the error density is smooth enough, see \autoref{rem:locfre:assu}. 

The trigonometric basis functions are periodic and the smoothness condition also requires $t \mapsto m_t$ to be periodic, i.e., identifying $t=0$ and $t=1$ should yield a well-defined function which is appropriately smooth at this transition.
\end{itemize}
\end{remark}
Further conditions are discussed in \autoref{rem:lintro:assu}.
\begin{theorem}[\texttt{OrtFre} Bounded]\label{cor:trifre:bounded}
	Let $(\mc Q, d)$ be a bounded metric space. Assume \textsc{VarIneq}, \textsc{Entropy}  with $\alpha=1$, \textsc{SobolevSmoothDensity}, and $N < n$. Then
	\begin{equation}
		\Ex*{\int_0^1 \ol{m_t}{\hat m_t}^2 \dl t} 
		\leq
		C_1 \br{N^{-2\beta} + N n^{1-2\beta}} + C_2 \frac Nn
		\eqcm 
	\end{equation}
	where
	\begin{align*}
		C_1 &= c_{\beta} C_{\ms{Vlo}}^2	C_{\ms{SmD}}^2 \diam(\mc Q)^2\eqcm\\
		C_2 &= c_{\beta} C_{\ms{Vlo}}^2	C_{\ms{Ent}}^2 \diam(\mc Q)^2\eqfs
	\end{align*}
\end{theorem}
\begin{theorem}[\texttt{OrtFre} Hadamard]\label{cor:trifre:hadamard}
	Let $(\mc Q, d)$ be a Hadamard metric space. Assume \textsc{Moment}, \textsc{Entropy} with $\alpha=1$, \textsc{SobolevSmoothDensity}, and $N \leq c \sqrt{n}$. Then
	\begin{equation}
		\Ex*{\int_0^1 \ol{m_t}{\hat m_t}^2 \dl t} 
		\leq
		C_1 \log(N+1)^2 \br{N^{-2\beta} + N n^{1-2\beta}} + C_2 \frac Nn
		\eqcm 
	\end{equation}
	where 	
	\begin{align*}
		C_1 &= c_{\kappa,\beta} C_{\ms{SmD}}^2 C_{\ms{Len}}^2 C_{\ms{Mom}}^2 C_{\ms{Int}}^2\eqcm\\
		C_2 &= c_{\kappa,\beta} C_{\ms{Mom}}^2 C_{\ms{Ent}}^2\eqfs
	\end{align*}
\end{theorem}
Note that for \texttt{OrtFre}, we require $\alpha=1$ in \textsc{Entropy} also in the case of Hadamard spaces. In contrast, for \texttt{LocFre} and \texttt{LocGeo} we allow $\alpha \in [1, 2)$.

We obtain the classical error bound for trigonometric series estimators with a bias term $N^{-2\beta}$ and a variance term $\frac Nn$. The term $Nn^{1-2\beta}$ is of lower order than $\frac Nn$ for $\beta > 1$ and can be discarded for large $n$ in this case. If we set $N = n^{\frac{1}{2\beta + 1}}$, we obtain the classical nonparametric rate of convergence $\Ex{\int_0^1 \ol{m_t}{\hat m_t}^2 \dl t} \leq C n^{-\frac{2\beta}{2\beta+1}}$ with an additional $\log(n)^2$ factor in the Hadamard case. The two theorems are derived from a more general result in the appendix, \autoref{thm:trifre}. Point-wise results are not obtained here.
\begin{remark}
	\autoref{cor:trifre:hadamard} applied to the real line $(\mc Q, d) = (\R, \norm)$ with $N = n^{\frac1{2\beta+1}}$ yields the same bound as the standard result for Euclidean trigonometric series regression \cite[Theorem 1.9]{tsybakov08} up to the $\log(n)^2$ factor and constant factors. The requirements are slightly stricter: A finite moment of order $\kappa>2$ is assumed instead of $\kappa=2$ and the error density needs to change smoothly, see point \textsc{SobolevSmoothDensity} in \autoref{rem:trifre:assu} and \textsc{HölderSmoothDensity} in \autoref{rem:locfre:assu}. 
\end{remark}
%
\section{Local Geodesic Regression}\label{sec:locgeo}
We investigate an estimator, \texttt{LocGeo}, that locally fits (generalized) geodesics of the form $x\mapsto \Exp(p, xv)$: Let $h \geq \frac2n$. Let $K \colon \R \to\R$ be a function, the kernel. For $t\in[0,1]$, define the weight function $w_h(t,x) = \frac1h K(\frac{x-t}{h})$ and the normalized weights $w_{i,t} = w_h(t,x_i) (\sum_{j=1}^n w_h(t,x_j))^{-1}$. Note that $w_{i,t}$ depends on $n, (x_j)_{j=1,\dots,n}$ and in particular on $h$, which is not indicated in the notation. Let $\Theta\subset\mc Q\times \R^k$ be a set, the set of parameters of geodesics. Let $R\geq 1$ and set 
\begin{equation}
	\Theta_h := \Theta \cap \br{\mc Q \times \overline{\ball(0, \norm, Rh^{-1})}}
	\eqfs
\end{equation}
Let $\Exp \colon \Theta \to \mc Q, (p,v) \mapsto \Exp(p, v)$ be a function, the exponential map. 
Let 
\begin{equation}
(\hat p_{t,h}, \hat v_{t,h}) \in \argmin_{(p,v)\in\Theta_h} \sum_{i=1}^n w_{i,t} \,d\Big(y_i,  \Exp\big(p, (x_i-t)v\big)\Big)^2
\quad\text{and}\quad\hat m_t = \hat p_{t,h}
\eqfs
\end{equation} 
\begin{remark}
	For a geodesic $t \mapsto  \Exp(p, tv)$ defined by $(p,v) \in \Theta$, the parameter $v$ determines the speed of the geodesic. In some spaces, allowing arbitrary speeds when fitting geodesics can have adverse effects:
	
	Consider the circle $\mc Q = \mb S^1 = [0, 1)$ with its intrinsic distance $d = d_{\mb S^1}$. 
	In contrast to our model, we assume here that the $x_i$ do not form a grid, but instead are irregular in following sense: If $\sum_{i=1}^n a_i x_i \in \Z$ for $a_i \in \Z$, then $a_i = 0$ for $i = 1,\dots n$. In particular, all $x_i$ and all ratios between different $x_i$ are irrational. Let $y_i \in \mb S^1$, $i = 1,\dots n$. Then we can find a geodesic on the circle that uniformly approximates all $(x_i, y_i)_{i=1,\dots,n}$ arbitrarily well: For all $\varepsilon >0$, there is $v \in \R$ such that 
	\begin{equation}
		d_{\mb S^1}(\Exp(0, x_i v), y_i) = \abs{[x_iv] - y_i} < \varepsilon\eqcm
	\end{equation}
	$i = 1,\dots n$, where $[a] = a - \max\{k \in \mb Z \colon k \leq a\} \in \mb S^1$.
	This is a consequence of Kronecker's theorem on diophantine approximation, see \autoref{prp:kronecker} below (with $p=1$).
	
	Even though we have a regular grid, $x_i = \frac in$, in our setting, similar effects might occur if we allow $v$ to be arbitrarily large.
	This is prevented be minimizing over $\Theta_h$ instead of $\Theta$ when fitting geodesics.
\end{remark}
\begin{proposition}[Kronecker's Theorem {\cite{kronecker68}}]\label{prp:kronecker}
	Let $X\in\R^{n\times k}$ and $y \in \R^n$. Then 
	\begin{equation}
		\forall \varepsilon > 0 \colon \exists v \in \Z^k, b \in \Z^n\colon  \abs{X v - b - y} < \varepsilon\eqfs
	\end{equation} 
	if and only if $a\tr X \in \Z^k$ implies $a\tr y \in \Z$ for $a\in\Z^n$.
\end{proposition}
\begin{assumptions}\label{ass:locgeo}\mbox{ }
\begin{itemize}
\item
	\textsc{HölderSmoothEx}:
	Let $\beta > 0$.
	There is $C_{\ms{Smo}} \in [1, \infty)$ such that for all $t\in[0,1]$, there is $(p_t, v_t) \in \Theta_h$ such that $\Ex{d(Y_x, \Exp(p_t, (x-t)v_t))^2 - d(Y_x, m_x)^2} \leq C_{\ms{Smo}}^2\abs{x-t}^{2\beta}$ for all $x\in[0,1]$.
\item \textsc{ExpMap}:
	There are $ C_{\ms{Mup}},  C_{\ms{Mlo}} \in [1,\infty)$ such that
	\begin{align}
		d\brOf{\Exp(q, v), \Exp(p, u)} &\leq C_{\ms{Mup}} \brOf{d(q,p) + \normof{v-u}}\eqcm \\
		\int_{-\frac12}^{\frac12} d\brOf{\Exp(q, xv), \Exp(p, xu)}^2 \dl x &\geq C_{\ms{Mlo}}^{-2} \brOf{d(q,p)^2 + \normof{v-u}^2}
	\end{align}
	for all $(q,v), (p,u) \in \Theta$ with $\normof{u}, \normof{v} \leq R$.
\end{itemize}
\end{assumptions}
\begin{remark}\mbox{ }\label{rem:locgeo:assu}
\begin{itemize}
\item \textsc{HölderSmoothEx}:
	\textsc{HölderSmoothEx} can be understood as a Hölder-smoothness condition. But it involves not only $t\mapsto m_t$ but also the distribution of the observations similar to \textsc{HölderSmoothDensity}. 
	
	\textsc{VarIneq} implies that $d(\Exp(p_t, (x-t) v_t), m_x)^2 \leq C_{\ms{Vlo}} C_{\ms{Smo}}^2 \abs{x-t}^{2\beta}$. For $\beta \in (1, 2]$ on the real line with $\Exp(p_t, (x-t) v_t) = m_t + (x-t)\dot m_t$, this becomes the standard Hölder-condition, i.e., $t\mapsto m_t \in \Sigma(\beta, L)$ for a constant $L>0$.
	
	If a reverse variance inequality holds, i.e.,  $\Ex{d(Y_t, \Exp(p_t, (x-t) v_t))^2 - d(Y_t, m_t)^2} \leq C_{\ms{Vup}} d(q,m_t)^2$, then the Hölder-type bounds on $d(\Exp(p_t, (x-t) v_t), m_x)^2$ and on $\Ex{d(Y_t, q)^2 - d(Y_t, m_t)^2}$ are equivalent (up to constants). Such a reverse variance inequality always holds in proper Alexandrov spaces of nonnegative curvature (like the the Euclidean spaces or hyperspheres) with $C_{\ms{Vup}} = 1$, \cite[Theorem 5.2]{ohta12}. See also \cite[Theorem 8]{gouic19} for a variance equality, from which both a variance inequality and a reverse variance inequality may be deduced in certain spaces.
\item \textsc{ExpMap}: 
	This condition relates two distances on $\Theta$, which are induced by $d$ and $\Exp$, to the metric $d$ on $\mc Q$ and the Euclidean norm on $\R^k$. The theorems below are derived from a more general result in the appendix, \autoref{thm:locgeo}, which shows how this condition may be relaxed (to conditions \textsc{IntBoundsSup} and  \textsc{Lipschitz}, \autoref{ass:locgeo:app}).
	
	In Euclidean spaces, the geodesics are $t\mapsto \Exp(p, tv) = p + tv$ for $p, v \in \R^k$. Thus, 
	\begin{equation}
		d\brOf{\Exp(q, v), \Exp(p, u)} \leq  \normof{q - p} + \normof{v - u}
	\end{equation}
	and 
	\begin{align}
		\int_{-\frac12}^{\frac12} d\brOf{\Exp(q, xv), \Exp(p, xu)}^2 \dl x
		&=
		\int_{-\frac12}^{\frac12} \normof{(q-p) + x(v-u)}^2 \dl x
		\\&=
		\normof{q-p}^2 + \frac1{12} \normof{v-u}^2
		\eqcm
	\end{align}
	i.e., the condition holds with $C_{\ms{Mup}} = 1$ and $C_{\ms{Mlo}} = \sqrt{12}$.
	
	\textsc{ExpMap} or (its relaxations in the appendix, \autoref{ass:locgeo:app}) are not fulfilled for branching geodesics, i.e., if there are geodesics $\gamma_1, \gamma_2$ such that $\gamma_1(t) = \gamma_2(t)$ for $t \in [a,b]$ for $a<b$ and  $\gamma_1(t) \neq \gamma_2(t)$ for $t \in [a\pr,b\pr]$ for $a\pr<b\pr$. The reason is that in this case the integral over the distance of the geodesics on an interval can be of smaller order than the supremum of the distance of the two geodesics on the interval.
\end{itemize}
\end{remark}
Further conditions are discussed in \autoref{rem:lintro:assu}.
\begin{theorem}[\texttt{LocGeo} Bounded]\label{cor:locgeo:bounded}
	Let $(\mc Q, d)$ be a bounded metric space.
	Assume \textsc{VarIneq}, \textsc{Entropy} with $\alpha=1$, \textsc{HölderSmoothEx}, \textsc{Kernel}, \textsc{ExpMap}, and $h\geq \frac cn$.
	Then 
	\begin{equation}
		\Ex*{\ol{m_t}{\hat m_t}^2} \leq  C_1 (nh)^{-1} + C_2 h^{2\beta}
		\eqcm
	\end{equation}
	where 
	\begin{align*}
		C_1 &= c C_{\ms{Kmi}}C_{\ms{Kma}} C_{\ms{Vlo}} C_{\ms{Smo}}^2\eqcm\\
		C_2 &= c  C_{\ms{Mup}}^4C_{\ms{Mlo}}^4 C_{\ms{Kmi}}^3C_{\ms{Kma}}^3 C_{\ms{Ent}}^2 C_{\ms{Vlo}}^2 R k \diam(\mc Q, d)^2\eqfs
	\end{align*}
\end{theorem}
\begin{theorem}[\texttt{LocGeo} Hadamard]\label{cor:locgeo:hadamard}
	Let $(\mc Q, d)$ be a Hadamard space.
	Assume \textsc{Entropy}, \textsc{Moment}, \textsc{HölderSmoothEx}, \textsc{Kernel}, \textsc{ExpMap}, and $h\geq \frac cn$.
	Then 
	\begin{equation}
		\Ex*{\ol{m_t}{\hat m_t}^2} \leq  C_1 (nh)^{-1} + C_2 h^{2\beta}
		\eqcm
	\end{equation}
	where 
	\begin{align*}
		C_1 &= c_\kappa C_{\ms{Kmi}}C_{\ms{Kma}}C_{\ms{Smo}}^2\eqcm\\
		C_2 &= c_{\alpha,\kappa}  \br{
		C_{\ms{Mup}}^4C_{\ms{Mlo}}^{2+2\alpha} C_{\ms{Kmi}}^3C_{\ms{Kma}}^3 C_{\ms{Ent}}^2 C_{\ms{Mom}}^2 R k}^{\frac{2}{2-\alpha}}\eqfs
	\end{align*}
\end{theorem}
The two theorems are derived from a more general result in the appendix, \autoref{thm:locgeo}.
As for \texttt{LocFre}, we obtain the classical error bound for local linear estimators with a bias term $h^{2\beta}$ and a variance term $(nh)^{-1}$.
If we set $h = n^{-\frac{1}{2\beta+1}}$, in both cases we obtain the classical nonparametric rate of convergence $\Ex{\ol{m_t}{\hat m_t}^2} \leq C n^{-\frac{2\beta}{2\beta+1}}$. By integrating the inequality, we obtain the same bound for the MISE $\Ex{\int_0^1 \ol{m_t}{\hat m_t}^2 \dl t}$.
\begin{remark}
	\autoref{cor:locgeo:hadamard} applied to the real line $(\mc Q, d) = (\R, \norm)$ yields almost the same result as the standard result for Euclidean local liner regression \cite[Proposition 1.13]{tsybakov08}. Aside from different constants, we require a finite moment of order $\kappa>2$ instead of $\kappa=2$. Furthermore, by minimizing over $\Theta_h$ instead of $\Theta$, we assume that the derivative of $t\mapsto m_t$ is bounded by $R h^{-1}$, which is not a significant drawback as any meaningful choice of $h$ implies $h\to 0$ as $n \to \infty$. In contrast to \texttt{LocFre}, the smoothness condition is equivalent to the usual Hölder smoothness assumption, see \autoref{rem:locgeo:assu} on \textsc{HölderSmoothEx}.
\end{remark}
\begin{remark}
	 As mentioned in \autoref{rem:locgeo:assu}, \textsc{HölderSmoothEx} becomes the standard Hölder condition of local linear estimation on the real line for $\beta\in (1,2]$. To be able to utilize higher order smoothness, higher degree polynomials are required. As these are not easily available in general geodesic metric spaces, we restrict the estimator to geodesics, which can be viewed as degree one polynomials. Even though the smoothness condition for \autoref{cor:locgeo:bounded} and \autoref{cor:locgeo:hadamard} is stated with arbitrary $\beta>0$, it is suspected to be difficult to find large classes of interesting functions where \textsc{HölderSmoothEx} holds with $\beta > 2$.
\end{remark}
%
\section{Short Discussion of Orthogonal Series Geodesic Regression}\label{sec:trigeo}
After establishing results for Fréchet regression with local linear and orthogonal series approaches and for geodesic regression with a local linear approach, a natural next combination to discuss is geodesic regression with orthogonal series approach.

Let $(\psi_j)_{j\in\N}$ be a sequence of functions that form an orthonormal base in $\mb L^2[0,1]$. An \texttt{OrtGeo} estimator $\hat m_t$ based on $N\in\N$ basis functions may be defined as
\begin{align}
	(\hat p,\, \hat v_1, \dots, \hat v_N) &\in \argmin_{p \in \mc Q, v_j \in \ms T_p\mc Q} d\brOf{\Exp\brOf{p, \sum_{j=1}^N \psi_j(x_i)v_j}, y_i}^2\eqcm
	\\\hat m_t &:= \Exp\brOf{\hat p, \sum_{j=1}^N \psi_j(t)\hat v_j}
	\eqfs
\end{align}
In contrast to \texttt{LocGeo}, observations are not weighted differently for different $t$. Thus, the estimated parameters $(\hat p, \hat v_1, \dots, \hat v_N)$ do not depend on $t$. Where the \texttt{LocGeo} estimator is $\Exp(\hat p_t, 0)$ and ignores the direction $\hat v_t$, \texttt{OrtGeo} uses the estimated directions $\hat v_1, \dots, \hat v_N$ to encode the time-dependence of the estimated curve.

For orthogonal series estimators, one usually bounds the mean integrated squared error (MISE), as this makes it possible to utilize the orthogonality property of $(\psi_j)_{j\in\N}$. The orthogonality allows to use the $N+1$ estimated parameters in an optimal way so that for a suitable choice of $N$ depending on $n$ the best possible rate of convergence can be achieved. In the metric space setting, geodesics may not be orthogonal in an $\mb L^2$-sense: For $p,u,v\in \R^k$, $\tilde j \neq j$, we have
\begin{equation}
\int_0^1\normOf{\br{p + \psi_j(x) u} - \br{p + \psi_{\tilde j}(x)v}}^2 \dl x = \normof{u}^2 + \normof{v}^2
\eqcm
\end{equation}
but for $p$ in a general metric space, the analogous equality with a left-hand side 
\begin{equation}
\int_0^1 d\brOf{\Exp\brOf{p, \psi_j(x)u}, \Exp\brOf{p, \psi_{\tilde j}(x) v}}^2 \dl x
\end{equation}
might not be true.

We were not able to show a theorem similar to the results in the previous sections. Of course, this does not mean that the estimator above will necessarily perform badly. 

The estimator was implemented for simulations (section \ref{sec:simu}). This revealed another drawback: High-dimensional non-convex optimization is required so that \texttt{OrtGeo} is -- by far -- the slowest of all tested methods. But in some settings the estimator performs quite well, making it or modifications of it appealing for further investigations. In other settings, the performance is much worse than for the other estimators. It is not clear, whether this is due to theoretical disadvantages or a worse outcome of the general purpose optimizer used for finding $(\hat p, \hat v_1, \dots, \hat v_N)$.
%
\section{Hypersphere}
To illustrate our results for the estimators \texttt{LocFre}, \texttt{OrtFre}, and \texttt{LocGeo}, we apply them to the hyperspheres.

Let $k \in \N$. Let $\mb S^k = \{x\in \R^{k+1} \colon \abs{x}=1\}$ be the hypersphere with radius 1 as a subset of $\R^{k+1}$. We equip $\mb S^k$ with its intrinsic metric $d(q, p) = \arccos(q\tr p)$. Let $\ms T\mb S^k = \bigcup_{q\in\mb S^k} (\{q\} \times \ms T_{q} \mb S^k)$ be the tangent bundle, where $\ms T_{q} \mb S^k = \{v \in \mb R^{k+1} \ |\ q\tr v = 0 \}$ is the tangent space at $q\in\mb S^k$. The exponential map is $\Exp \colon \ms T\mb S^k \to \mb S^k,\, (q,v) \mapsto \cos(\abs{v})q + \sin(\abs v)\frac{v}{\abs v}$. Geodesics can be represented by a tuple $(p,v) \in \ms T\mb S^k$ as $x \mapsto \Exp(p, xv)$. 

For $t \in [0, 1]$, let $Y_t$ be a $\mb S^k$-valued random variable. Let the regression function $m \colon [0,1] \to \mb S^k$ be a minimizer $m_t \in \argmin_{q\in \mb S^k} \Ex{d(Y_t, q)^2}$. 
Let $x_i = \frac in$ and let $(y_i)_{i=1,\dots, n}$ be independent random variables with values in $\mb S^k$ such that $y_i$ has the same distribution as $Y_{x_i}$. 

In the following corollaries, we will always assume \textsc{VarIneq}: There is $C_{\ms{Vlo}}\in[1,\infty)$ such that $C_{\ms{Vlo}}^{-1} \,\ol q{m_t}^2 \leq \Ex{\ol{Y_t}q^2 -\ol{Y_t}{m_t}^2}$ for all $q\in\mb S^k$ and $t\in[0,1]$. This condition implies that $m_t$ is the unique minimizer of $\Ex{d(Y_t, q)^2}$. The hypersphere is a proper Alexandrov space of nonnegative curvature. Thus, \autoref{prp:contracted} shows that large classes of distributions fulfill this property. 

To fulfill the \textsc{Kernel} conditions for the local estimators, we here use the Epanechnikov kernel $x \mapsto \frac34(1-x^2) \ind_{[-1,1]}(x)$, i.e., we can set $C_{\ms{Kmi}} = \frac{16}9$ and $C_{\ms{Kma}} = 1$.

Each estimator requires a different smoothness condition as stated below. To state those, let $\mu$ be a the measure of the uniform distribution on $\mb S^k$.
\begin{corollary}[\texttt{LocFre} Hypersphere]\label{cor:locfre:sphere}
	Let $\beta> 0$ and $C_{\ms{SmD}} \geq 1$.		
	Assume \textsc{VarIneq} and use the Epanechnikov kernel.
	Choose $h = n^{-\frac{1}{2\beta+1}}$.
	Then the \texttt{LocFre} estimator $\hat m_t$ of order $\ell=\lfloor \beta \rfloor$ achieves
	\begin{equation}
		\limsup_{n\to\infty} \sup_{(P^{Y_t})_{t\in[0,1]}} n^{\frac{2\beta}{2\beta+1}} \Ex*{\int_0^1 \ol{m_t}{\hat m_t}^2 \dl t} 
		\leq 
		C 
		\eqcm
	\end{equation}
	where $C =  c C_{\ms{Vlo}}^2 C_{\ms{SmD}}^2 k$ and the supremum is taken over all distributions $(P^{Y_t})_{t\in[0,1]}$ of each $Y_t$ such that the following smoothness condition is fulfilled: $P^{Y_t}$ has a $\mu$-density $y\mapsto \rho(y|t)$ and for $\mu$-almost all $y\in\mb S^k$, $t \mapsto  \rho(y|t) \in \Sigma(\beta, C_{\ms{SmD}})$.
\end{corollary}
\begin{corollary}[\texttt{OrtFre} Hypersphere]\label{cor:trifre:sphere}
	Let $\beta> 0$ and $C_{\ms{SmD}} \geq 1$.		
	Assume \textsc{VarIneq}. Choose $N = \lfloor n^{\frac{1}{2\beta+1}} \rfloor$. 
	Then the \texttt{OrtFre} estimator $\hat m_t$ achieves
	\begin{equation}
		\limsup_{n\to\infty} \sup_{(P^{Y_t})_{t\in[0,1]}} n^{\frac{2\beta}{2\beta+1}}  \Ex*{\int_0^1 \ol{m_t}{\hat m_t}^2 \dl t} 
		\leq
		C
		\eqcm
	\end{equation}
	where $C = c_{\beta} C_{\ms{Vlo}}^2 C_{\ms{SmD}}^2 k$ and the supremum is taken over all distributions $(P^{Y_t})_{t\in[0,1]}$ of each $Y_t$ such that the following smoothness condition is fulfilled: $P^{Y_t}$ has a $\mu$-density $y\mapsto \rho(y|t)$ and for $\mu$-almost all $y\in\mb S^k$, $t \mapsto \rho(y|t) \in W^{\ms{per}}(\beta, C_{\ms {SmD}})$.
\end{corollary}
\begin{corollary}[\texttt{LocGeo} Hypersphere]\label{cor:locgeo:sphere}
	Let $\beta > 0$ and $C_{\ms{Smo}} \in [1, \infty)$.
	Assume \textsc{VarIneq} and use the Epanechnikov kernel.
	Choose $h = n^{-\frac{1}{2\beta+1}}$.
	Let $\Theta = \ms T\mb S^k$ and set $\Theta_h = \cb{(p,v) \in \Theta \colon |v| \leq h^{-1}}$.
	Then the \texttt{LocGeo} estimator $\hat m_t$ achieves
	\begin{equation}
		\limsup_{n\to\infty} \sup_{(P^{Y_t})_{t\in[0,1]}} n^{\frac{2\beta}{2\beta+1}} \Ex*{\int_0^1 \ol{m_t}{\hat m_t}^2 \dl t} 
		\leq 
		C
		\eqcm
	\end{equation}
	where $C = c C_{\ms{Smo}}^2 C_{\ms{Vlo}}^2 k^2$ and the supremum is taken over all distributions $(P^{Y_t})_{t\in[0,1]}$ of each $Y_t$ such that the following smoothness condition is fulfilled: For all $x,t\in[0, 1]$, $d(m_x, \Exp(m_t, (x-t)\dot m_t)) \leq C_{\ms{Smo}}\abs{x-t}^\beta$, where $\dot m_t \in \ms T_{m_t} \mb S^k$ is the derivative of $m_t$.
\end{corollary}
%
\section{Simulation}\label{sec:simu}
There is a total of 4 methods discussed in this article: \texttt{LocGeo}, \texttt{LocFre}, \texttt{OrtGeo}, \texttt{OrtFre}. For the latter two, we only consider the trigonometric basis. To illustrate and compare these methods on the sphere, the R-package \texttt{spheregr} was developed. All code used for this paper, including all scripts which create the plots and run and evaluate the experiments shown in this section, are freely available at \url{https://github.com/ChristofSch/spheregr}.

Each method requires numerical optimization. We use R's general purpose optimizers \texttt{stats::optim(method = "L-BFGS-B")} and \texttt{stats::optimize()}, both without explicit implementation of derivatives, but with several starting points. The implementations could potentially be improved by using the algorithm presented in \cite{eichfelder19}. For alternative implementation of geodesic regression, see \cite{shin20}.

The Fréchet methods are faster than geodesic methods, as the optimization problem for geodesics is of higher dimension. We use \textit{leave-one-out cross-validation} (LOOCV) to estimate the hyperparameters ($h$ for \texttt{LocGeo} and \texttt{LocFre}, $N$ for \texttt{OrtFre}). For \texttt{OrtGeo} it did not seem feasible to do many repetitions of the experiments with cross-validation in each run. Instead we set $N=3$ for this method to be able to calculate a result. In doing so, we effectively reduce the method to a parametric estimator. See \autoref{table:opi} for a summary of the optimization dimensions and frequencies used in the simulation. For \texttt{LocGeo} and \texttt{LocFre}, we use the Epanechnikov kernel.

\begin{table}
\begin{center}
\begin{tabular}{c|c|c|c|c}
 & \texttt{LocFre} & \texttt{OrtFre} & \texttt{LocGeo} & \texttt{OrtGeo} ($N = 3$) \\
\hline
space to optimize in & $\mb S^2$ & $\mb S^2$ & $\mb S^2 \times \ms T \mb S^2$  & $\mb S^2 \times (\ms T \mb S^2)^3$  \\
\hline
dimension & 2 & 2 & 4 & 8 \\
\hline
frequency & $\forall t$ & $\forall t$ & once & once \\
\hline
repetitions for LOOCV & $n$ & $n$ & $n$ & 1 
\end{tabular}
\end{center}
\caption{Properties of the optimizations executed in the implementation of the four regression methods to evaluate $\hat m_t$.}\label{table:opi}
\end{table}
\subsection{Contracted Uniform Distribution}
For the distribution of $Y_t$, we choose the contracted uniform distribution $\ms{CntrUnif}(m_t, a)$ with $a\in(0,1)$, which we define next. The contracted uniform distribution is obtained from the uniform distribution on the sphere by moving all points towards a center point along the connecting geodesic by a given fraction of the total distance. 

Let $\mb S^2 = \{x\in \R^{3} \colon \abs{x}=1\}$ be the sphere with radius 1 and intrinsic metric $d(q, p) = \arccos(q\tr p)$. We may describe points $q \in \mb S^2$ via two angles $(\vartheta_q, \varphi_q) \in [0,\pi]\times[0,2\pi)$ such that $q = (\sin(\vartheta_q)\cos(\varphi_q), \sin(\vartheta_q)\sin(\varphi_q), \cos(\vartheta_q))$.

\begin{definition}\label{def:contrunif}
	Let $a\in[0,1]$.
	Let $(\Theta, \Phi)$ be random angles with values in $[0,\pi]\times [0,2\pi)$ that form a uniform distribution on the sphere, i.e., they are independent, $\Theta$ has Lebesgue density $\frac12\sin(x)\ind_{[0,\pi]}(x)$, and $\Phi$ is uniformly distributed on $[0, 2\pi)$.
	Let 
	\begin{equation}
		Z_a = 
		\begin{pmatrix}
			\sin(a\Theta)\cos(\Phi)\\
			\sin(a\Theta)\sin(\Phi)\\
			\cos(a\Theta)
		\end{pmatrix}
		\eqfs
	\end{equation}
	Let $m \in \mb S^2$. Let $R_m \in O(3) \subset \R^{3\times3}$ be any orthogonal matrix that fulfills $m = R_m e_3$, where $e_3\tr = (0\ 0 \ 1)$. Then the \emph{contracted uniform distribution} $\ms{CntrUnif}(m, a)$ at $m$ with contraction parameter $a$ is defined as the distribution of $R_m Z_a$.
\end{definition}
The matrix $R_m$ in the definition above is not unique, but the symmetry of the distribution of $Z_a$ ensures that the contracted uniform distribution is well-defined.

Two important properties are implied by \autoref{prp:contracted}: For $a \in[0,1)$, $m\in \mb S^2$ is the unique Fréchet mean of $\ms{CntrUnif}(m, a)$. Furthermore, \textsc{VarIneq} is fulfilled with $C_{\ms{Vlo}} = (1-a)^{-1}$.

Lastly, we calculate the variance of the contracted uniform distribution.
Let $m \in \mb S^2$, $a\in[0,1]$, and $Y \sim \ms{CntrUnif}(m, a)$.
Let $Z_a$ and $\Theta$ as in \autoref{def:contrunif}. Then $\Ex{d(Y, m)^2} = \Ex{d(Z_a, e_3)^2}$ because of symmetry. The distance does only depend on $\Theta$ and is equal to $a\Theta$. Thus, $\Ex{d(Y, m)^2} = \Ex{(a\Theta)^2} = \frac12 a^2 \int_0^\pi x^2 \sin(x) \dl x = \frac12(\pi^2-4)a^2$.
\subsection{Setup and Illustration}
Let $t\mapsto m_t$ be one of the two curves named \textit{simple} and \textit{spiral}, which are described below. We set $x_i = \frac{i-1}{n-1}$ and sample independent $y_i \sim \ms{CntrUnif}(m_{x_i}, a)$ to obtain our data $(x_i, y_i)_{i=1,\dots,n}$. The parameter $a\in[0,1]$ is chosen so that the distribution has a given standard deviation $\mathtt{sd}$. Then we calculate the four different nonparametric regression estimators \texttt{LocFre}, \texttt{OrtFre}, \texttt{LocGeo}, and \texttt{OrtGeo}.

We first show some illustrating plots \autoref{fig:nonparam:closed} and \autoref{fig:nonparam:open}. In these, we want to depict functions of the form $[0,1] \to [0,\pi]\times[0,2\pi),\, t\mapsto (\vartheta_{m_t}, \varphi_{m_t})$. The graph of such a function is 3-dimensional and hard to understand on 2D-paper. Creating two plots, one for $[0,1] \to [0,\pi],\, t\mapsto \vartheta_{m_t}$ and another for $[0,1] \to [0,2\pi),\, t\mapsto \varphi_{m_t}$, is also difficult to interpret, as one has to always take both graphs into account at the same time. Instead we show the image of the functions $\{(\vartheta_{m_t}, \varphi_{m_t}) \colon t\in[0,1]\} \subset [0,\pi]\times[0,2\pi)$.

The rectangle of the two angles $(\vartheta, \varphi) \in [0,\pi]\times[0,2\pi)$ parameterizing the sphere is the \textit{Mercator projection}. This projection (as any projection of the sphere to the euclidean plane) distorts the surface of the sphere. This is made visible by the thin gray lines in the plots, which are geodesics and replace the usual grid lines. The plots show the image of $t\mapsto m_t$ (black line) and the different estimators $t\mapsto \hat m_t$ (colored lines). The covariate $t$ is not shown directly. But the positions $t = 0.25, 0.5, 0.75$ are marked on each curve by a square, a rhombus, and a triangle, respectively. Note that distances are distorted: Distances close to the equator ($\vartheta=\frac12\pi$) are larger than they appear and smaller at the poles ($\vartheta\in\{0, \pi\}$). The observations $y_i$ (black dots in the top plots) are connected via thin black lines to $m_{x_i}$.

We test two different regression functions $t\mapsto m_t$. The first one, named \textit{simple} has angles $t\mapsto (\frac14\pi, \frac12 + 2\pi t)$, see \autoref{fig:nonparam:closed}. This seems to be a straight line in the Mercator projection but is a curved function on the sphere and cannot be approximated well by a single geodesic. This \textit{simple} curve is periodic.
Moreover, it can be written as $t \mapsto \Exp(p, \sin(2\pi t) v_1 + \cos(2\pi t) v_2)$ with the appropriated choices of $p \in \mb S^2, v_1, v_2 \in \ms T_p \mb S^2$. Thus, this curve lies in the model space of \texttt{OrtGeo} if $N\geq2$. Recall that we fixed $N=3$.
The second curve is described by $t\mapsto (\frac18\pi + \frac34\pi t, \frac12 + 3\pi t)$. Again this curve is not geodesic. It \textit{spirals} around the sphere, see \autoref{fig:nonparam:open}, and is not periodic.
To estimate nonperiodic functions with \texttt{OrtGeo} and \texttt{OrtFre}, which require periodicity, we copy the data and append it in reverse order to estimate the periodic function 
\begin{align}
t\mapsto 
\begin{cases}
	m_{2t}& \text{if}\  t <\frac12\eqcm\\
	m_{2-2t}& \text{if}\  t \geq\frac12
	\eqfs
\end{cases}
\end{align}
This may lead to boundary effects.

\begin{figure}
\includegraphics[width=0.49\textwidth]{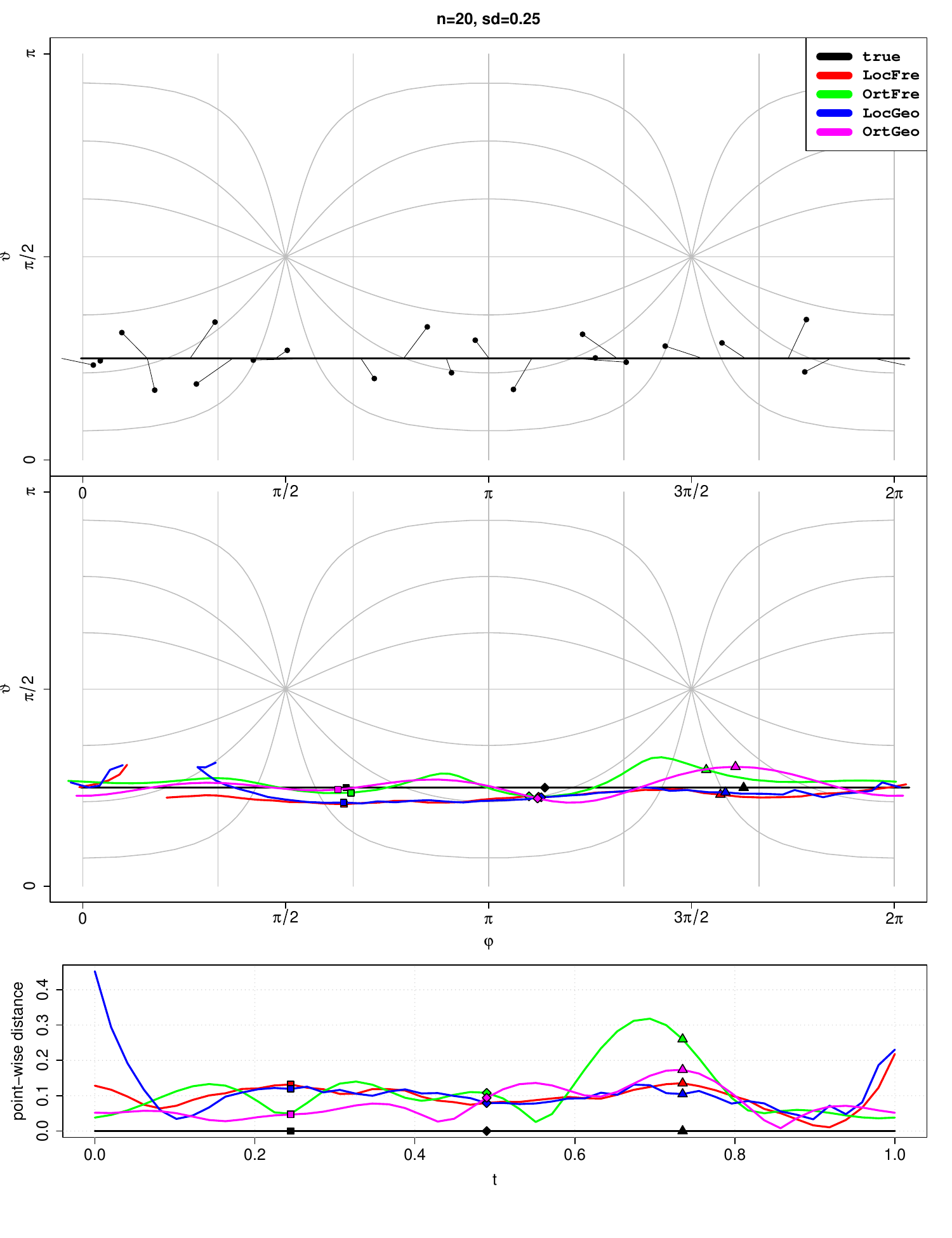}
\includegraphics[width=0.49\textwidth]{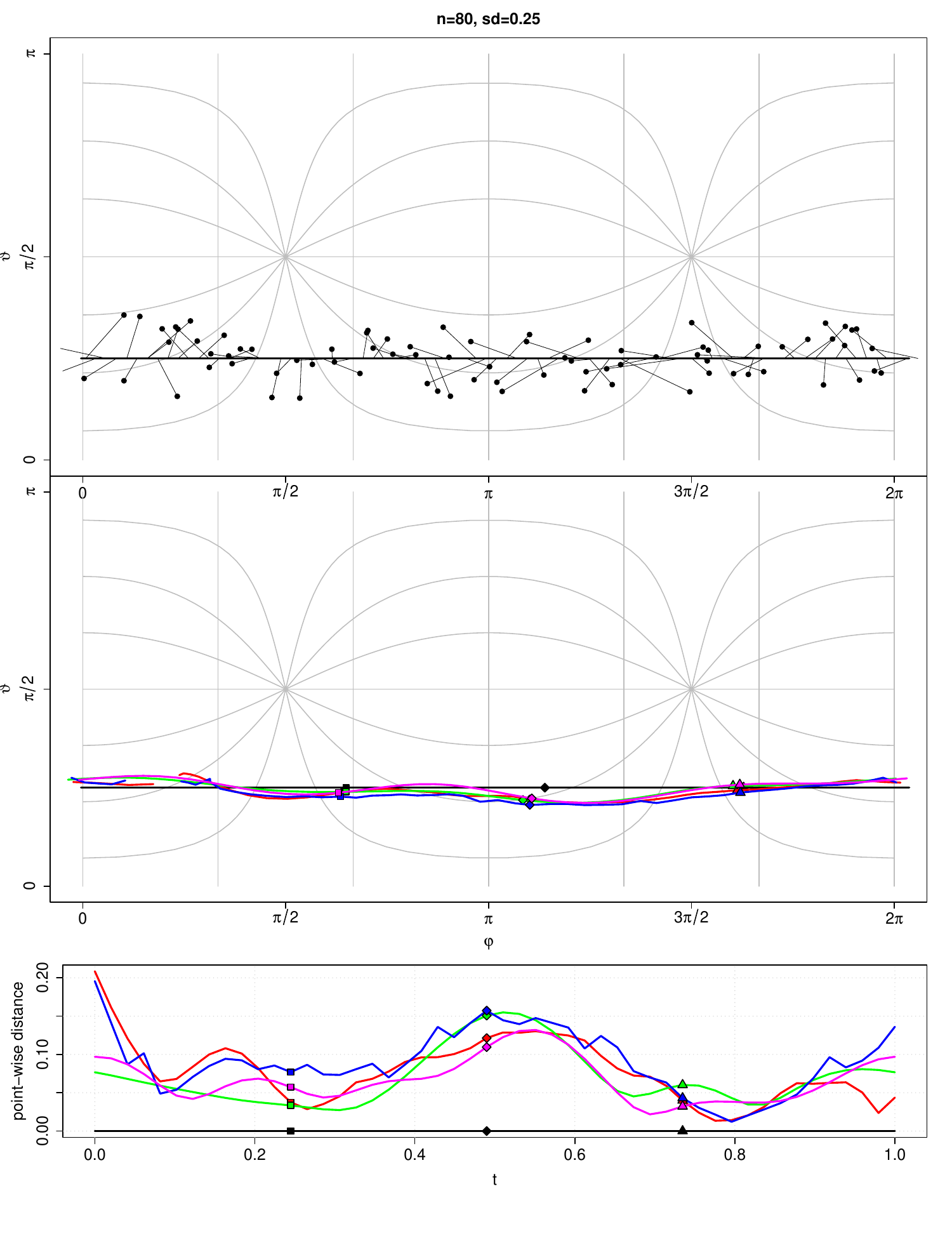}
\includegraphics[width=0.49\textwidth]{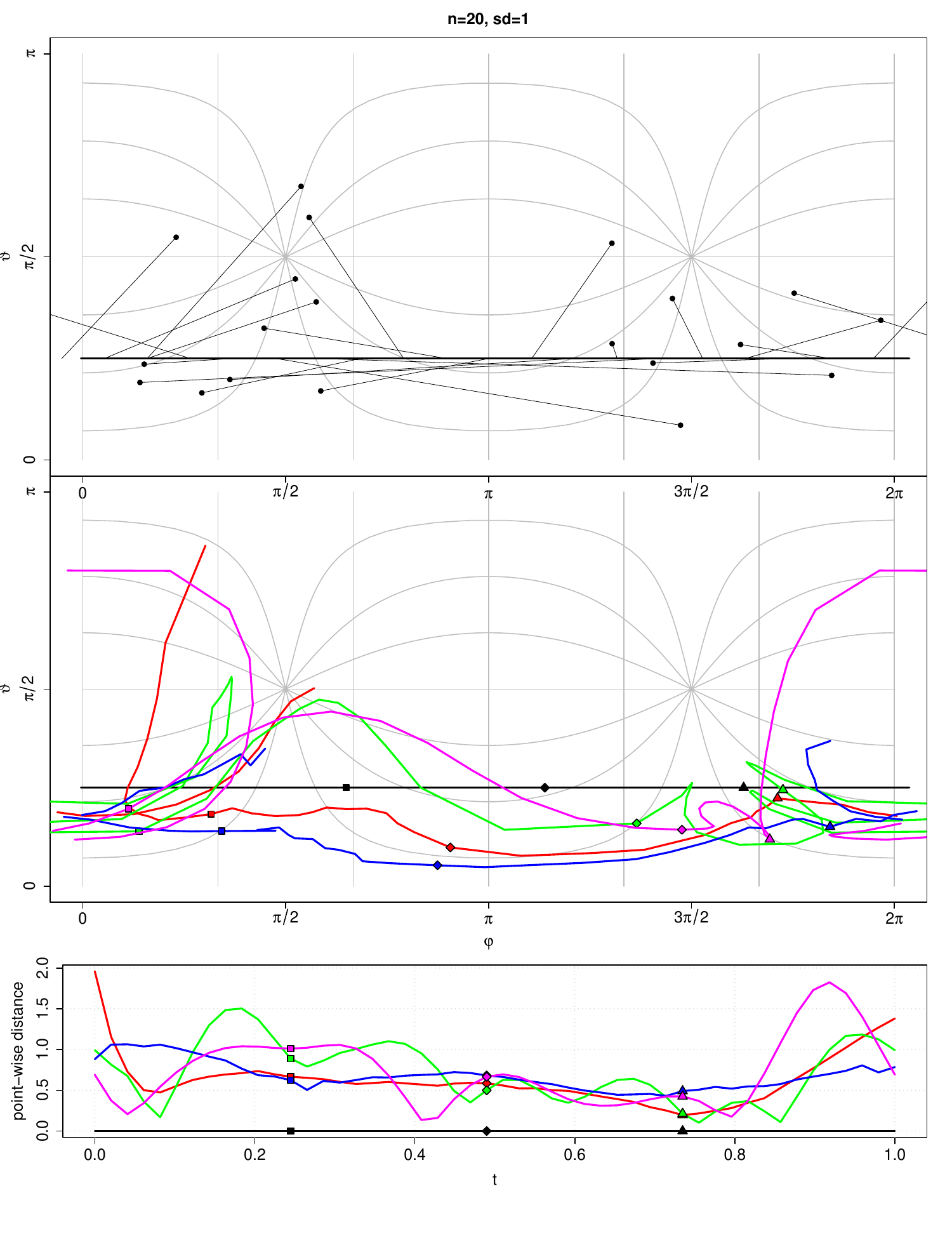}
\includegraphics[width=0.49\textwidth]{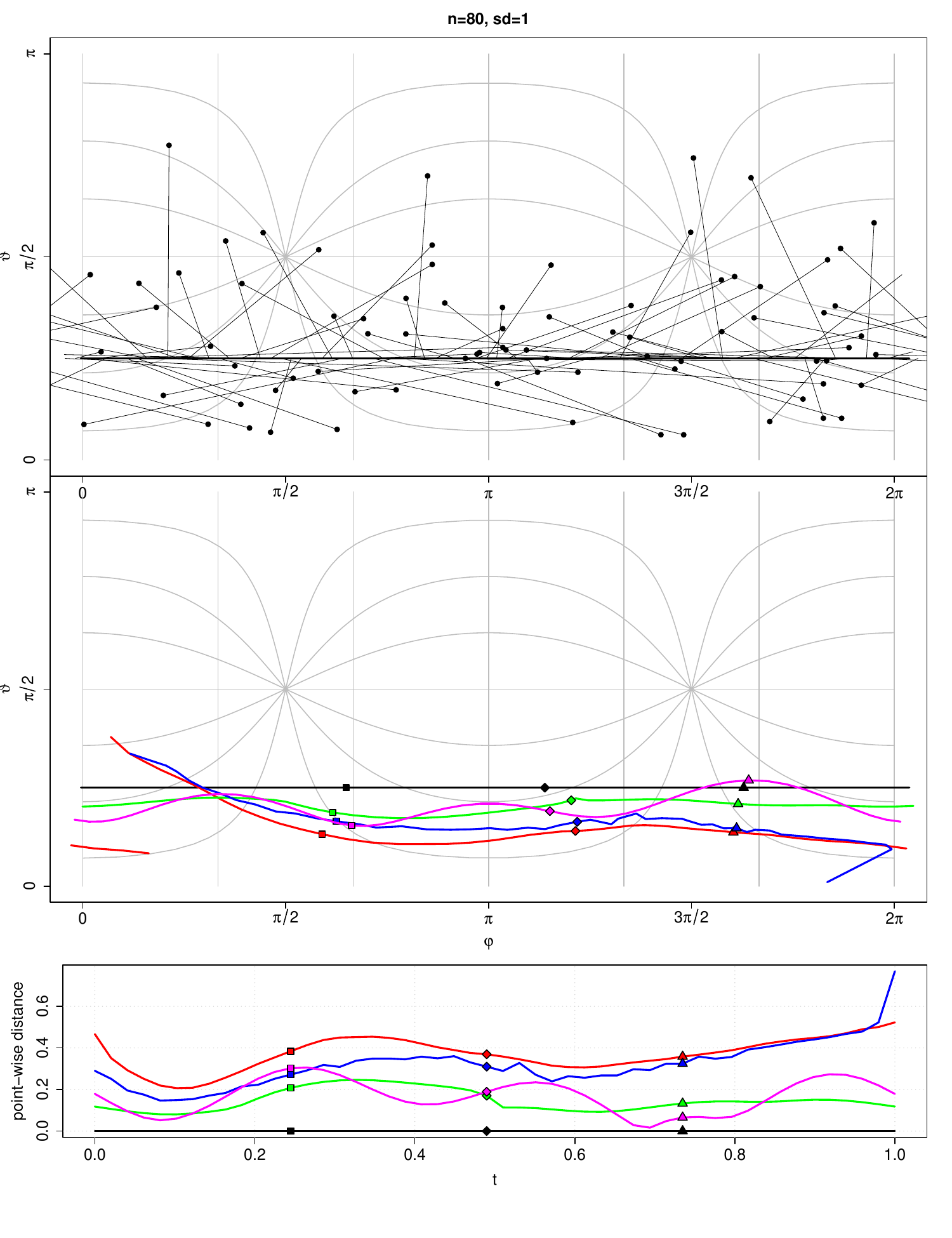}
\caption{For the \textit{simple} curve, we sample $n \in \{20, 80\}$ observations with contracted uniform noise of standard deviation $\mathtt{sd} \in \{\frac14, 1\}$ (top plot of each quadrant). Then we apply \texttt{LocGeo}, \texttt{LocFre}, \texttt{OrtGeo}, \texttt{OrtFre} (middle part of each quadrant). The distance of the estimated curve to the true one at each point in time is shown in the plots at the bottom of each quadrant.}
\label{fig:nonparam:closed}
\end{figure}
\begin{figure}
\includegraphics[width=0.49\textwidth]{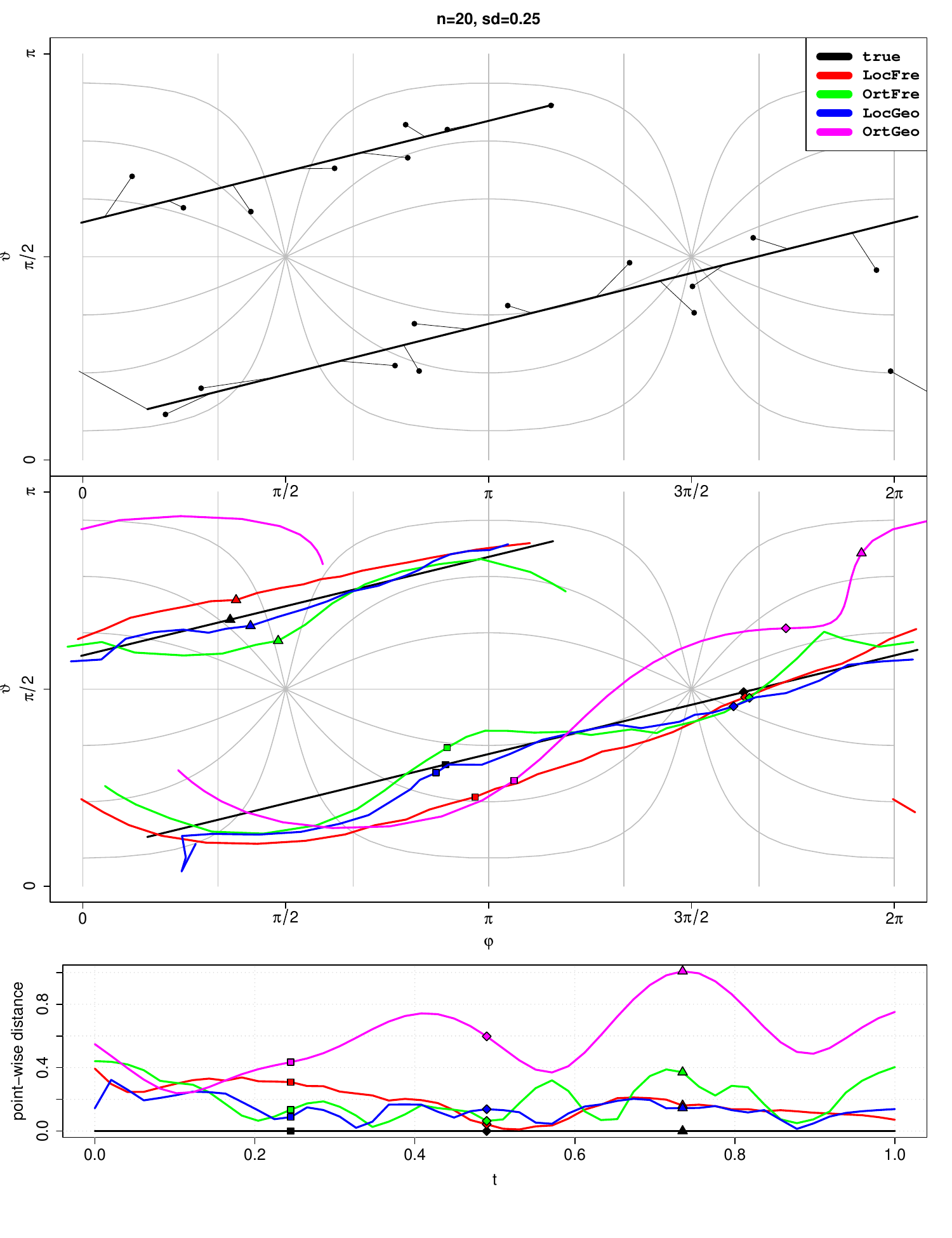}
\includegraphics[width=0.49\textwidth]{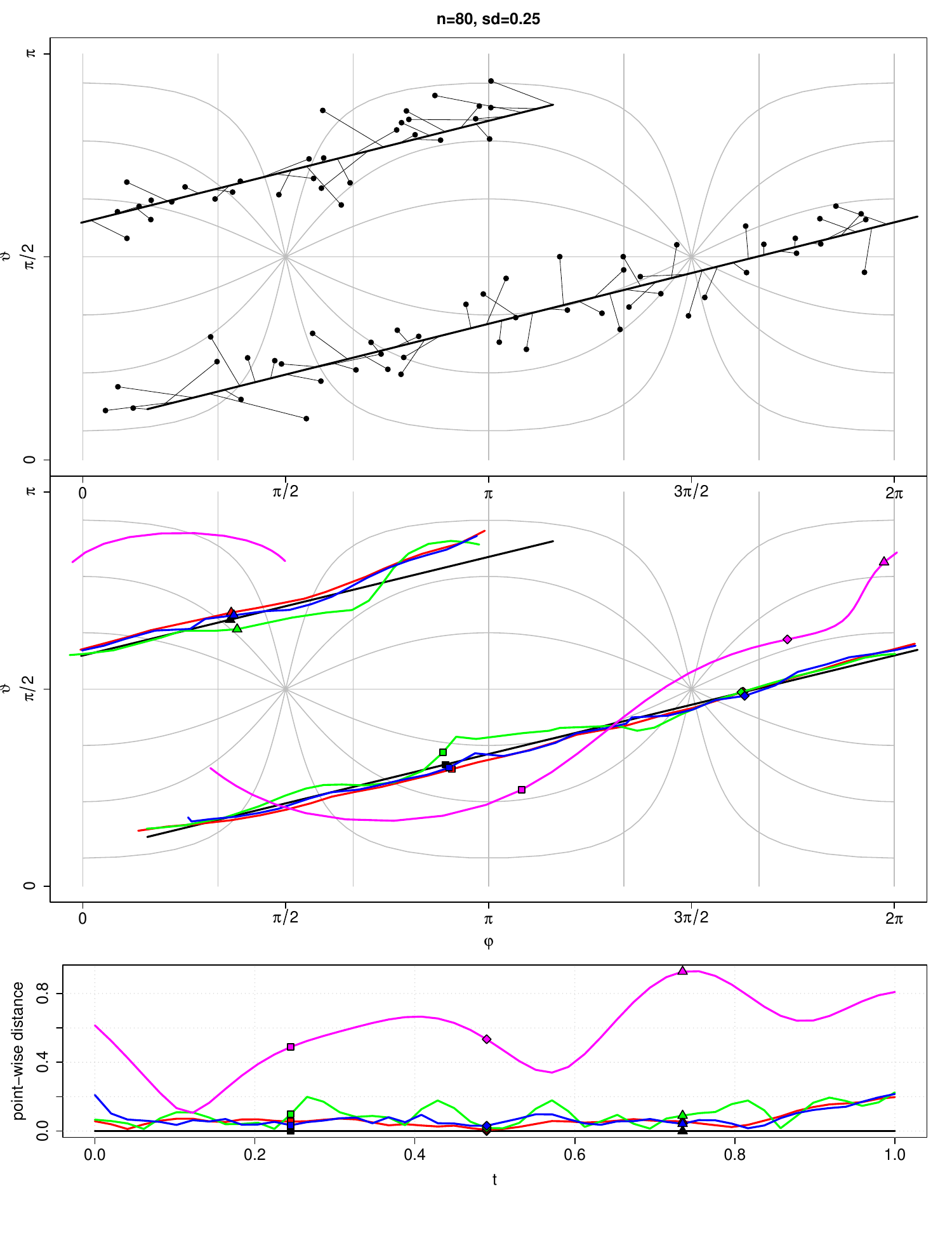}
\includegraphics[width=0.49\textwidth]{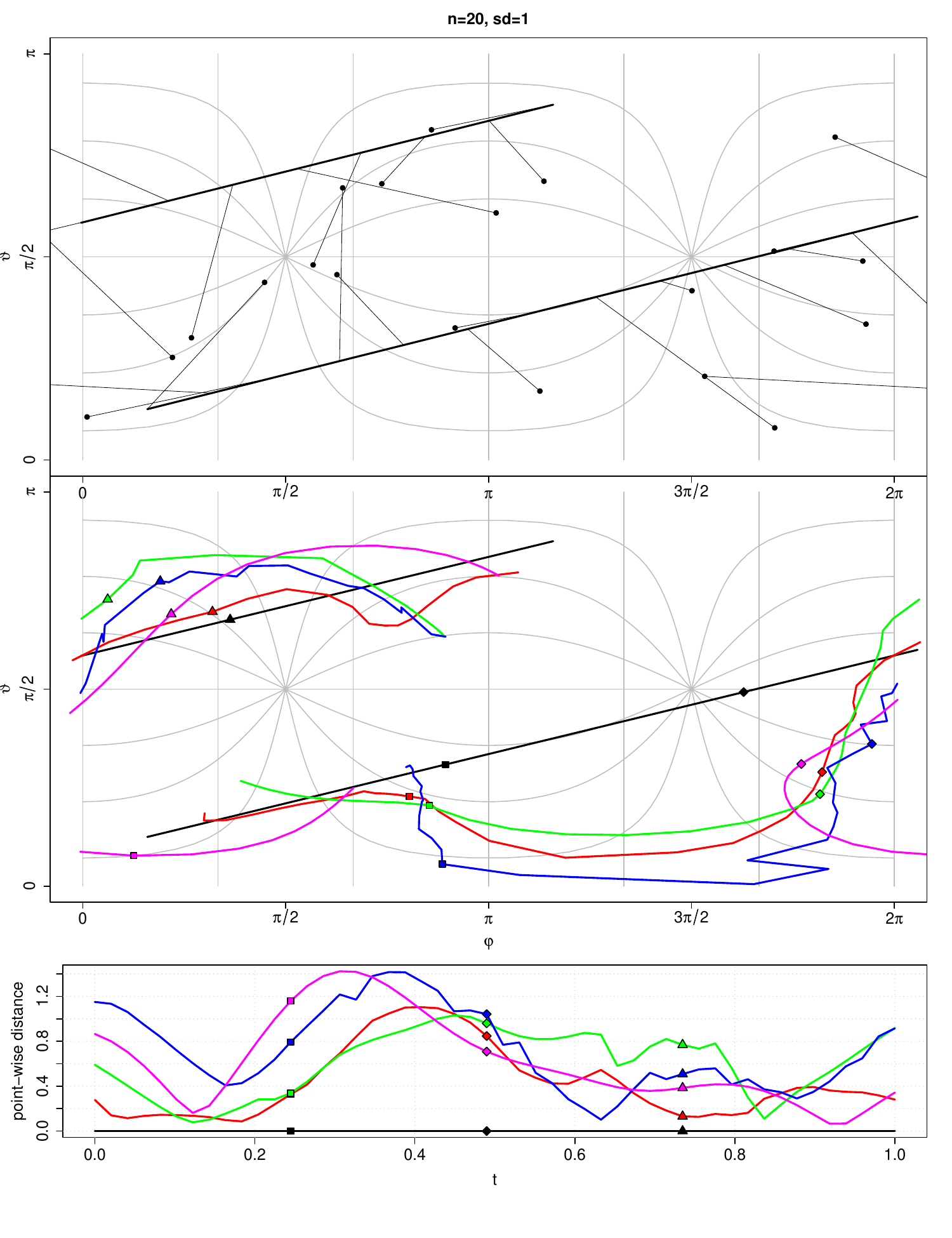}
\includegraphics[width=0.49\textwidth]{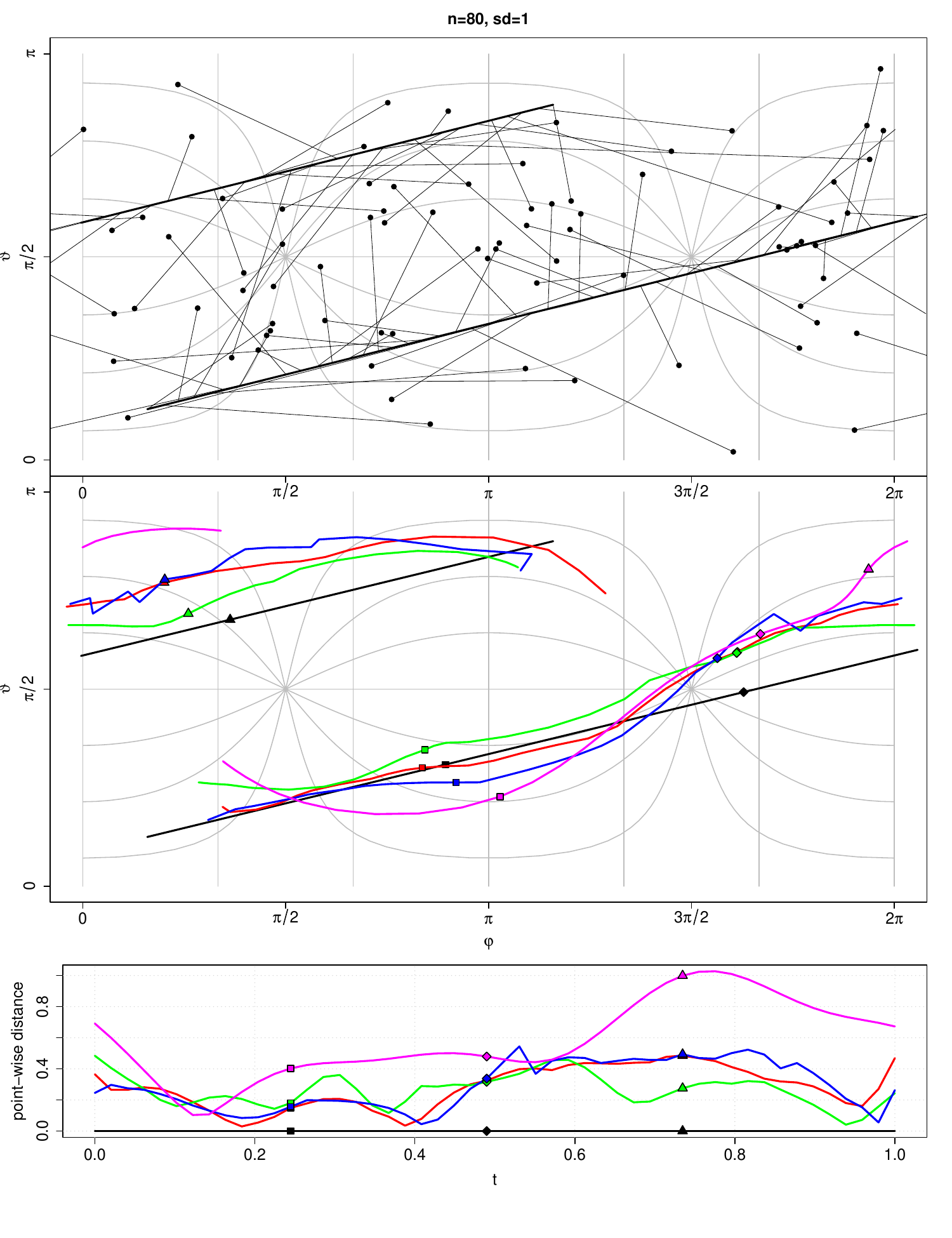}
\caption{For the \textit{spiral}, we sample $n \in \{20, 80\}$ observations with contracted uniform noise of standard deviation $\mathtt{sd} \in \{\frac14, 1\}$. Then we apply \texttt{LocGeo}, \texttt{LocFre}, \texttt{OrtGeo}, \texttt{OrtFre}. The distance of the estimated curve to the true one at each point in time is shown in the plots at the bottom of each quadrant.}
\label{fig:nonparam:open}
\end{figure}

Roughly speaking and judging only from \autoref{fig:nonparam:closed} and \autoref{fig:nonparam:open}, all estimators seem to perform similarly, except for a worse outcome for \texttt{OrtGeo} on the \textit{spiral}. In the setting $(n=20, \mathtt{sd} = 1)$ the estimators are not able to come close to the true curves. 
\subsection{Results}
We approximate the MISE values in different settings with the \textit{simple} and the \textit{spiral} curve. To this end, the simulations are repeated 500 times and the integrated squared errors of these repetitions are averaged. The results are presented in \autoref{tbl:mise:nonparam}.
\begin{table}[h!]
\centering
\begin{tabular}[t]{r|r|l|>{}r|>{}r|>{}r|>{}r}
\hline
\multicolumn{3}{c|}{\textbf{Setting}} & \multicolumn{4}{c}{\textbf{MISE}} \\
\cline{1-3} \cline{4-7}
$n$ & $\mathtt{sd}$ & \textbf{curve} & \texttt{LocFre} & \texttt{OrtFre} & \texttt{LocGeo} & \texttt{OrtGeo}\\
\hline
20 & 0.25 & simple & \cellcolor[HTML]{ABFF80}{0.02070} & \cellcolor[HTML]{BBFF80}{0.02410} & \cellcolor[HTML]{C5FF80}{0.02595} & \cellcolor[HTML]{80FF80}{0.01397}\\
\hline
80 & 0.25 & simple & \cellcolor[HTML]{CCFF80}{0.00731} & \cellcolor[HTML]{C2FF80}{0.00662} & \cellcolor[HTML]{DEFF80}{0.00851} & \cellcolor[HTML]{80FF80}{0.00361}\\
\hline
20 & 1.00 & simple & \cellcolor[HTML]{80FF80}{0.34890} & \cellcolor[HTML]{8CFF80}{0.39052} & \cellcolor[HTML]{85FF80}{0.36356} & \cellcolor[HTML]{E4FF80}{0.86604}\\
\hline
80 & 1.00 & simple & \cellcolor[HTML]{9CFF80}{0.12056} & \cellcolor[HTML]{82FF80}{0.09350} & \cellcolor[HTML]{94FF80}{0.11026} & \cellcolor[HTML]{80FF80}{0.09228}\\
\hline
20 & 0.25 & spiral & \cellcolor[HTML]{80FF80}{0.02899} & \cellcolor[HTML]{CFFF80}{0.05902} & \cellcolor[HTML]{8CFF80}{0.03268} & \cellcolor[HTML]{FF8080}{0.38623}\\
\hline
80 & 0.25 & spiral & \cellcolor[HTML]{80FF80}{0.00900} & \cellcolor[HTML]{BBFF80}{0.01534} & \cellcolor[HTML]{8CFF80}{0.01008} & \cellcolor[HTML]{FF8080}{0.37191}\\
\hline
20 & 1.00 & spiral & \cellcolor[HTML]{87FF80}{0.56768} & \cellcolor[HTML]{80FF80}{0.52354} & \cellcolor[HTML]{85FF80}{0.54786} & \cellcolor[HTML]{BDFF80}{0.91824}\\
\hline
80 & 1.00 & spiral & \cellcolor[HTML]{85FF80}{0.15185} & \cellcolor[HTML]{80FF80}{0.14662} & \cellcolor[HTML]{80FF80}{0.14677} & \cellcolor[HTML]{FFFD80}{0.47189}\\
\hline
\end{tabular}
\includegraphics[width=0.99\textwidth]{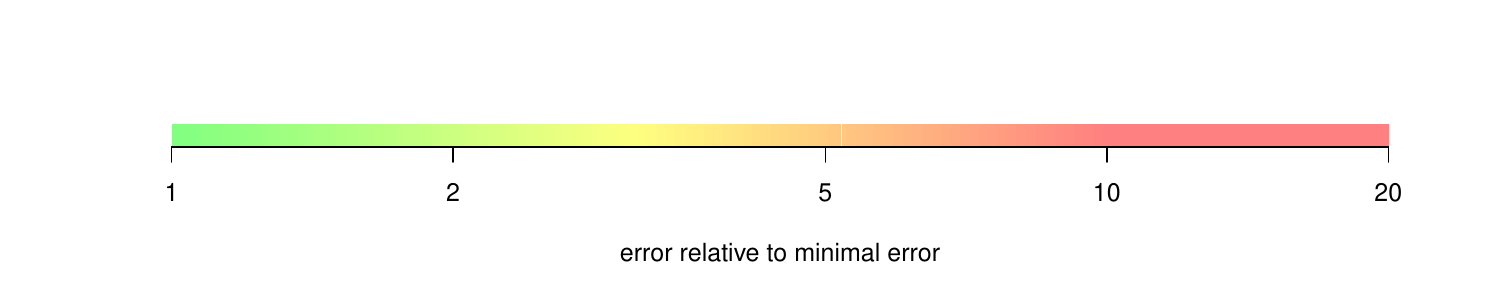}
\caption{Approximated MISE values for nonparametric regression methods. The colors give a visual indication of the MISE value of the given methods divided by the best MISE value in the row.}
\label{tbl:mise:nonparam}
\end{table}
The more reliable analysis of the approximated MISE-values confirms that all estimators behave similar, except \texttt{OrtGeo}, which has some bad outcomes. This may have several reasons: We were not able to show an error bound for this method and argued that it may be sub-optimal, i.e., it may be inherently worse than the other methods. Furthermore, we do not use cross-validation for \texttt{OrtGeo}, as we do for the other methods, but fix $N=3$. Thus, the comparison might be unfair, because the hyper-parameters are not tuned equally. Lastly, in \texttt{OrtGeo}, we have to numerically solve an 8-dimensional non-convex optimization problem (2 dimensions for each of $\hat p$, $\hat v_1$, $\hat v_2$, $\hat v_3$). There are 4 dimensions for \texttt{LocGeo} and 2 for the Fréchet methods, see \autoref{table:opi}. Our program might return values farther away from the optimum in those methods with higher dimensional optimization problems.
\begin{figure}
\includegraphics[width=0.49\textwidth]{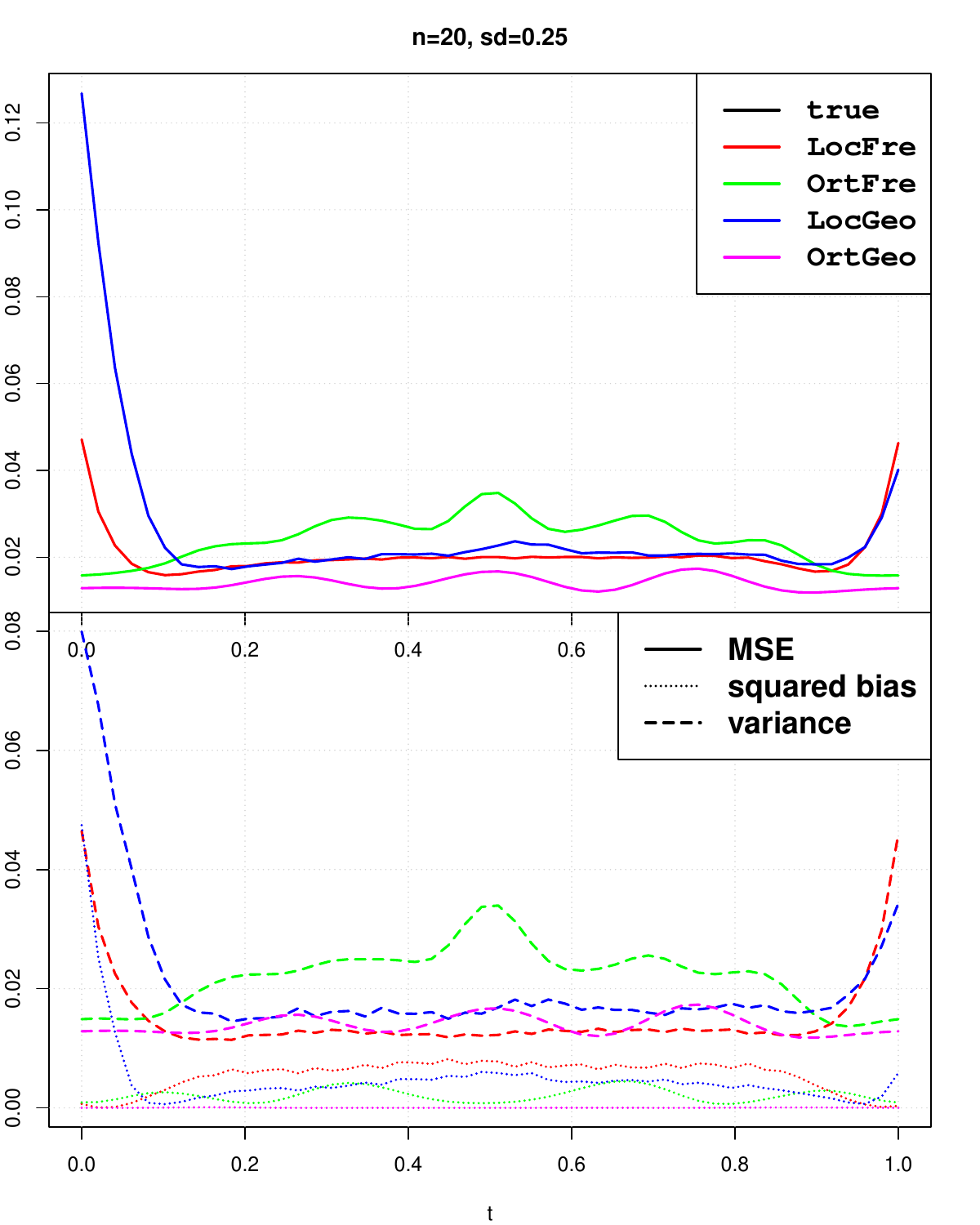}
\includegraphics[width=0.49\textwidth]{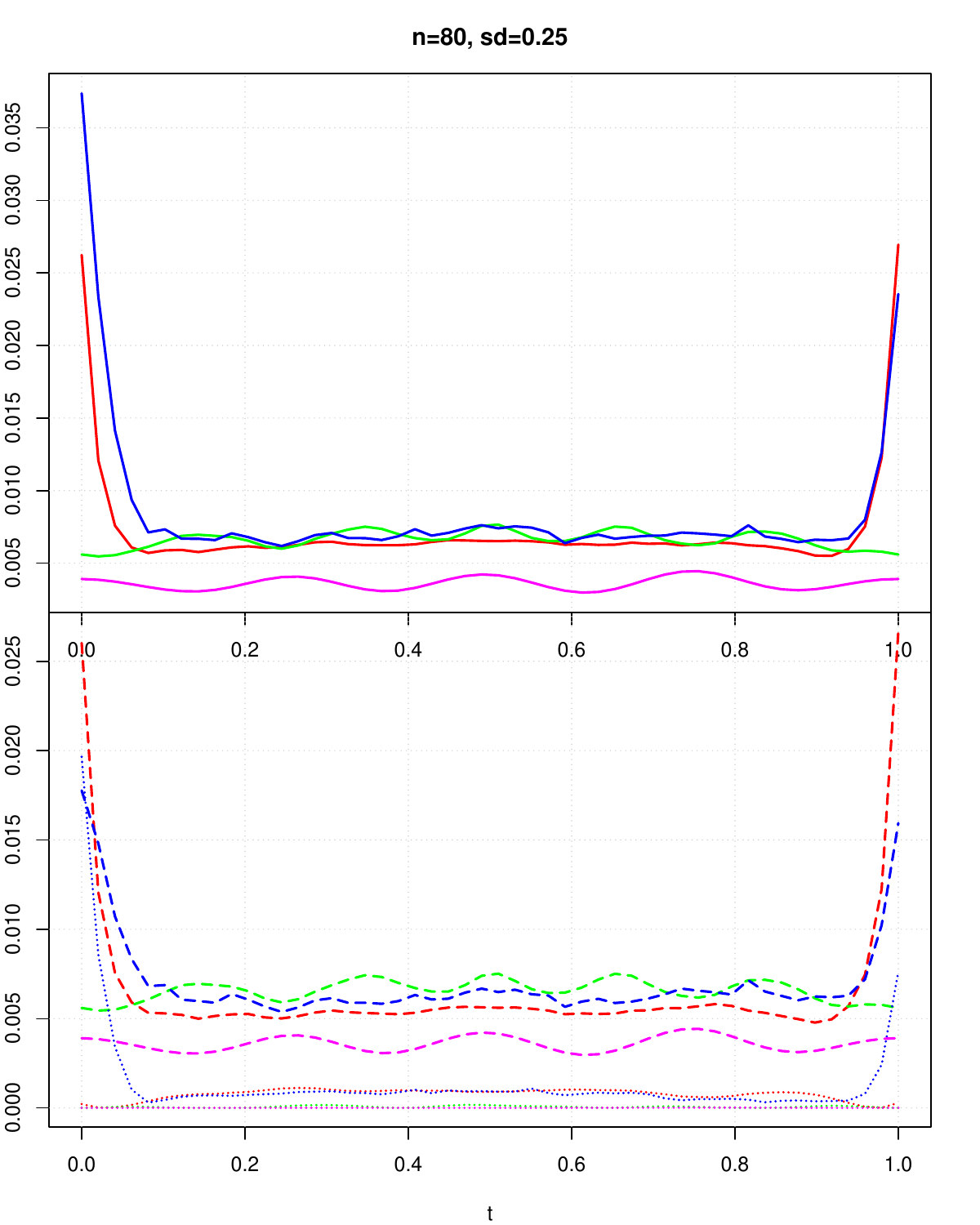}
\includegraphics[width=0.49\textwidth]{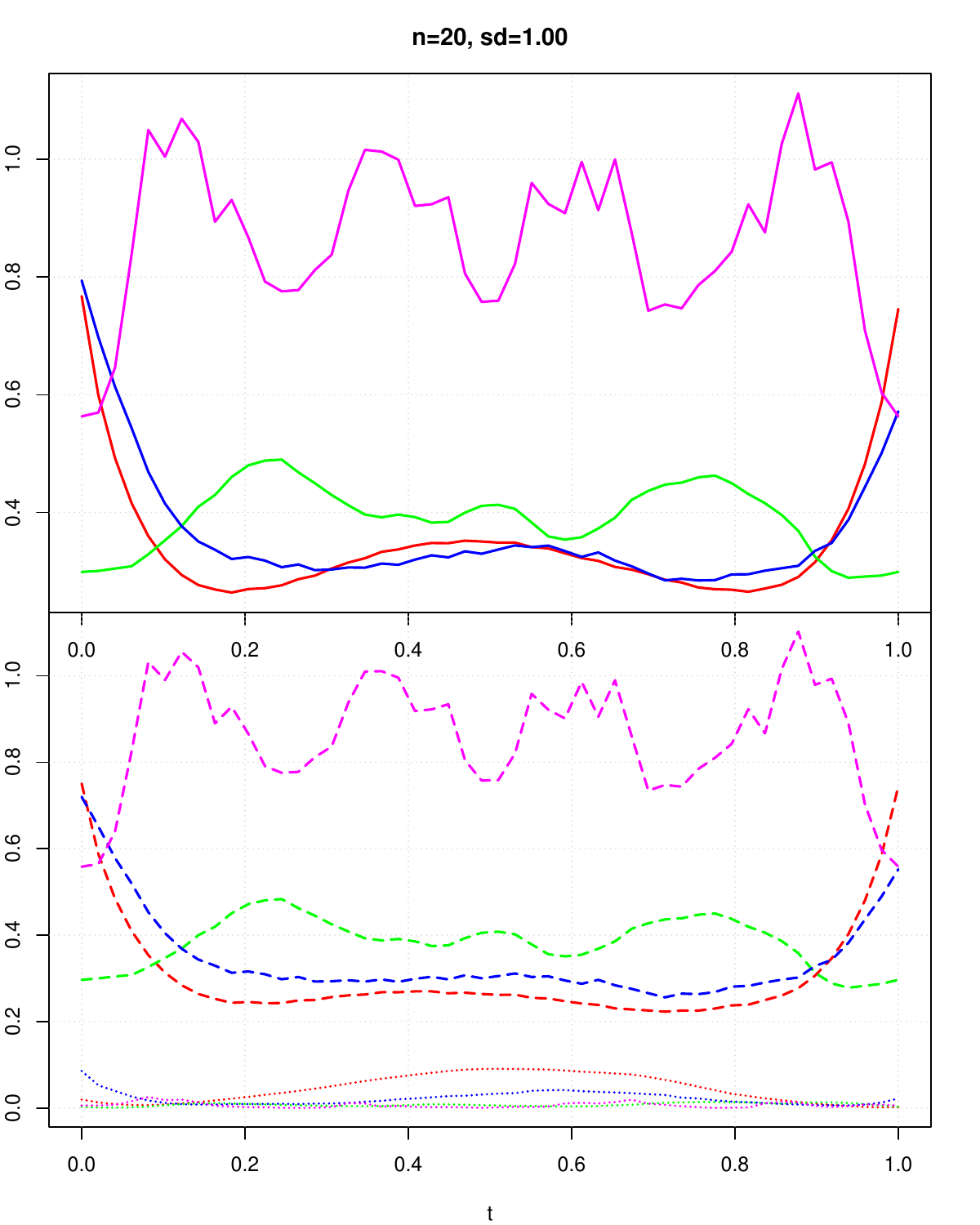}
\includegraphics[width=0.49\textwidth]{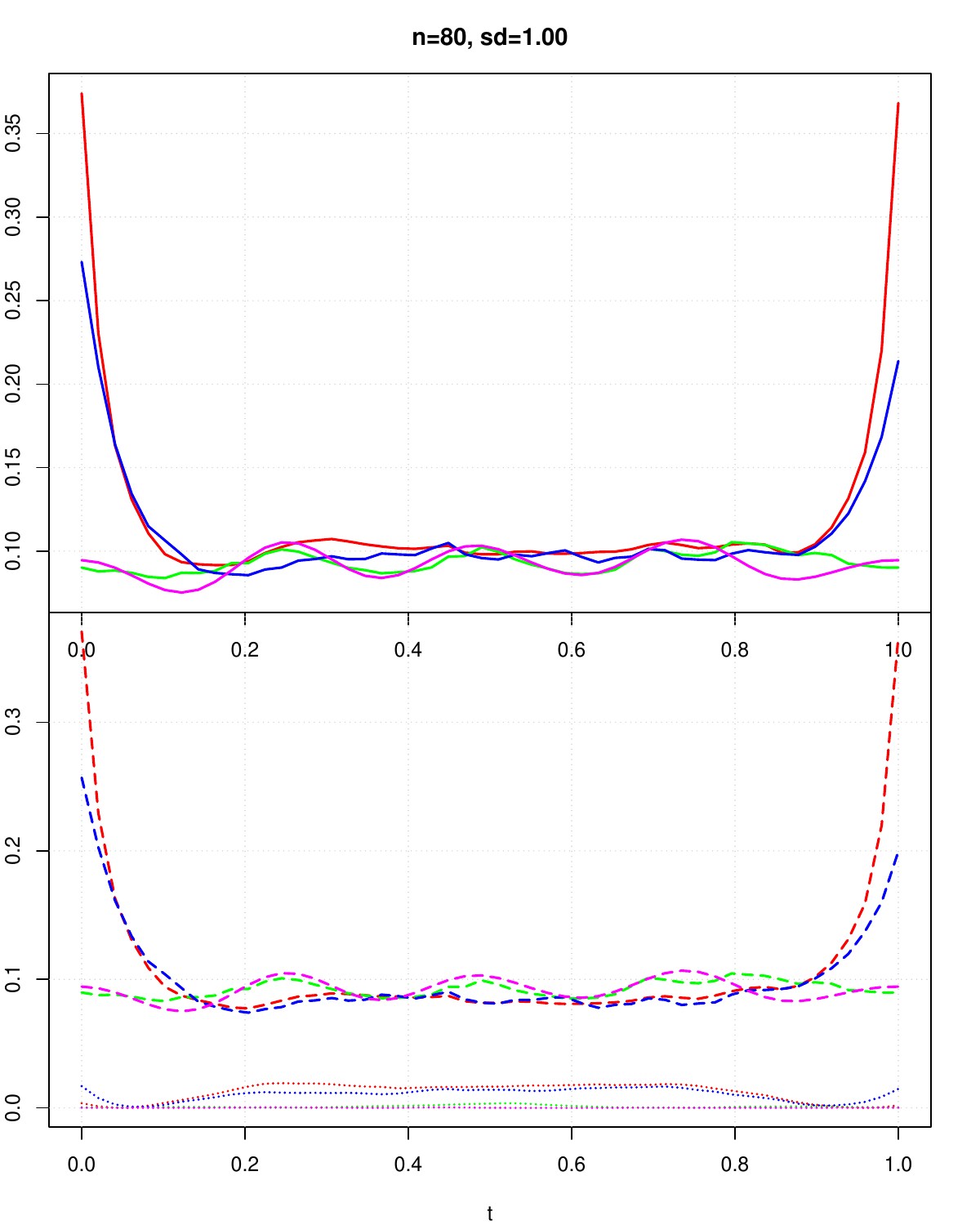}
\caption{Point-wise MSE, squared bias, and variance for the \emph{simple} curve.}
\label{fig:nonparam:mse:closed}
\end{figure}

\begin{figure}
\includegraphics[width=0.49\textwidth]{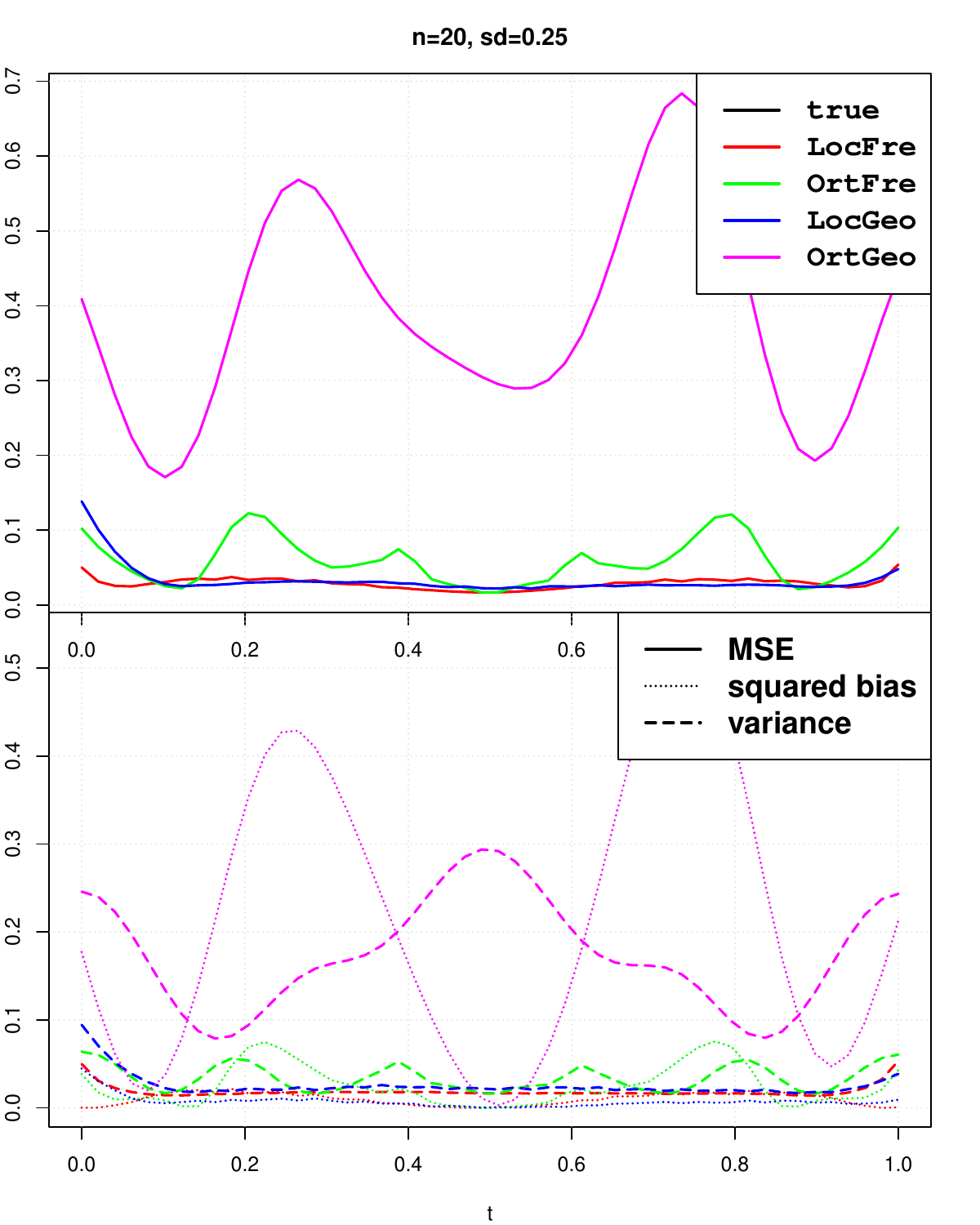}
\includegraphics[width=0.49\textwidth]{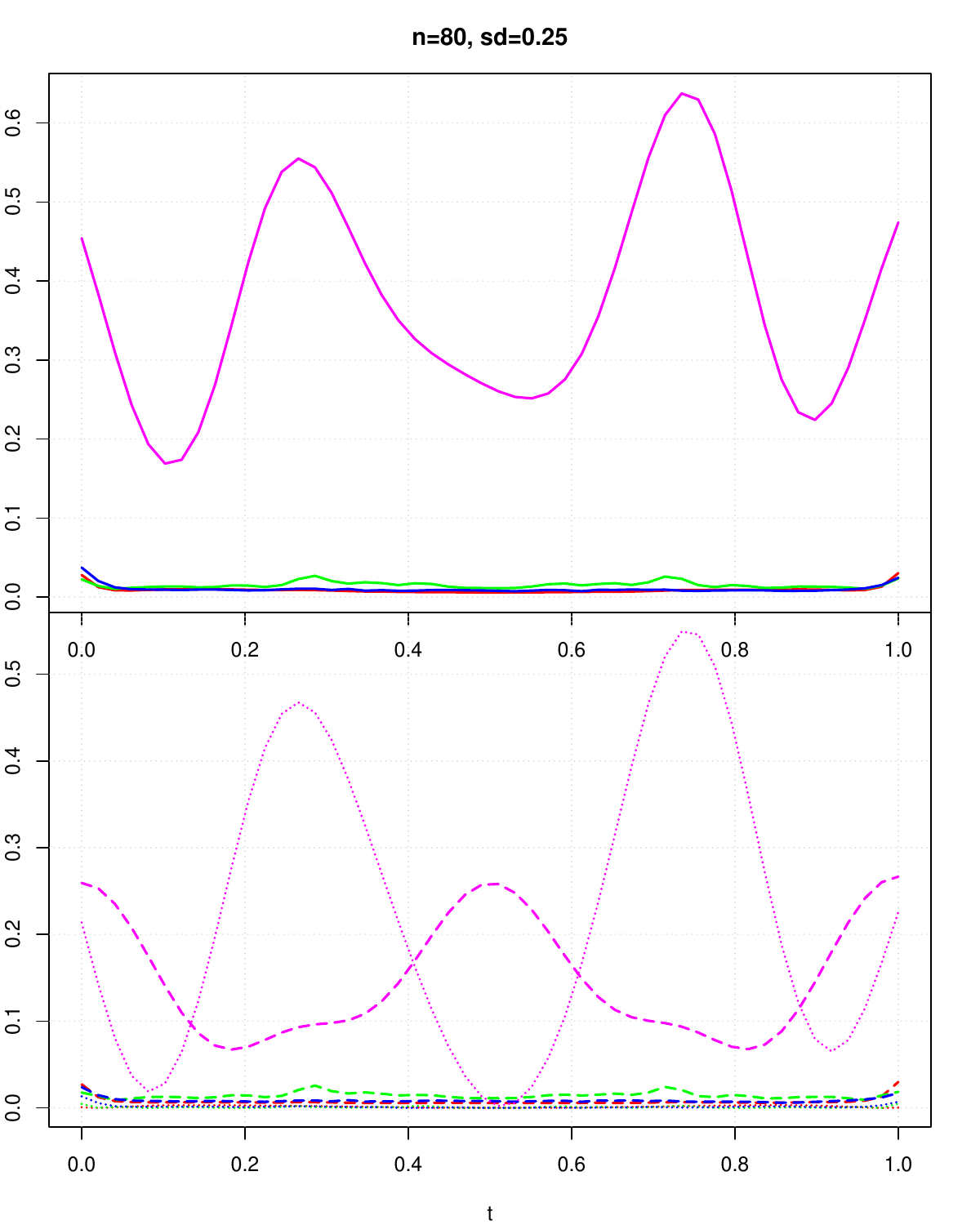}
\includegraphics[width=0.49\textwidth]{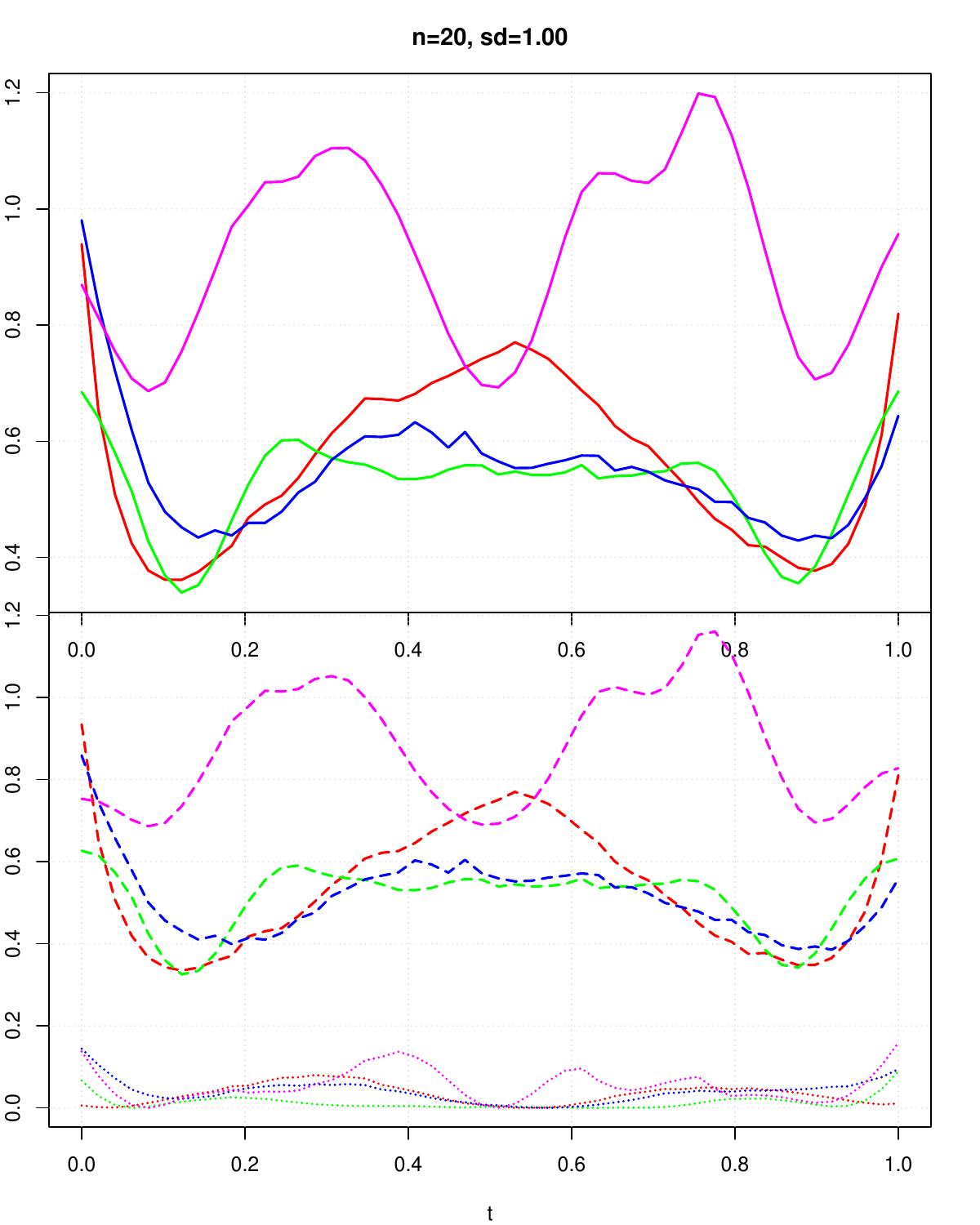}
\includegraphics[width=0.49\textwidth]{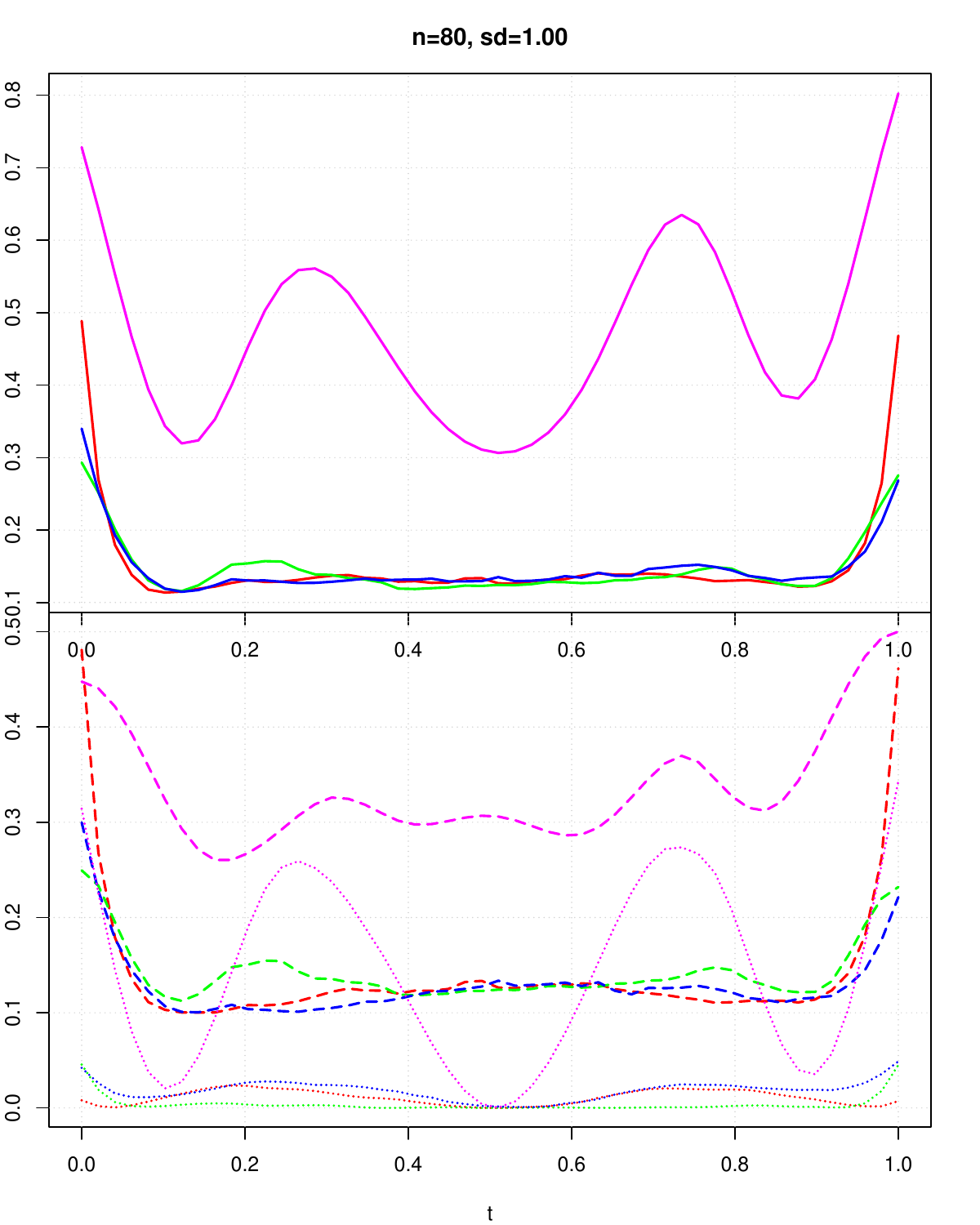}
\caption{Point-wise MSE, squared bias, and variance for the \emph{spiral}.}
\label{fig:nonparam:mse:open}
\end{figure}
\autoref{fig:nonparam:mse:closed} and \autoref{fig:nonparam:mse:open} show the approximated point-wise mean squared error in the upper part of each plot. In the lower part, a point-wise decomposition into a squared bias and a variance term is shown. This decomposition is not straight forward in curved spaces: We calculate the Fréchet mean $\bar m_t$ of our repetitions $(\hat m_t^{j})_{j=1,\dots,500}$. The dotted line in each plot is $t\mapsto d(\bar m_t, m_t)^2 =: \ms{Bias}^2_t$. The dashed line is $\frac1{500} \sum_{j=1}^{500} d(\bar m_t, \hat m_t^{j})^2 =: \ms{Var}_t$. But, in nonstandard spaces, there is no guarantee that $\frac1{500} \sum_{j=1}^{500} d(m_t, \hat m_t^{j})^2 =: \ms{MSE}_t = \ms{Bias}^2_t + \ms{Var}_t$. Still this decomposition is valuable. It shows that \texttt{OrtGeo} is an unbiased estimator of the \emph{simple} curve, which is not surprising as \texttt{OrtGeo} with $N=3$ is a parametric estimator and the \emph{simple} curve is in its model space. On the \textit{spiral} the estimators suffer from boundary effects. On the \textit{simple} curve this only affects the local estimators as this curve is periodic and does not have a boundary for the trigonometric estimators.%
\begin{appendix}
\section{Proofs}\label{sec:proofs}
Recall the general metric space model.
Let $(\mc Q, d)$ be a metric space. For $t \in [0, 1]$, let $Y_t$ be a $\mc Q$-valued random variable with finite second moment, i.e., $\Ex{d(Y_t, q)^2} < \infty$ for all $t\in[0,1]$ and $q\in\mc Q$. Let the regression function $m \colon [0,1] \to \mc Q$ be a minimizer $m_t \in \argmin_{q\in \mc Q} \Ex{d(Y_t, q)^2}$. 
We consider nonparametric estimators which have access to following data: 
Let $x_i = \frac in$ and let $(y_i)_{i=1,\dots, n}$ be independent random variables with values in $\mc Q$ such that $y_i$ has the same distribution as $Y_{x_i}$. 

We introduce some further notation.
Define 
\begin{align*}
	\ol{q}{p} &:= d(q,p)\eqcm\\
	\lozenge(y,z,q,p) &:= d(y, q)^2 - d(y, p)^2 - d(z, q)^2 + d(z, p)^2\eqcm\\
	\mf a(y, z) &:= \sup_{q,p\in\mc Q, q\neq p} \frac{\lozenge(y,z,q,p)}{d(q, p)}
	\eqfs
\end{align*}
\subsection{\texttt{LocFre}}
\subsubsection{A General Result}
To prove the theorems from section \ref{sec:locfre} concerning the \texttt{LocFre} estimator, we show a more general results first.

For $a > 0$, define $\lfloor a\rfloor$ as the largest integer strictly smaller than $a$.
The Hölder class $\Sigma(\beta, L)$ for $\beta, L > 0$ is defined as the set of $\lfloor \beta\rfloor$-times continuously differentiable functions $f\colon [0,1]\to\R$ with $\abs{f^{(\lfloor \beta\rfloor)}(t)-f^{(\lfloor \beta\rfloor)}(x)} \leq L\abs{x-t}^{\beta-\lfloor \beta\rfloor}$ for all $x,t\in [0,1]$.
\begin{assumptions}\label{ass:locfre:smooth_density}\mbox{ }
\begin{itemize}
\item 
	\textsc{VarIneq}: There is $C_{\ms{Vlo}}\in[1,\infty)$ such that $C_{\ms{Vlo}}^{-1}\,\ol{q}{m_t}^2 \leq \Ex{d(Y_t, q)^2 - d(Y_t, m_t)^2}$ for all $q\in\mc Q$ and $t\in[0,1]$.
\item
	\textsc{Entropy}: 
	There are $C_{\ms{Ent}} \in [1,\infty)$ and $\alpha \in [1,2)$ such that 
	\begin{equation*}
		\gamma_2(\mc B, d) \leq C_{\ms{Ent}}\max(\diam(\mc B, d), \diam(\mc B, d)^\alpha)
	\end{equation*} 
	for all $\mc B \subset \mc Q$, where $\gamma_2$ is the measure of entropy defined \autoref{def:entropy}.
\item 
	\textsc{Moment}:
	There are $\kappa > \frac{2}{2-\alpha}$ and $C_{\ms{Mom}} \in [1,\infty)$ such that $\Ex{d(Y_t, m_t)^\kappa}^{\frac1\kappa} \leq C_{\ms{Mom}}$ for all $t \in [0, 1]$.
\item \textsc{Kernel}:
	There are $C_{\ms{Kmi}}, C_{\ms{Kma}} \in [1,\infty)$ such that 
	\begin{equation*}
		C_{\ms{Kmi}}^{-1} \ind_{[-\frac12,\frac12]}(x) \leq K(x) \leq C_{\ms{Kma}} \ind_{[-1,1]}(x)
	\end{equation*}
	for all $x\in\R$.
\item \textsc{HölderSmoothDensity}:
	The function $[0,1] \to \mc Q,\, t\mapsto m_t$ is continuous. Let $C_{\ms{Len}} \in [1,\infty)$ such that $\sup_{s,t\in[0,1]} d(m_s, m_t) \leq C_{\ms{Len}}$. 
	Let $\mu$ be a probability measure on $\mc Q$. Let $C_{\ms{Int}} \in [1,\infty)$ such that $\int \ol{y}{m_0}^2 \mu (\dl y) \leq C_{\ms{Int}}$. 
	Let $y\to\rho(y|t)$ be the $\mu$-density of $Y_t$.
	Let $\beta>0$ with $\ell = \lfloor\beta\rfloor$. For $\mu$-almost all $y\in\mc Q$, there is $L(y)\geq0$ such that $t \mapsto  \rho(y|t) \in \Sigma(\beta, L(y))$. Furthermore, there is a constant $C_{\ms{SmD}} > 0$, $\int L(y)^2 \dl \mu(y) \leq C_{\ms{SmD}}^2$.
\item
	\textsc{BiasMoment}: Define $H(q,p) = (\int  \br{ \ol yq + \ol yp}^2 \mu (\dl y))^\frac12$. There is $C_{\ms{Bom}}\in[1,\infty)$ such that $\Ex{H(\hat m_t,m_t)^\kappa}^{\frac1\kappa} \leq C_{\ms{Bom}}$ for all $t\in[0,1]$.
\end{itemize}
\end{assumptions}
\begin{theorem}[\texttt{LocFre} General]\label{thm:locfre}
	Assume \textsc{HölderSmoothDensity}, \textsc{BiasMoment}, \textsc{Kernel}, \textsc{VarIneq}, \textsc{Entropy}, \textsc{Moment}.	Let $\ell = \lfloor\beta\rfloor$.
	Then, for $t\in[0,1]$, $n \geq c$, and $h \geq \frac cn$, the local polynomial Fréchet estimator $\hat m_{t}$ of order $\ell$ fulfills,
	\begin{equation*}
		\Ex*{\ol {m_{t}}{\hat m_{t}}^2} \leq 
		C_1 \br{h^{2\beta} + h^\frac{2\beta}{2-\alpha}} + C_2 \br{(nh)^{-1} + (nh)^{-\frac{1}{2-\alpha}}}
		\eqcm
	\end{equation*}
	where $C_1 = c_{\alpha,\kappa} \br{C_{\ms{Vlo}} C_{\ms{Kmi}}C_{\ms{Kma}} C_{\ms{SmD}} C_{\ms{Bom}}}^{\frac{2}{2-\alpha}}$ and
	$C_2 = c_{\alpha,\kappa}
	\br{C_{\ms{Vlo}} C_{\ms{Mom}} C_{\ms{Ent}} C_{\ms{Kmi}}^2C_{\ms{Kma}}^2}^{\frac{2}{2-\alpha}}$.
\end{theorem}
To prove \autoref{thm:locfre}, We first apply the variance inequality to relate a bound on the objective functions to a bound on the minimizers. The required uniform bound on the objective functions can be split into a bias and a variance part, which are bounded separately thereafter. Then, these results are put together in the application of a peeling device, which is used to bound the tail probabilities of the error. Integrating the tails leads to the required bounds in expectation.
\subsubsection{Proof of the General Result}
\medskip\noindent\textbf{Kernel.}
First we state some properties of the weights $w_{i,t}$ to be used later.
\begin{lemma}[{\cite[Proposition 1.12, Lemma 1.3, Lemma 1.5]{tsybakov08}}]\label{lmm:locfre:weights}
	Assume \textsc{Kernel}. Let $f \colon \R \to\R$ be a polynomial of degree $\leq \ell$. Then
	\begin{align*}
	&w_{i,t} = 0 \text{ if $\abs{x_i-t} > h$}\eqcm
	&&\sum_{i=1}^n w_{i,t} = 1\eqcm
	&& \sum_{i=1}^n f(x_i) w_{i,t} = f(t)\eqcm\\
	&\abs{w_{i,t}} \leq c\frac{C_{\ms{Kmi}}C_{\ms{Kma}}}{nh}\eqcm
	&&\sum_{i=1}^n |w_{i,t}| \leq cC_{\ms{Kmi}}C_{\ms{Kma}}\eqcm
	&&\sum_{i=1}^n w_{i,t}^2 \leq c\frac{C_{\ms{Kmi}}^2C_{\ms{Kma}}^2}{nh}
	\eqfs
	\end{align*}
	for all $t\in[0,1]$, $h\geq \frac cn$, $n \geq c$.
\end{lemma}
\begin{proof}
The first statement is due to the bounded support of the kernel. For the other statements in the first row, see \cite[Proposition 1.12]{tsybakov08}. The next two bounds follow from \cite[Lemma 1.3, Lemma 1.5]{tsybakov08}. The last bound is a consequence of the previous two.
\end{proof}

\medskip\noindent\textbf{Variance Inequality and Split.}
We define following notation for the objective functions
\begin{align*}
	\hat F_t(q) &= \sum_{i=1}^n w_{i,t} d(y_i, q)^2 & \hat F_t(q,p) &=\hat F_t(q) - \hat F_t(p)\eqcm\\
	\bar F_t(q) &= \sum_{i=1}^n w_{i,t} \Ex{d(y_i, q)^2} & \bar F_t(q,p) &=\bar F_t(q) - \bar F_t(p)\eqcm\\
	F_t(q) &= \Ex{d(Y_t, q)^2} &  F_t(q,p) &= F_t(q) -  F_t(p)
	\eqfs
\end{align*}
Using \textsc{VarIneq} and the minimizing property of $\hat m_t$ we obtain
\begin{align*}
	C_{\ms{Vlo}}^{-1} d(\hat m_t, m_t)^\alpha
	&\leq  
	F_{t}(\hat m_t, m_t)
	\\&\leq  
	F_{t}(\hat m_t, m_t) - \hat F_{t}(\hat m_t, m_t)
	\\&=
	\br{F_{t}(\hat m_t, m_t) - \bar F_{t}(\hat m_t, m_t)}
	+ 
	\br{\bar F_{t}(\hat m_t, m_t) - \hat F_{t}(\hat m_t, m_t)}
\end{align*}
The first parenthesis represents the bias part, the second one the variance part. We will bound the former using \textsc{HölderSmoothDensity}, the later by an empirical process argument.

\medskip\noindent\textbf{Variance.}
Define
\begin{align*}
	Z_i(q) &= w_{i,t} \br{d(y_i, q)^2 - d(y_i, m_{t})^2} - \Ex*{w_{i,t} \br{d(y_i, q)^2 - d(y_i, m_{t})^2}}
	\eqfs
\end{align*}
Then $Z_1, \dots, Z_n$ are independent and centered processes with $Z_i(m_t) = 0$. 
They are integrable due to \textsc{Moment}.
By the definition of $\mf a$,
\begin{equation*}
	\abs{Z_i(q) - Z_i(p) - Z_i\pr(q) +Z_i\pr(p)} \leq \abs{w_{i,t}} \mf a(y_i, y_i\pr) d(q,p)
	\eqcm
\end{equation*}
where $Z_i(q)\pr$ and $y_i\pr$ are independent copies of $Z_i(q)$ and $y_i$, respectively.
\autoref{thm:empproc} yields
\begin{align*}
&\Ex*{\sup_{q\in\ball(m_t, d, \delta)}\abs{\bar F_{t}(q, m_t) - \hat F_{t}(q, m_t)}^\kappa} 
= 
\Ex*{\sup_{q\in\ball(m_t, d, \delta)} \abs{\sum_{i=1}^nZ_i(q)}^\kappa}
\\&\leq 
c_\kappa \br{\Ex*{\br{\sum_{i=1}^n w_{i,t}^2 \mf a(y_i, y_i\pr)^2}^{\frac\kappa2}}^{\frac1\kappa} \gamma_2(\ball(m_t, d, \delta), d)}^\kappa
\end{align*}
for a constant $c_\kappa$ depending only on $\kappa$.
Define $W = \sum_{i=1}^n w_{i,t}^2$ and $v_i = w_{i,t}^2/W$. We apply \textsc{Moment},
\begin{align*}
	\Ex*{\br{\sum_{i=1}^n w_{i,t}^2 \mf a(y_i, y_i\pr)^2}^{\frac\kappa2}}
	&=
	\Ex*{\br{W \sum_{i=1}^n v_i \mf a(y_i, y_i\pr)^2}^{\frac\kappa2}}
	\\&\leq
	\Ex*{W^{\frac\kappa2} \sum_{i=1}^n v_i \mf a(y_i, y_i\pr)^\kappa}
	\\&=
	W^{\frac\kappa2} \sum_{i=1}^n v_i \Ex*{\mf a(y_i, y_i\pr)^\kappa}
	\\&\leq
	W^{\frac\kappa2} C_{\ms{Mom}}^\kappa
	\eqfs
\end{align*}
By \autoref{lmm:locfre:weights}, $W \leq c C_{\ms{Kmi}}^2C_{\ms{Kma}}^2 (nh)^{-1}$. By \textsc{Entropy}, $\gamma_2(\ball(m_t, d, \delta), d) \leq C_{\ms{Ent}} \max(\delta, \delta^\alpha)$. Thus,
\begin{equation*}
\Ex*{\sup_{q\in\ball(m_t, d, \delta)}\abs{\bar F_{t}(q, m_t) - \hat F_{t}(q, m_t)}^\kappa} 
\leq 
c_\kappa \br{C_{\ms{Mom}} C_{\ms{Ent}} C_{\ms{Kmi}}^2C_{\ms{Kma}}^2 \max(\delta, \delta^\alpha) (nh)^{-\frac12}}^\kappa
\eqfs
\end{equation*}

\medskip\noindent\textbf{Bias.}
As $\sum_{i=1}^n w_{i,t} = 1$ (\autoref{lmm:locfre:weights}), we have
\begin{equation*}
	F_{t}(q, m_t) - \bar F_{t}(q, m_t) 
	=
	\sum_{i=1}^n w_{i,t} \Ex{\lozenge(Y_t, y_i, q, m_t)}
	\eqfs
\end{equation*}
Using the $\mu$-density $y \mapsto \rho(y|t)$ of $Y_t$, we can write 
	$\Ex{\ol{Y_t}q^2-\ol{Y_t}p^2} =$ \linebreak $\int \br{\ol{y}q^2 - \ol{y}p^2} \rho(y|t) \dl \mu(y)$.		
	By \textsc{HölderSmoothDensity}, $t\mapsto \rho(y|t) \in \Sigma(\beta, L(y))$. Thus,
	there are $a_k(y)$ such that $\rho(y|x) = R_{y}(x, x_0) + \sum_{k=0}^\ell a_k(y) (x-x_0)^k$ with $\abs{R_{y}(x, x_0)} \leq L(y) \abs{x-x_0}^\beta$. Using that the weights annihilate polynomials of order $\ell$ \cite[equation (1.68)]{tsybakov08}, we obtain
	\begin{align*}
		\sum_{i=1}^n w_{i,t} \Ex{\lozenge(Y_t, y_i, q, p)}
		&=
		\int \sum_{i=1}^n w_{i,t} \br{\ol{y}q^2 - \ol{y}p^2} \br{\rho(y|t)-\rho(y|x_i)} \dl \mu(y)
		\\&=
		\int \sum_{i=1}^n w_{i,t} \br{\ol{y}q^2 - \ol{y}p^2} R_{y}(t, x_i) \dl \mu(y)
		\\&\leq
		\int \sum_{i=1}^n \abs{w_{i,t}} \abs{\ol{y}q^2 - \ol{y}p^2} \abs{R_{y}(t, x_i)} \dl \mu(y)
		\eqfs
	\end{align*}
	It holds
	\begin{align*}
		\abs{\ol{y}q^2 - \ol{y}p^2}\abs{R_{y}(x, x_0)}
		&\leq
		\ol qp \abs{x - x_0}^\beta \br{\ol yq + \ol yp} L(y)
		\eqfs
	\end{align*}
	Together with $\sum_{i=1}^n \abs{w_{i,t}} \leq c C_{\ms{Kmi}}C_{\ms{Kma}}$ from \autoref{lmm:locfre:weights}, we obtain
	\begin{equation*}
		\abs{\sum_{i=1}^n w_{i,t} \Ex{\lozenge(Y_t, y_i, q, p)}} \leq c C_{\ms{Kmi}}C_{\ms{Kma}}\, \ol qp\, h^\beta \int\br{\ol yq + \ol yp} L(y) \dl \mu(y)
	\end{equation*}
	Recall $H(q,p) = \br{\int  \br{\ol yq+ \ol yp}^2 \mu (\dl y)}^\frac12$. By the Cauchy--Schwartz inequality and \textsc{HölderSmoothDensity},
	\begin{equation*}
		\int\br{\ol yq + \ol yp} L(y) \dl \mu(y) \leq H(q,p) \br{\int L(y)^2 \dl \mu(y)}^{\frac12} \leq H(q,p) C_{\ms{SmD}}
		\eqfs
	\end{equation*}
	Thus, 
	\begin{equation}\label{eq:locfre:replace}
		F_{t}(q, m_t) - \bar F_{t}(q, m_t)  \leq c C_{\ms{Kmi}}C_{\ms{Kma}}C_{\ms{SmD}}\, \ol qp\, h^\beta H(q,m_t) 
	\end{equation}
	\textsc{BiasMoment} states $\Ex{H(\hat m_t,m_t)^\kappa}^{\frac1\kappa} \leq C_{\ms{Bom}}$. Finally we obtain
\begin{align*}
	&\Ex*{\abs{F_{t}(\hat m_t, m_t) - \bar F_{t}(\hat m_t, m_t)}^\kappa \ind_{[0,\delta]}(d(\hat m_t, m_t))}^{\frac1\kappa}
	\\&\leq
	\Ex*{\abs{c C_{\ms{Kmi}}C_{\ms{Kma}} C_{\ms{SmD}} d(\hat m_t, m_t) H(\hat m_t, m_t) h^\beta}^\kappa \ind_{[0,\delta]}(d(\hat m_t, m_t))}^{\frac1\kappa}
	\\&\leq
	c C_{\ms{Kmi}}C_{\ms{Kma}} C_{\ms{SmD}} C_{\ms{Bom}}\delta h^\beta
	\eqfs
\end{align*}

\medskip\noindent\textbf{Peeling.}
For $\delta > 0$ define
\begin{equation*}
	\Delta_\delta(q,p) = \br{\abs{F_{t}(q, p) - \bar F_{t}(q, p)} + \abs{\bar F_{t}(q, p) - \hat F_{t}(q, p)}} \ind_{[0, \delta]}(d(q, p))
	\eqfs
\end{equation*}
Recall that the variance inequality implies
\begin{equation*}
	C_{\ms{Vlo}}^{-1}d(\hat m_t, m_t)^2
	\leq
	\br{F_{t}(\hat m_t, m_t) - \bar F_{t}(\hat m_t, m_t)}
	+ 
	\br{\bar F_{t}(\hat m_t, m_t) - \hat F_{t}(\hat m_t, m_t)}
	\eqfs
\end{equation*}
Let $0 < a < b < \infty$. The inequality above and Markov's inequality yield
\begin{align*}
	\PrOf{d(\hat m_t, m_t) \in [a,b]} 
	\leq 
	\PrOf{a^2 \leq C_{\ms{Vlo}} \Delta_{b}(\hat m_t, m_t)}
	\leq 
	\frac{C_{\ms{Vlo}}^\kappa \Ex{\Delta_{b}(\hat m_t, m_t)^\kappa}}{a^{2\kappa}}
	\eqfs
\end{align*}
Our previous consideration allow us the bound the expectation by a variance and a bias term:
\begin{align*}
	\Ex{\Delta_\delta(\hat m_t, m_t)^\kappa}
	&\leq
	2^{\kappa-1} \Bigg(\Ex*{\abs{F_{t}(\hat m_t, m_t) - \bar F_{t}(\hat m_t, m_t)}^\kappa \ind_{[0,\delta]}(d(\hat m_t, m_t))} 
		\\&\quad+ \Ex*{\sup_{q\in\ball(m_t, d, \delta)}\abs{\bar F_{t}(q, m_t) - \hat F_{t}(q, m_t)}^\kappa}\Bigg)
	\\&\leq
	c_\kappa \br{ C_{\ms{Kmi}}C_{\ms{Kma}} C_{\ms{SmD}} C_{\ms{Bom}} h^\beta + C_{\ms{Mom}} C_{\ms{Ent}}  C_{\ms{Kmi}}^2C_{\ms{Kma}}^2 (nh)^{-\frac12}}^\kappa \max(\delta, \delta^\alpha)^\kappa
	\eqfs
\end{align*}
We are now prepared to apply peeling (also called slicing): Let $s > 0$. Set $A= C_{\ms{Vlo}} C_{\ms{Kmi}}C_{\ms{Kma}} C_{\ms{SmD}} C_{\ms{Bom}} h^\beta + C_{\ms{Vlo}} C_{\ms{Mom}} C_{\ms{Ent}} c C_{\ms{Kmi}}^2 C_{\ms{Kma}}^2 (nh)^{-\frac12}$. It holds
\begin{align*}
	\PrOf{d(\hat m_t, m_t) > s} 
	&\leq 
	\sum_{k=0}^\infty \PrOf{d(\hat m_t, m_t) \in [2^ks, 2^{k+1}s]} 
	\\&\leq 
	\sum_{k=0}^\infty 
	\frac{c_\kappa A^\kappa \max(2^{k+1}s, (2^{k+1}s)^\alpha)^\kappa}{(2^ks)^{2\kappa}}
	\\&\leq 
	c_\kappa A^\kappa \br{s^{-\kappa}+s^{-\kappa(2-\alpha)}} \sum_{k=0}^\infty 2^{-k\kappa(2-\alpha)}
	\\&\leq 
	c_{\kappa} A^\kappa \br{s^{-\kappa}+s^{-\kappa(2-\alpha)}}
	\eqfs
\end{align*}
We integrate this tail bound to bound the expectation. For this we require $\kappa > \frac{2}{2-\alpha}$. Set $B =c_\kappa A^\kappa$, then
\begin{align*}
\Ex{d(\hat m_t, m_t)^2} 
&= 
2 \int_0^\infty s \PrOf{d(\hat m_t, m_t) > s} \dl s
\\&\leq
2 \int_0^\infty s \min\brOf{1, B\br{s^{-\kappa}+s^{-\kappa(2-\alpha)}}} \dl s
\\&\leq
2 \int_0^\infty s \min\brOf{1, Bs^{-\kappa}} \dl s
+
2 \int_0^\infty s \min\brOf{1, Bs^{-\kappa(2-\alpha)}} \dl s
\eqfs
\end{align*}
For the first summand,
\begin{align*}
2\int_0^\infty s \min\brOf{1, Bs^{-\kappa}} \dl s
&=
2\int_0^{B^{\frac1\kappa}} s\, \dl s + 2B\int_{B^{\frac1\kappa}}^\infty s^{1-\kappa} \dl s
\\&=
B^{\frac{2}{\kappa}} + \frac{2B}{\kappa - 2} B^\frac{2-\kappa}{\kappa}
\\&=
\frac{\kappa}{\kappa - 2} B^\frac{2}{\kappa}
\eqfs
\end{align*}
Similarly,
\begin{equation*}
2\int_0^\infty s \min\brOf{1, Bs^{-\kappa(2-\alpha)}} \dl s \leq \frac{\kappa(2-\alpha)}{\kappa(2-\alpha) - 2} B^\frac{2}{\kappa(2-\alpha)}
\end{equation*}
Thus, 
\begin{align*}
\Ex{d(\hat m_t, m_t)^2}  
&\leq 
c_\kappa \br{A^2 + A^{\frac{2}{2-\alpha}}}
\\&\leq
c_{\alpha,\kappa} \br{C_{\ms{Vlo}} C_{\ms{Kmi}}C_{\ms{Kma}} C_{\ms{SmD}} C_{\ms{Bom}}}^{\frac{2}{2-\alpha}} \br{h^{2\beta} + h^\frac{2\beta}{2-\alpha}} +
\\&\hphantom{\leq}\ \, c_{\alpha,\kappa}
\br{C_{\ms{Vlo}} C_{\ms{Mom}} C_{\ms{Ent}} C_{\ms{Kmi}}^2C_{\ms{Kma}}^2}^{\frac{2}{2-\alpha}} \br{(nh)^{-1} + (nh)^{-\frac{1}{2-\alpha}}}
\eqfs
\end{align*} 
\subsubsection{Main Theorems}
We use \autoref{thm:locfre} to prove the two main theorems concerning \texttt{LocFre}.
Recall $H(q,p) = \br{\int  \br{\ol yq+ \ol yp}^2 \mu (\dl y)}^\frac12$. 
\begin{proof}[Proof of \autoref{cor:locfre:bounded}]
	We want to apply \autoref{thm:locfre} with $\alpha=1$.
	As $\diam(\mc Q, d) < \infty$, $H(q,p) \leq 2 \diam(\mc Q)$ for all $q,p\in\mc Q$, and we can set $C_{\ms{Bom}} = 2\,\diam(\mc Q, d)$.
	Furthermore, $\ol yq^2 - \ol yp^2 - \ol zq^2 + \ol zp^2 \leq 4 \ol qp \diam(\mc Q, d)$. Thus, $\mf a(y, z) \leq 4  \diam(\mc Q, d)$ and we can choose $C_{\ms{Mom}} = 4  \diam(\mc Q, d)$. 
	Lastly, we may integrate the inequality $\Ex{\ol {m_{t}}{\hat m_{t}}^2} \leq C_1 h^{2\beta} + C_2 (nh)^{-1}$ with respect to $t$ to obtain the bound for the mean integrated squared error.
\end{proof}
\begin{proposition}\label{prop:locfre:BomBound}
	Let $\mc Q$ be a Hadamard space. Assume \textsc{HölderSmoothDensity}, \textsc{Kernel}, \textsc{Moment}.
	To fulfill \textsc{BiasMoment}, we can choose
	\begin{equation*}
		C_{\ms{Bom}} = c_\kappa C_{\ms{Mom}} C_{\ms{Kmi}}C_{\ms{Kma}} C_{\ms{Len}}C_{\ms{Int}}
		\eqfs
	\end{equation*}
\end{proposition} 
\begin{proof}[Proof of \autoref{prop:locfre:BomBound}]
	Using the triangle inequality
	\begin{align*}
		H(q,p)^2 
		&= 
		\int  \br{\ol yq + \ol yp}^2 \mu (\dl y)
		\\&\leq 
		\int  \br{\ol qp + 2\ol yp}^2 \mu (\dl y)
		\\&\leq 
		2\int  \ol qp^2 + 4\ol yp^2 \mu (\dl y)
		\\&\leq 
		2 \ol qp^2 + 8 \int \ol yp^2 \mu (\dl y)
	\end{align*}
	as $\mu$ is a probability measure.
	\begin{align*}
		\Ex{H(\hat m_t,m_t)^\kappa}^{\frac1\kappa} 
		&\leq 
		\Ex*{\br{2 \ol {\hat m_t}{m_t}^2 + 8 \int \ol y{m_t}^2 \mu (\dl y)}^{\frac\kappa2}}^{\frac1\kappa}
		\\&\leq 
		c_\kappa \br{\Ex*{\ol {\hat m_t}{m_t}^{\kappa}}^\frac1\kappa +  \br{\int \ol y{m_0}^2 \mu (\dl y)}^{\frac12} +  \ol{m_t}{m_0}}
		\\&\leq 
		c_\kappa \br{\Ex*{\ol {\hat m_t}{m_t}^{\kappa}}^\frac1\kappa +  C_{\ms{Int}} +  C_{\ms{Len}}}
		\eqfs
	\end{align*}
	Next, we will bound $\Ex{\ol{m_t}{\hat m_t}^\kappa}$.
	Let $W = \sum_{i=1}^n \abs{w_{i,t}}$.
	First, as \textsc{VarIneq} holds in Hadamard spaces with $C_{\ms{Vlo}}=1$, $\lozenge(y, z, q, p) \leq 2 \ol yz\,\ol qp$ in Hadamard spaces, and the minimizing property of $\hat m_t$,
	\begin{align*}
		\ol{m_t}{\hat m_t}^2
		&\leq 
		F_t(\hat m_t, m_t)
		\\&\leq 
		F_t(\hat m_t, m_t) - \hat F_t(\hat m_t, m_t)
		\\&=
		\sum_{i=1}^n w_{i,t} \Ex{\lozenge(Y_t, y_i, m_t, \hat m_t) | y_{1\dots n}}
		\\&\leq 
		2 \sum_{i=1}^n \abs{w_{i,t}} \ol {\hat m_t}{m_t}\, \Ex{d(Y_t, y_i) | y_i}
		\eqfs
	\end{align*}
	Thus,
	\begin{equation*}
		\ol{m_t}{\hat m_t} \leq 
		\sum_{i=1}^n \abs{w_{i,t}} \Ex{d(Y_t, y_i) | y_i}
	\end{equation*}
	With Jensen's inequality
	\begin{align*}
		\Ex{\ol{m_t}{\hat m_t}^\kappa}
		&\leq
		\Ex*{\br{\sum_{i=1}^n \abs{w_{i,t}} \Ex{d(Y_t, y_i) | y_i}}^\kappa}
		\\&=
		W^\kappa \Ex*{\br{\sum_{i=1}^n \frac{\abs{w_{i,t}}}{W} \Ex{d(Y_t, y_i) | y_i}}^\kappa}
		\\&\leq
		W^\kappa \sum_{i=1}^n \frac{\abs{w_{i,t}}}{W}  \Ex*{\Ex{d(Y_t, y_i) | y_i}^\kappa}
		\\&\leq
		W^\kappa \sum_{i=1}^n \frac{\abs{w_{i,t}}}{W}  \Ex{d(Y_t, y_i)^\kappa}
		\\&\leq
		W^\kappa \sup_{s,t\in[0,1]}\Ex*{d(Y_t, Y\pr_s)^\kappa}
		\eqfs
	\end{align*}
	As $d$ is a metric,
	\begin{align*}
		\Ex*{d(Y_t, Y\pr_s)^\kappa}
		&\leq
		\Ex*{\br{d(Y_t, m_t) + d(m_t, m_s) + d(m_s, Y\pr_s)}^\kappa}
		\\&\leq
		3^\kappa \br{2\sup_{t\in[0,1]}\Ex*{d(Y_t, m_t)^\kappa} + d(m_t, m_s)^\kappa}
		\\&\leq
		c_\kappa \br{C_{\ms {Mom}}^\kappa +C_{\ms{Len}}^\kappa}
		\eqfs
	\end{align*}
	\autoref{lmm:locfre:weights} shows $W \leq c C_{\ms{Kmi}}C_{\ms{Kma}}$. This completes the proof.
\end{proof}
\begin{proof}[Proof of \autoref{cor:locfre:hadamard}]
	We want to apply \autoref{thm:locfre}.
	\textsc{VarIneq} holds in Hadamard spaces with $C_{\ms{Vlo}} = 1$. Furthermore, the quadruple inequality in Hadamard spaces yields $\mf a(y, z) = 2d(y, z)$, which allows to state the moment condition with respect to $d$ instead of $\mf a$.
	We bound $\Ex{H(\hat m_t,m_t)^\kappa}^{\frac1\kappa} \leq C_{\ms{Bom}}$ using
	\begin{equation*}
		C_{\ms{Bom}} = c_\kappa C_{\ms{Mom}}C_{\ms{Kmi}}C_{\ms{Kma}}C_{\ms{Len}}C_{\ms{Int}}\eqcm
	\end{equation*}
	see \autoref{prop:locfre:BomBound}.
	Lastly, we may integrate the inequality 
	\begin{equation*}
		\Ex*{\ol {m_{t}}{\hat m_{t}}^2} \leq C_1 \br{h^{2\beta}+h^{\frac{2\beta}{2-\alpha}}} + C_2 \br{(nh)^{-1} + (nh)^{-\frac{1}{2-\alpha}}}
	\end{equation*}
	with respect to $t$ to obtain the bound for the mean integrated squared error.
\end{proof}
\subsection{\texttt{OrtFre}}
\subsubsection{A General Result}
We prove a general theorem that implies the main theorems concerning \texttt{OrtFre}. 

\begin{assumptions}\label{ass:trifre:app}\mbox{ }
\begin{itemize}
\item 
	\textsc{VarIneq}: There is $C_{\ms{Vlo}}\in[1,\infty)$ such that $C_{\ms{Vlo}}^{-1} \,\ol q{m_t}^2 \leq \Ex{\ol{Y_t}{q}^2 - \ol{Y_t}{m_t}^2}$ for all $q\in\mc Q$ and $t\in[0,1]$.
\item
	\textsc{Entropy}: 
	There are $C_{\ms{Ent}} \in [1,\infty)$ and $\alpha \in [1,2)$ such that 
	\begin{equation*}
		\gamma_2(\mc B, d) \leq C_{\ms{Ent}}\max(\diam(\mc B, d), \diam(\mc B, d)^\alpha)
	\end{equation*}
	for all $\mc B \subset \mc Q$, where $\gamma_2$ is the measure of entropy defined \autoref{def:entropy}.
\item 
	\textsc{Moment}:
	There are $\kappa > \frac{2}{2-\alpha}$ and $C_{\ms{Mom}} \in [1,\infty)$ such that $\Ex{d(Y_t, m_t)^\kappa}^{\frac1\kappa} \leq C_{\ms{Mom}}$ for all $t \in [0, 1]$.
\item \textsc{SobolevSmoothDensity}: 
	The function $[0,1] \to \mc Q,\, t\mapsto m_t$ is continuous. Let $C_{\ms{Len}} \in [1,\infty)$ such that $\sup_{s,t\in[0,1]} d(m_s, m_t) \leq C_{\ms{Len}}$. 
	Let $\mu$ be a probability measure on $\mc Q$. Let $C_{\ms{Int}} \in [1,\infty)$ such that $\int \ol{y}{m_0}^2 \mu (\dl y) \leq C_{\ms{Int}}$. 
	For all $t\in[0,1]$, the random variable $Y_t$ has a density $y\mapsto \rho(y|t)$ with respect to $\mu$. 
	Let $\beta\geq 1$. For $\mu$-almost all $y\in\mc Y$, there is $L(y)\geq0$ such that $t \mapsto \rho(y|t) \in W^{\ms{per}}(\beta, L(y))$. Furthermore, there is $C_{\ms {SmD}}\in[1,\infty)$ such that $\int L(y)^2 \dl \mu(y) \leq C_{\ms {SmD}}^2$. 
\item
	\textsc{BiasMoment}: Define $H(q,p) = \br{\int  \br{ \ol yq + \ol yp}^2 \mu (\dl y)}^\frac12$. There is $C_{\ms{Bom}}\in[1,\infty)$ such that $\Ex{H(\hat m_t,m_t)^\kappa}^{\frac1\kappa} \leq C_{\ms{Bom}}$ for all $t\in[0,1]$.
\end{itemize}
\end{assumptions}
\begin{theorem}[\texttt{OrtFre} General]\label{thm:trifre}
	Assume \textsc{VarIneq}, \textsc{Entropy} with $\alpha=1$, \textsc{Moment}, \textsc{BiasMoment},  \textsc{SobolevSmoothDensity}.
	Then
	\begin{equation*}
		\Ex*{\int_0^1 \ol{m_t}{\hat m_t}^2 \dl t} \leq 
				C_1 \br{N^{-2\beta} + 
					N n^{1-2\beta}} + C_2 \frac Nn
				\eqcm 
	\end{equation*}
	where
	$C_1 =	c_{\kappa,\beta} C_{\ms{Vlo}}^2	C_{\ms{SmD}}^2 C_{\ms{Bom}}^2$
	and
	$C_2 =	c_{\kappa,\beta} C_{\ms{Vlo}}^2	C_{\ms{Mom}}^2 C_{\ms{Ent}}^2$.
\end{theorem}
The difference of the objective functions is split into three parts in \autoref{lmm:split}.
In \autoref{lmm:peeling}, we use a peeling device and the variance inequality to relate this difference to the distance between the minimizers $\hat m_t$ and $m_t$, which is the quantity to be bounded in the theorem. Of the three parts, two bias related quantities are bounded in \autoref{lmm:trifre:r} and \autoref{lmm:trifre:fr} with an auxiliary result in \autoref{lmm:bessel}. The third part, a variance term, is bounded in \autoref{lmm:fepsilon} via chaining. The bounds on the three parts are summarized in \autoref{lmm:kappa_moment}. In the end, the integral over $t$ is applied to calculate the mean integrated squared error. Here, the auxiliary result \autoref{lmm:h} is applied.
\subsubsection{Proof of the General Result}
For shorter notation define $F_t(q,p) := F_t(q)-F_t(p)$ and $\hat F_t(q,p) := \hat F_t(q)-\hat F_t(p)$.
We introduce the Fourier coefficients $\vartheta_j(q,p)$ of $t \mapsto F_{t}(q,p)$ with respect to the trigonometric basis
\begin{equation*} 
	\vartheta_j(q,p) = \int_0^1 \psi_j(x) F_x(q,p) \dl x
\end{equation*}
such that $F_t(q,p) = \sum_{j=1}^\infty \vartheta_j(q,p) \psi_j(t)$ due to \textsc{SobolevSmoothDensity}.
Define 
\begin{align*}
r_t(q,p) &= \sum_{k=N+1}^\infty \vartheta_j(q,p) \psi_j(t)\eqcm
&
F^r_t(q,p) &= \Psi_N(t)\tr \frac1n \sum_{i=1}^n \Psi_N(x_i) r_{x_i}(q,p)\eqcm
\\
\varepsilon_t(y,q,p) &= F_t(q,p)- \br{\ol yq^2 - \ol yp^2}\eqcm
&
F^\varepsilon_t(q,p) &= \Psi_N(t)\tr \frac1n \sum_{i=1}^n \Psi_N(x_i) \varepsilon_{x_i}(y_i,q,p)\eqfs
\end{align*}
\begin{lemma}\label{lmm:split}
	If $N<n$, then
	\begin{equation*}
		F_t(q,p) - \hat F_t(q,p) = r_t(q, p) + F^\varepsilon_t(q,p) - F^r_t(q,p)
		\eqfs
	\end{equation*}
\end{lemma}
\begin{proof}[Proof of \autoref{lmm:split}]\label{proof:lmm:split}
	It holds
	\begin{equation*}
		\frac1n \sum_{i=1}^n \psi_j(x_i)\psi_{\tilde j}(x_i) = \delta_{j \tilde j}
	\end{equation*}
	for $j,\ell\in \cb{1, \dots, n-1}$, see \cite[Lemma 1.7]{tsybakov08}. 
	Set
	\begin{equation*}
		F_t^N(q,p) = \sum_{k=1}^N \vartheta_j(q,p) \psi_j(t)
		\eqfs
	\end{equation*}
	Then $\frac1n \sum_{i=1}^n \psi_j(x_i) F_{x_i}^N(q,p) = \vartheta_j(q,p)$ for $j\leq N < n$. Thus,
	\begin{equation*}
		F_t^N(q,p) = \Psi_N(t)\tr \frac1n \sum_{i=1}^n \Psi_N(x_i) F_{x_i}^N(q,p)
		\eqfs
	\end{equation*}
	As $F_t(q,p) - r_t(q, p) = F_t^N(q,p)$, we obtain
	\begin{align*}
		&F_t(q,p) - \hat F_t(q,p) - r_t(q,p) 
		\\&=
		\Psi_N(t)\tr \frac1n \sum_{i=1}^n \Psi_N(x_i) F^N_{x_i}(q,p) - \Psi_N(t) \frac1n \sum_{i=1}^n \Psi_N(x_i) \br{\ol {y_i}q^2 - \ol{y_i}p^2}
		\\&=
		\Psi_N(t)\tr \frac1n \sum_{i=1}^n \Psi_N(x_i) \br{F^N_{x_i}(q,p) - F_{x_i}(q,p) + F_{x_i}(q,p) - \br{\ol {y_i}q^2 - \ol{y_i}p^2}}
		\\&=
		\Psi_N(t)\tr \frac1n \sum_{i=1}^n \Psi_N(x_i) \br{-r_{x_i}(q,p) + \varepsilon_{x_i}(y_i,q,p)}
		\\&=
		F^\varepsilon_t(q,p) - F^r_t(q,p)
		\eqfs
	\end{align*}
\end{proof}
Next, we apply the peeling device.
\begin{lemma}\label{lmm:peeling}
	For $b > 0$, define 
	\begin{equation*}
		U_{t,b} = \sup_{q\in \ball(m_t,b,d)} F_t^\epsilon(q, m_t) + 
		\br{r_t(\hat m_t,m_t) - 
			F^r_t(\hat m_t,m_t)} \ind_{[0,b]}(\ol{\hat m_t}{m_t}))
		\eqfs
	\end{equation*}
	Let $\kappa > 2$. Define
	\begin{equation*}
		h(t) = \sup_{b > 0} \br{\frac{\Ex{U_{t,b}^\kappa}}{b^{\kappa}}}^{\frac1\kappa}
	\end{equation*}
	Assume \textsc{VarIneq}.
	Then
	\begin{equation*}
		\Ex*{\ol{\hat m_t}{m_t}^2} 
		\leq
		\frac{4\kappa}{\kappa-2} C_{\ms{Vlo}}^2 h(t)^2
		\eqfs
	\end{equation*}
\end{lemma}
\begin{proof}[Proof of \autoref{lmm:peeling}]\label{proof:lmm:peeling}
	For a function $h(t) > 0$, we have
	\begin{align*}
		\Ex*{\frac{\ol{\hat m_t}{m_t}^2}{h(t)^2}} 
		&=
		\int_0^\infty
		2 s \PrOf{\ol{\hat m_t}{m_t} > s h(t)}  \dl s
		\eqfs
	\end{align*}
	By \textsc{VarIneq}, the minimizing property of $\hat m_t$, and \autoref{lmm:split}, we obtain
	\begin{align*}
		C_{\ms{Vlo}}^{-1}\, \ol{\hat m_t}{m_t}^2
		&\leq 
		F_t(\hat m_t,m_t)
		\\&\leq 
		F_t(\hat m_t,m_t) - \hat F_t(\hat m_t,m_t)
		\\&=
		r_t(\hat m_t,m_t) + \hat F_t^\epsilon(\hat m_t,m_t) - F^r_t(\hat m_t,m_t)
		\eqfs
	\end{align*}
	If $\ol{\hat m_t}{m_t} \in[a, b]$ for $0<a<b$, then 
	\begin{align*}
		C_{\ms{Vlo}}^{-1} a^2
		&\leq 
		C_{\ms{Vlo}}^{-1} \, \ol{\hat m_t}{m_t}^2
		\\&\leq
		F_t^\epsilon(\hat m_t, m_t) + r_t(\hat m_t,m_t) - F^r_t(\hat m_t,m_t)
		\\&\leq 
		\sup_{q\in \ball(m_t, b, d)} F_t^\epsilon(q, m_t) + 
		\br{r_t(\hat m_t,m_t) - 
			F^r_t(\hat m_t,m_t)} \ind_{[0,b]}(\ol{\hat m_t}{m_t})
		\\&= U_{t,b}
		\eqfs
	\end{align*}
	Thus, by  Markov's inequality
	\begin{align*}
		\PrOf{\ol{\hat m_t}{m_t} \in[a, b]} \leq \PrOf{a^2 \leq C_{\ms{Vlo}} U_{t,b}} \leq \frac{C_{\ms{Vlo}}^\kappa \Ex{U_{t,b}^\kappa}}{a^{2\kappa}}\eqfs
	\end{align*}
	Let $a_k(s)= 2^ksh(t)$.
	As $\Ex{U_{t,b}^\kappa} \leq b^{\kappa} h(t)^{\kappa}$, we have
	\begin{align*}
		\PrOf{\ol{\hat m_t}{m_t} > s h(t)}
		&\leq 
		\min\brOf{1, \sum_{k=0}^\infty
			\PrOf{\ol{\hat m_t}{m_t} \in[a_k, a_{k+1})}}
		\\&\leq 
		\min\brOf{1, C_{\ms{Vlo}}^\kappa \sum_{k=0}^\infty
			\frac{a_{k+1}^{\kappa} h(t)^{\kappa}}{a_k^{2\kappa}}}
			\eqfs
	\end{align*}
	We obtain
	\begin{align*}
		\frac{a_{k+1}^{\kappa} h(t)^{\kappa}}{a_k^{2\kappa}}
		=
		\frac{\br{2^{k+1}sh(t)}^\kappa h(t)^{\kappa}}{\br{2^ksh(t)}^{2\kappa}}
		=
		\br{\frac{2\cdot 2^ksh(t)h(t)}{2^{2k}s^2h(t)^2}}^\kappa
		=
		\br{2\cdot 2^{-k} s^{-1}}^\kappa
	\end{align*}
	and thus
	\begin{align*}
		 \sum_{k=0}^\infty
		 			\frac{a_{k+1}^{\kappa} h(t)^{\kappa}}{a_k^{2\kappa}}
		=
		2^\kappa s^{-\kappa} \sum_{k=0}^\infty 2^{-k\kappa} 
		=
		\frac{2^\kappa}{1-2^{-\kappa}} s^{-\kappa}
	\end{align*}
	Putting everything together with $c_\kappa = \frac{2^\kappa}{1-2^{-\kappa}} C_{\ms{Vlo}}^\kappa$ yields 
	\begin{align*}
		h(t)^{-2}\Ex*{\ol{\hat m_t}{m_t}^2} 
		&=
		2\int_0^\infty
		s \PrOf{\ol{\hat m_t}{m_t} > s h(t)}  \dl s
		\\&\leq
		2\int_0^\infty
		s \min\brOf{1, c_\kappa s^{-\kappa}}
		\dl s
		\\&=
		\int_0^{c_\kappa^{\frac1\kappa}}
		2s 
		\dl s
		+
		2  c_\kappa \int_{c_\kappa^{\frac1\kappa}}^\infty
		s^{1-\kappa}
		\dl s
		\\&=
		c_\kappa^{\frac2\kappa}
		+ 2 c_\kappa \frac{1}{\kappa-2} \br{c_\kappa^{\frac1\kappa}}^{2-\kappa}
		\\&=
		c_\kappa^{\frac2\kappa} \br{1 + \frac{2}{\kappa-2}}
		\\&\leq
		\frac{4\kappa}{\kappa-2} C_{\ms{Vlo}}^2
		\eqfs
	\end{align*}
\end{proof}%
Using the smoothness assumption, we are able to bound the $r$-term.
\begin{lemma}[Bound on $r$]\label{lmm:trifre:r}
Assume \textsc{SobolevSmoothDensity}. Then
	\begin{align*}
		\Ex{\abs{r_t(\hat m_t,m_t)}^\kappa\ind_{[0,b]}(\ol{\hat m_t}{m_t})} 
		&\leq
		b^\kappa h_N(t)^\kappa C_{\ms{Bom}}^\kappa
		\eqcm
	\end{align*}
	where
	\begin{align*}
		h_N(t) &= \br{ \int \br{\sum_{\ell=N+1}^\infty \xi_\ell(y) \psi_\ell(t)}^2 \mu (\dl y)}^{\frac12}
		\\
		H(q,p) &= \br{\int  \br{\ol yq + \ol yp}^2 \mu (\dl y)}^\frac12 
		\eqfs
	\end{align*}
\end{lemma}
\begin{proof}
It holds
\begin{align*}
	\vartheta_j(q,p) 
	&= 
	\int_0^1 \psi_j(x) F_x(q,p) \dl x
	\\&=
	\int_0^1 \int \psi_j(x) \br{\ol yq^2 - \ol yp^2} \rho(y|x) \dl \mu (y) \dl x
	\\&=
	\int \br{\ol yq^2 - \ol yp^2} \int_0^1 \psi_j(x) \rho(y|x) \dl x \dl \mu (y)
	\\&=
	\int \br{\ol yq^2 - \ol yp^2} \xi(y) \dl \mu (y)
	\eqfs
\end{align*}
Thus, 
\begin{align*}
	r_t(q, p) 
	&= 
	\int \br{\ol yq^2 - \ol yp^2} \sum_{\ell=N+1}^\infty \xi_\ell(y) \psi_\ell(t) \mu (\dl y)
	\\&\leq 
	\br{\int  \br{\ol yq^2 - \ol yp^2}^2 \mu (\dl y)}^\frac12 \br{ \int \br{\sum_{\ell=N+1}^\infty \xi_\ell(y) \psi_\ell(t)}^2 \mu (\dl y)}^{\frac12}
	\\&\leq 
	\ol qp H(q,p) h_N(t)
	\eqfs
\end{align*}
Finally, we obtain
\begin{align*}
	\Ex{\abs{r_t(\hat m_t,m_t)}^\kappa\ind_{[0,b]}(\ol{\hat m_t}{m_t})} 
	&\leq
	b^\kappa h_N(t)^\kappa \Ex{H(\hat m_t,m_t)^\kappa}
	\eqfs
\end{align*}
\end{proof}
Using the previous result, we can also establish a bound on $F^r$.
\begin{lemma}[Bound on $F^r$]\label{lmm:trifre:fr}
	\begin{align*}
	\Ex{F^r_t(\hat m_t,m_t)^\kappa\ind_{[0,b]}(\ol{\hat m_t}{m_t})} 
	\leq c_\kappa \br{N n^{1-2\beta} C_{\ms{SmD}}}^{\kappa} b^\kappa C_{\ms{Bom}}^\kappa
	\end{align*}
	where $c_\kappa \in [1, \infty)$ depends only on $\kappa$.
\end{lemma}
\begin{proof}
We will show that asymptotically $F^r_t(q,p) \lesssim r_t(q, p)$.
Recall
\begin{align*}
	F^r_t(q,p) &= \Psi_N(t)\tr \frac1n \sum_{i=1}^n \Psi_N(x_i) r_{x_i}(q,p)
	\\
	r_t(q,p) &= \sum_{k=N+1}^\infty \vartheta_j(q,p) \psi_j(t)
\end{align*}
and define
\begin{equation*}
r_{n,t}(q,p) = \sum_{\ell=n}^\infty \vartheta_\ell(q,p) \psi_\ell(t)
\end{equation*}
It holds
\begin{align*}
	F^r_t(q,p) 
	&\leq
	\normOf{\Psi_N(t)}	\normof{\frac1n \sum_{i=1}^n \Psi_N(x_i) r_{x_i} (q,p)}
\end{align*}
By \autoref{lmm:bessel} below, to be shown below,
\begin{equation*}
\normOf{\frac1n\sum_{i=1}^n\Psi_N(x_i) r_{x_i} (q,p)}^2 \leq \frac1n \sum_{i=1}^n r_{x_i} (q,p)^2
\end{equation*}
As in the proof of \autoref{lmm:trifre:r}, we have
\begin{align*}
	\abs{r_{n,t}(q,p)}
	&\leq
	\ol qp h_n(t)^\kappa H(q,p)
	\eqcm
\end{align*}
where 
\begin{align*}
	h_{n}(t)^2 = \int \br{\sum_{\ell=n}^\infty \xi_\ell(y) \psi_\ell(t)}^2 \mu (\dl y)
\end{align*}
Thus,
\begin{align*}
	F^r_t(q,p)^2 
	\leq 
	\ol qp^2 H(q,p)^2 \normof{\Psi_N(t)}^2 \frac1n \sum_{i=1}^n h_{n}(x_i)^2
\end{align*}
\begin{equation*}
	\normof{\Psi_N(t)}^2 \leq 2 N
\end{equation*}
As $\xi(y) \in \mc E(\beta, L(y))$, we have $\sum_{k=1}^\infty \xi_j(y)^2 a_j^{-2} \leq L(y)^2$ with $a_{2j+1} = a_{2j} = (2j)^{-\beta}$.
\begin{equation*}
	\sum_{k=n}^\infty a_j^2 \leq c n^{1-2\beta}
	\eqfs
\end{equation*}
Thus,
\begin{align*}
	\frac1n \sum_{i=1}^n \br{\sum_{j=n}^\infty \xi_j(y) \psi_j(x_i)}^2 
	&\leq
	\frac1n \sum_{i=1}^n \sum_{j=n}^\infty a_j^{-2} \xi_j(y)^2 \sum_{j=n}^\infty a_j^2 \psi_j(x_i)^2 
	\\&\leq
	2 \sum_{j=n}^\infty a_j^{-2} \xi_j(y)^2 \sum_{j=n}^\infty a_j^2
	\\&\leq
	c_0 L(y)^2 n^{1-2\beta}
	\eqfs
\end{align*}
We obtain
\begin{align*}
	\frac1n \sum_{i=1}^n h_{n}(x_i)^2 
	&\leq
	\frac1n \sum_{i=1}^n \int \br{\sum_{\ell=n}^\infty \xi_\ell(y) \psi_\ell(x_i)}^2 \mu (\dl y)
	\\&\leq
	c_0 n^{1-2\beta} \int L(y)^2 \mu (\dl y)
\end{align*}
and can bound
\begin{align*}
	F^r_t(q,p)^2 
	\leq 
	2 c_0 \ol qp^2 H(q,p)^2 N n^{1-2\beta} \int L(y)^2 \mu (\dl y)
	\eqfs
\end{align*}
Finally, the inequalities above yield
\begin{align*}
	\Ex{F^r_t(\hat m_t,m_t)^\kappa\ind_{[0,b]}(\ol{\hat m_t}{m_t})} 
	\leq \br{2 c_0 N n^{1-2\beta} \int L(y)^2 \mu (\dl y)}^{\frac\kappa2} b^\kappa \Ex{H(\hat m_t,m_t)^\kappa} 
	\eqfs
\end{align*}
\end{proof}
We still have to prove following lemma, which was used in the previous proof.
\begin{lemma}\label{lmm:bessel}
Let $f \colon [0,1]\to \R$ be any function and $N < n$.
Then
\begin{equation*}
	\normOf{\frac1n\sum_{i=1}^n\Psi_N(x_i) f(x_i)}^2 \leq \frac1n \sum_{i=1}^n f(x_i)^2
\end{equation*}
\end{lemma}
\begin{proof}[Proof of \autoref{lmm:bessel}]\label{proof:lmm:bessel}
	Let $b_\ell = \frac1n \sum_{i=1}^n \psi_\ell(x_i) f(x_i)$ and $s(t) = f(t) - \sum_{\ell=1}^N b_\ell \psi_\ell(t)$.
	Then 
	\begin{align*}
		\frac1n \sum_{i=1}^n s(x_i) \psi_j(x_i)
		&=
		\frac1n \sum_{i=1}^n \br{f(x_i) - \sum_{\ell=1}^N b_\ell \psi_\ell(x_i)} \psi_j(x_i)
		\\&=
		\frac1n \sum_{i=1}^n f(x_i) \psi_j(x_i) - \sum_{\ell=1}^N b_\ell \frac1n \sum_{i=1}^n  \psi_\ell(x_i) \psi_j(x_i)
		\\&=
		b_j - b_j
		\\&=
		0
	\end{align*}
	and thus
	\begin{align*}
		\frac1n \sum_{i=1}^n f(x_i)^2 
		&= 
		\frac1n \sum_{i=1}^n \br{s(x_i) + \sum_{\ell=1}^Nb_\ell \psi_\ell(x_i)}^2
		\\&= 
		\frac1n \sum_{i=1}^n \br{s(x_i)^2 + s(x_i)\sum_{\ell=1}^Nb_\ell \psi_\ell(x_i) + \sum_{\ell,j=1}^Nb_\ell b_j \psi_\ell(x_i)\psi_j(x_i)}
		\\&= 
		\frac1n \sum_{i=1}^n s(x_i)^2 + \sum_{\ell=1}^Nb_\ell \frac1n \sum_{i=1}^n s(x_i) \psi_\ell(x_i) + \sum_{\ell,j=1}^Nb_\ell b_j \frac1n \sum_{i=1}^n \psi_\ell(x_i)\psi_j(x_i)
		\\&= 
		\frac1n \sum_{i=1}^n s(x_i)^2 + \sum_{\ell}^Nb_\ell^2
		\eqfs
	\end{align*}
	Furthermore,
	\begin{align*}
		\normOf{\frac1n\sum_{i=1}^n\Psi_N(x_i) f(x_i)}^2
		&=
		\sum_{\ell=1}^N \br{\psi_\ell(x_i) f(x_i)}^2
		\\&=
		\sum_{\ell=1}^N b_\ell^2
	\end{align*}
	As $\frac1n \sum_{i=1}^n s(x_i)^2 \geq 0$ we have proved the claim.
\end{proof}
Next, we tackle the variance term.
\begin{lemma}[Bound on $F^\varepsilon$]\label{lmm:fepsilon}
	Assume \textsc{Moment}, \textsc{Entropy}.
	Then
	\begin{equation*}
		\Ex*{\sup_{q\in\mc B}F^\varepsilon_t(q,p)^\kappa} \leq c_\kappa C_{\ms{Mom}}^\kappa n^{-\frac\kappa2} C_{\ms{Ent}}^\kappa b^\kappa \br{\Psi_N(t)\tr\Psi_N(t)}^\frac\kappa2 
		\eqfs
	\end{equation*}
\end{lemma}
\begin{proof}[Proof of \autoref{lmm:fepsilon}]\label{proof:lmm:fepsilon}
	Recall $F^\varepsilon_t(q,p) = \Psi_N(t)\tr \frac1n \sum_{i=1}^n \Psi_N(x_i) \varepsilon_{x_i}(y_i,q,p)$.
	Define $\alpha_i = \frac1n \Psi_N(t)\tr \Psi_N(x_i)$, $\varepsilon_i(q,p) = \varepsilon_{x_i}(y_i,q,p)$.
	Then
	\begin{equation*}
		F^\varepsilon_t(q,p) = \sum_{i=1}^n \alpha_i \varepsilon_i(q,p)\eqcm
	\end{equation*}
	where $\varepsilon_1, \dots,\varepsilon_n$ are independent and $\Ex{\varepsilon_i(q,p)} = 0$.
	We want to apply \autoref{thm:empproc} with $Z_i(q)- Z_i(p) =  \alpha_i \varepsilon_i(q,p)$ and $A_i = \alpha_i \mf a( y_i, y_i\pr)$. We need to show
	\begin{equation*}
		\abs{Z_i(q)-Z_i(p)-Z_i\pr(q)+ Z_i\pr(p)} \leq A_i \, \ol qp
	\end{equation*} 
	to obtain
	\begin{equation*}
		\Ex*{\sup_{q\in\mc B} \abs{\sum_{i=1}^nZ_i(q)}^\kappa} \leq C\, \Ex*{\abs{A}^\kappa} \, 
		\gamma_2(\mc B, d)^\kappa
		\eqfs
	\end{equation*}
	Using the quadruple property, we obtain
	\begin{align*}
		\varepsilon_i(q,p) - \varepsilon\pr_i(q,p)
		&= 
		\br{F(q,p,x_i)-\br{\ol{y_i}q^2 - \ol{y_i}p^2}} - \br{F(q,p,x_i)-\br{\ol{y_i}q^2 - \ol{y_i}p^2}}
		\\&\leq
		\mf a( y_i, y_i\pr) \,\ol qp  
		\eqfs
	\end{align*}
	Thus, \autoref{thm:empproc} yields
	\begin{equation*}
		\Ex*{\sup_{q\in\mc B}F^\varepsilon_t(q,p)^\kappa} \leq C \gamma_2(\mc B, d)^\kappa \Ex*{\br{\sum_{i=1}^n \alpha_i^2 \mf a(y_i, y_i\pr)^2}^\frac\kappa2}
		\eqfs
	\end{equation*}
	Let $a_i = \frac{\alpha_i^2}{\sum_{i=1}^n \alpha_i^2}$.
	\begin{align*}
		\Ex*{\br{\sum_{i=1}^n \alpha_i^2 \mf a(y_i, y_i\pr)^2}^\frac\kappa2}
		&=
		\br{\sum_{i=1}^n \alpha_i^2}^\frac\kappa2 \Ex*{\br{\sum_{i=1}^n a_i  \mf a(y_i, y_i\pr)^2}^\frac\kappa2}
		\\&\leq
		\br{\sum_{i=1}^n \alpha_i^2}^\frac\kappa2 \Ex*{\sum_{i=1}^n a_i  \mf a(y_i, y_i\pr)^\kappa}
		\\&=
		\br{\sum_{i=1}^n \alpha_i^2}^\frac\kappa2 \sum_{i=1}^n a_i  \Ex*{\mf a(y_i, y_i\pr)^\kappa}
		\\&\leq
		\br{\sum_{i=1}^n \alpha_i^2}^\frac\kappa2 \sup_{t}\Ex*{\mf a(Y_t, Y_t\pr)^\kappa}
		\eqfs
	\end{align*}
	As $\mf a$ is a pseudo-metric, we have, using \textsc{Moment},
	\begin{align*}
		\Ex*{\mf a(Y_t, Y_t\pr)^\kappa} \leq 2^\kappa C_{\ms{Mom}}^\kappa
		\eqfs
	\end{align*}
	Furthermore, it holds 
	\begin{equation*}
		\sum_{i=1}^n \alpha_i^2 = \frac1{n^2} \sum_{i=1}^n \Psi_N(t)\tr\Psi_N(x_i)\Psi_N(x_i)\tr\Psi_N(t) = \frac1n \Psi_N(t)\tr\Psi_N(t)
		\eqfs
	\end{equation*}
	Together we get
	\begin{equation*}
		\Ex*{\sup_{q\in\mc B}F^\varepsilon_t(q,p)^\kappa} \leq c_\kappa C_{\ms{Mom}}^\kappa n^{-\frac\kappa2} \gamma_2(\mc B, d)^\kappa \br{\Psi_N(t)\tr\Psi_N(t)}^\frac\kappa2 
		\eqfs
	\end{equation*}
\end{proof}
Finally, we put the previous results together to proof our main theorem of this section.
\begin{lemma}\label{lmm:kappa_moment}
	There is a constant $c_\kappa>0$ depending only on $\kappa$ such that
	\begin{equation*}
		h(t)^\kappa 
		\leq  
		c_\kappa \br{
			h_N(t)^\kappa C_{\ms{Bom}}^\kappa + 
			\br{N n^{1-2\beta} C_{\ms{SmD}}}^{\kappa} C_{\ms{Bom}}^\kappa +
			C_{\ms{Mom}}^\kappa n^{-\frac\kappa2} C_{\ms{Ent}}^\kappa \normof{\Psi_N(t)}^\kappa 
		}
	\end{equation*}
\end{lemma}
\begin{proof}[Proof of \autoref{lmm:kappa_moment}]\label{proof:lmm:kappa_moment}
	\autoref{lmm:trifre:r}, \autoref{lmm:trifre:fr}, and \autoref{lmm:fepsilon}.
\end{proof}

\begin{lemma}\label{lmm:h}
	For the function $h_N$ defined in \autoref{lmm:trifre:r}, it holds
	\begin{align*}
		\int_0^1 h_N(t)^2 \dl t &\leq c\beta N^{-2\beta} C_{\ms{SmD}}^2
		\eqfs
	\end{align*}
\end{lemma}
\begin{proof}[Proof of \autoref{lmm:h}]\label{proof:lmm:h}
	We use Fubini's theorem and the weights $a_{2j+1} = a_{2j} = (2j)^{-\beta}$ from the definition of the ellipsoid $\mc E(\beta, L)$ and obtain
	\begin{align*}
		\int_0^1 h_N(t)^2 \dl t 
		&= 
		\int \int_0^1 \br{\sum_{\ell=N+1}^\infty \xi_\ell(y) \psi_\ell(t)}^2 \dl t  \dl \mu(y)
		\\&=
		\int_0^1 \int \br{\sum_{\ell=N+1}^\infty \xi_\ell(y) \psi_\ell(t)}^2 \dl \mu(y) \dl t 
		\\&=
		\int \sum_{\ell=N+1}^\infty \xi_\ell(y)^2 \dl \mu(y)
		\\&\leq 
		\int a_{N+1}^2 \sum_{\ell=N+1}^\infty \xi_\ell(y)^2 a_\ell^{-2} \dl \mu(y)
		\\&\leq 
		c\beta N^{-2\beta} \int L(y)^2 \dl \mu(y)
		\eqfs
	\end{align*}
\end{proof}
\begin{proof}[Proof of \autoref{thm:trifre}]\label{proof:thm:trifre}
	We apply \autoref{lmm:peeling}, \autoref{lmm:kappa_moment}, and \autoref{lmm:h} together with
	\begin{align*}
		\int_0^1 \normOf{\Psi_N(t)}^2 \dl t = \int_0^1 \sum_{\ell=1}^N\psi_\ell(t)^2 \dl t = N
	\end{align*}
	to finally obtain
	\begin{align*}
		\int_0^1\Ex*{\ol{\hat m_t}{m_t}^2} \dl t
		&\leq
		c_{\kappa} C_{\ms{Vlo}}^2 \int_0^1 h(t)^2 \dl t
		\\&\leq
		c_{\kappa} C_{\ms{Vlo}}^2
		\br{
			 C_{\ms{Bom}}^2 \int_0^1 h_N(t)^2 \dl t+ 
			N n^{1-2\beta} C_{\ms{SmD}}^2 C_{\ms{Bom}}^2 +
			C_{\ms{Mom}}^2 n^{-1} 	C_{\ms{Ent}}^2 \int_0^1 \normof{\Psi_N(t)}^2 \dl t
		}
		\\&\leq
		c_{\kappa,\beta} C_{\ms{Vlo}}^2
		\br{
 			C_{\ms{Bom}}^2 C_{\ms{SmD}}^2 N^{-2\beta} + 
			C_{\ms{SmD}}^2 C_{\ms{Bom}}^2 N n^{1-2\beta} +
			C_{\ms{Mom}}^2 C_{\ms{Ent}}^2 \frac Nn
		}
		\eqfs
	\end{align*}
\end{proof}
\subsubsection{Main Theorems}
We use \autoref{thm:trifre} to prove the two main theorems concerning \texttt{OrtFre}.
Recall $H(q,p) = \br{\int  \br{\ol yq + \ol yp}^2 \mu (\dl y)}^\frac12$.
\begin{proof}[Proof of \autoref{cor:trifre:bounded}]
	If $\diam(\mc Q, d) < \infty$, then 
	\begin{align*}
	H(q,p) 
	\leq 
	\br{\int  \br{2\,\diam(\mc Q, d)}^2 \mu (\dl y)}^\frac12 
	= 2\,\diam(\mc Q, d)
	\eqfs
	\end{align*}
	Thus, we can choose $C_{\ms{Bom}} = 2\,\diam(\mc Q, d)$.
	Using the triangle inequality we get $ \ol yq^2 - \ol yp^2 - \ol zq^2 + \ol zp^2 \leq 4 \ol qp \diam(\mc Q, d)$. Thus, $\mf a(y, z) \leq 4  \diam(\mc Q, d)$ and we can choose $C_{\ms{Mom}} = 4  \diam(\mc Q, d)$.
\end{proof}
\begin{proposition}\label{prop:trifre:BomBound}
	Let $\mc Q$ be a Hadamard space. Assume \textsc{SobolevSmoothDensity} and \textsc{Moment}.
	To fulfill $\Ex{H(\hat m_t,m_t)^\kappa}^{\frac1\kappa} \leq C_{\ms{Bom}}$, we can choose
	\begin{equation*}
		C_{\ms{Bom}} = c_\kappa C_{\ms{Len}}C_{\ms{Mom}}C_{\ms{Int}}\br{1 + \log(N) + \frac{N^2}{n}}
	\end{equation*}
	where $c_\kappa > 0$ depends only on $\kappa$.
\end{proposition}
This proposition is proven in two steps: \autoref{lmm:trifre:weak_moment_bound} and \autoref{lmm:abs_weights_bound}.
Let $w_i = \frac1n \abs{\Psi_N(t)\tr \Psi_N(x_i)}$ and $W = \sum_{i=1}^n \abs{w_i}$.
\begin{lemma}\label{lmm:trifre:weak_moment_bound}
	There is a constant $c_\kappa \in[1,\infty)$ depending only on $\kappa$ such that
	\begin{equation*}
		\Ex{H(\hat m_t,m_t)^\kappa}^{\frac1\kappa} 
		\leq
		c_\kappa \br{W \br{C_{\ms{Len}} + C_{\ms{Mom}}} + C_{\ms{Int}} +  C_{\ms{Len}}}
		\eqfs
	\end{equation*}
\end{lemma}
\begin{proof}[Proof of \autoref{lmm:trifre:weak_moment_bound}]
	Using the triangle inequality
	\begin{align*}
		H(q,p)^2 
		&= 
		\int  \br{\ol yq + \ol yp}^2 \mu (\dl y)
		\\&\leq 
		\int  \br{\ol qp + 2\ol yp}^2 \mu (\dl y)
		\\&\leq 
		2\int  \ol qp^2 + 4\ol yp^2 \mu (\dl y)
		\\&\leq 
		2 \ol qp^2 + 8 \int \ol yp^2 \mu (\dl y)
	\end{align*}
	as $\mu$ is a probability measure. Using bounds in \textsc{SobolevSmoothDensity}, we get
	\begin{align*}
		\Ex{H(\hat m_t,m_t)^\kappa}^{\frac1\kappa} 
		&\leq 
		\Ex*{\br{2 \ol {\hat m_t}{m_t}^2 + 8 \int \ol y{m_t}^2 \mu (\dl y)}^{\frac\kappa2}}^{\frac1\kappa}
		\\&\leq 
		c_\kappa \br{\Ex*{\ol {\hat m_t}{m_t}^{\kappa}}^\frac1\kappa +  \br{\int \ol y{m_0}^2 \mu (\dl y)}^{\frac12} +  \ol{m_t}{m_0}}
		\\&\leq 
		c_\kappa \br{\Ex*{\ol {\hat m_t}{m_t}^{\kappa}}^\frac1\kappa +  C_{\ms{Int}} +  C_{\ms{Len}}}
		\eqfs
	\end{align*}
	Next, we will bound $\Ex{\ol{m_t}{\hat m_t}^\kappa}$.
	First, by \textsc{VarIneq} and the minimizing property of $\hat m_t$,
	\begin{align*}
		\ol{m_t}{\hat m_t}^2
		&\leq 
		F_t(\hat m_t, m_t)
		\\&\leq 
		F_t(\hat m_t, m_t) - \hat F_t(\hat m_t, m_t)
		\\&\leq 
		2 \sum_{i=1}^n \abs{w_i}\, \ol{\hat m_t}{m_t} \,\Ex{d(Y_t, y_i) | y_i}
	\end{align*}
	Thus,
	\begin{equation*}
		\ol{m_t}{\hat m_t} \leq 
		2 \sum_{i=1}^n \abs{w_i} \Ex{d(Y_t, y_i) | y_i}
	\end{equation*}
	With Jensen's inequality
	\begin{align*}
		\Ex{\ol{m_t}{\hat m_t}^\kappa}
		&\leq
		c_\kappa \Ex*{\br{\sum_{i=1}^n \abs{w_i} \Ex{d(Y_t, y_i) | y_i}}^\kappa}
		\\&=
		c_\kappa W^\kappa \Ex*{\br{\sum_{i=1}^n \frac{\abs{w_i}}{W} \Ex{d(Y_t, y_i) | y_i}}^\kappa}
		\\&\leq
		c_\kappa W^\kappa \sum_{i=1}^n \frac{\abs{w_i}}{W}  \Ex*{\Ex{d(Y_t, y_i) | y_i}^\kappa}
		\\&\leq
		c_\kappa W^\kappa \sum_{i=1}^n \frac{\abs{w_i}}{W}  \Ex{d(Y_t, y_i)^\kappa}
		\\&\leq
		c_\kappa W^\kappa \sup_{s,t\in[0,1]}\Ex*{d(Y_t, Y\pr_s)^\kappa}
		\eqfs
	\end{align*}
	As $d$ is a metric,
	\begin{align*}
		\Ex*{d(Y_t, Y\pr_s)^\kappa}
		&\leq
		\Ex*{\br{d(Y_t, m_t) + d(m_t, m_s) + d(m_s, Y\pr_s)}^\kappa}
		\\&\leq
		3^\kappa \br{2\sup_{t\in[0,1]}\Ex*{d(Y_t, m_t)^\kappa} + d(m_t, m_s)^\kappa}
		\\&\leq
		c_\kappa \br{C_{\ms {Mom}}^\kappa +C_{\ms{Len}}^\kappa}
		\eqfs
	\end{align*}
\end{proof}
\begin{lemma}\label{lmm:abs_weights_bound}
	There is an universal constant $c \in (0,\infty)$ such that
	\begin{align*}
		W \leq c \br{1 + \log(N) + \frac{N^2}{n}}
		\eqfs
	\end{align*} 
\end{lemma}
\begin{proof}[Proof of \autoref{lmm:abs_weights_bound}]\label{proof:lmm:abs_weights_bound}
	Let $g_t(s) = \abs{\sum_{\ell=1}^N\psi_\ell(t) \psi_\ell(s)}$.  Then
	\begin{align*}
	W 
	= 
	\sum_{i=1}^n \abs{w_i}
	=
	\frac1n \sum_{i=1}^n \abs{\Psi_N(t)\tr \Psi_N(x_i)}
	=
	\frac1n \sum_{i=1}^n g_t(x_i) 
	\eqfs
	\end{align*}
	By the standard comparison between an integral of a Lipschitz--continuous function an the corresponding Riemann sum, we obtain 
	\begin{align*}
	\abs{\int_0^1 g_t(s)\dl s - \frac1n \sum_{i=1}^n g_t(x_i)}
	&\leq 
	\sup_{s\in[0,1]} \frac{\abs{g_t\pr(s)}}{n}
	\\&\leq 
	4\pi \frac{N^2}{n}
	\eqfs
	\end{align*}
	This bound is quite rough and could be improved. But we will choose $N_n \leq n^{\frac13}$ and thus $\frac{N_n^2}{n} \to 0$.
	For $x\in \R$ denote $[x]$ the fractional part of $x$, i.e., the number $[x]\in[0,1)$ that fulfills $[x] = x-k$ for a $k\in \mb Z$.
	For $\ell \geq 2$,
	\begin{equation*}
		\psi_\ell(t) \psi_\ell(s) = \frac12\br{(-1)^\ell\cos(2\pi\ell[t+s]) + \cos(2\pi\ell[t-s])}\eqfs
	\end{equation*}
	The function $(s,t) \mapsto \sum_{\ell=1}^N\psi_\ell(t) \psi_\ell(s)$ only depends on $[s+t]$ and $[s-t]$.
	When integrating $s$ from 0 to 1, $[s+t]$ and  $[s-t]$ run through every value in $[0,1)$.
	Thus 
	\begin{align*}
	&\sup_{t\in [0,1]} \int_0^1 \abs{1+\sum_{\ell=2}^N\psi_\ell(t) \psi_\ell(s)} \dl s
	\\&=
	\sup_{t\in [0,1]} \int_0^1 \abs{1+ \frac12 \sum_{\ell=2}^N\br{(-1)^\ell\cos(2\pi\ell[t+s]) + \cos(2\pi\ell[t-s])}} \dl s
	\\&\leq
	1 + \frac12 \sup_{t\in [0,1]} \int_0^1 \abs{\sum_{\ell=2}^N\br{(-1)^\ell\cos(2\pi\ell[t+s])}} \dl s
	\\&\qquad+ \frac12 \sup_{t\in [0,1]} \int_0^1 \abs{ \sum_{\ell=2}^N \cos(2\pi\ell[t-s]) }\dl s
	\\&=
	1 + \frac12 \int_0^1 \abs{\sum_{\ell=2}^N (-1)^\ell \cos(2\pi\ell s)}\dl s + \frac12 \int_0^1 \abs{ \sum_{\ell=2}^N  \cos(2\pi\ell s) }\dl s
	\eqfs
	\end{align*}
	Lagrange's trigonometric identities state
	\begin{align*}
		2 \sum_{\ell=1}^L \cos(\ell x) &= -1 + \frac{\sin\brOf{(L+\frac12)x}}{\sin\brOf{\frac x2}}\eqcm
		\\
		2 \sum_{\ell=1}^L (-1)^\ell\cos(\ell x) &= -1 + \frac{(-1)^{L+1}\sin\brOf{(L+\frac12)x}}{-\sin\brOf{\frac x2}}\eqfs
	\end{align*}
	Thus, we have to bound the integral
	\begin{align*}
	\int_0^1 \abs{\frac{\sin\brOf{(2L+1)\pi s}}{\sin\brOf{\pi s}}} \dl s
	\eqfs
	\end{align*}
	It holds
	$\abs{\sin(\pi x)} \geq \frac12 \pi\min(x, 1-x)$ for $x \in [0,1]$. Let $a = k\pi$ for $k \in \N$. Then
	\begin{align*}
	\int_0^1 \abs{\frac{\sin\brOf{a s}}{\sin\brOf{\pi s}}} \dl s
	&\leq
	\frac2\pi\int_0^1 \frac{\abs{\sin\brOf{a s}}}{\min(s, 1-s)} \dl s
	\\&=
	\frac4\pi\int_0^{\frac12} \frac{\abs{\sin\brOf{a s}}}{s} \dl s
	\\&=
	\frac{4}\pi\int_0^{\frac12a} \frac{\abs{\sin\brOf{t}}}{t} \dl t
	\eqfs
	\end{align*}
	We bound this integral as follows,
	\begin{align*}
	\int_0^{\frac12k\pi} \frac{\abs{\sin\brOf{t}}}{t} \dl t
	&=
	\int_0^{\pi} \frac{\abs{\sin\brOf{t}}}{t} \dl t
	+
	\int_\pi^{\frac12k\pi} \frac{\abs{\sin\brOf{t}}}{t} \dl t
	\\&\leq
	\int_0^{\pi} \frac{\sin\brOf{t}}{t} \dl t
	+
	\int_\pi^{\frac12k\pi} \frac{1}{t} \dl t
	\\&\leq
	2 + \log(\frac12k\pi)-\log(\pi)
	\\&=
	2 + \log(\frac12k)
	\eqfs
	\end{align*}
	Thus, we obtain
	\begin{equation*}
	\int_0^1 \abs{\frac{\sin\brOf{2 k \pi s}}{\sin\brOf{\pi s}}} \dl s
	\leq
	\frac{8}\pi + \frac{4}\pi\log(\frac12k)
	\eqcm
	\end{equation*}
	which yields
	\begin{align*}
	\sup_{t\in [0,1]} \int_0^1 \abs{1+\sum_{\ell=2}^N\psi_\ell(t) \psi_\ell(s)} \dl s
	&\leq c_0 + c_1 \log(N)
	\eqfs
	\end{align*}
\end{proof}
\begin{proof}[Proof of \autoref{cor:trifre:hadamard}]
\textsc{VarIneq} holds in Hadamard spaces with $C_{\ms{Vlo}} = 1$. We bound $\Ex{H(\hat m_t,m_t)^\kappa}^{\frac1\kappa} \leq C_{\ms{Bom}}$ using
\begin{equation*}
	C_{\ms{Bom}} = c_\kappa C_{\ms{Len}}C_{\ms{Mom}}C_{\ms{Int}}\br{1 + \log(N) + \frac{N^2}{n}}
	\eqcm
\end{equation*}
see \autoref{prop:trifre:BomBound}. As $N\leq c\sqrt{n}$ the term $\frac{N^2}{n}$ can be bounded by a constant.
\end{proof}
\subsection{\texttt{LocGeo}}
\subsubsection{A General Result}
We prove a general theorem that implies the main theorems concerning \texttt{LocGeo}. 

Recall the definitions needed to construct the \texttt{LocGeo}-estimator:
Let $h \geq \frac2n$, $K \colon \R \to\R$. For $t\in[0,1]$, define $w_h(t,x) := \frac1h K(\frac{x-t}{h})$ and  $w_{i,t} = w_h(t,x_i) (\sum_{j=1}^n w_h(t,x_j))^{-1}$.
We will show a theorem with a more general notion of parameterized curves than those induced by an exponential map. To this end, let $\Theta$ be a set with subset $\Theta_h \subset \Theta$. Let  $g \colon \R \times \Theta \to \mc Q$. 
Let $\hat \theta_{t,h} \in \argmin_{\theta\in\Theta_h} \sum_{i=1}^n w_{i,t}\,d(y_i, g(x_i-t, \theta))^2$ and $\hat m_t = g(0, \hat \theta_{t,h})$.

The distance $d$ induces following two distances on $\Theta$, which we will make use of later.
\begin{align*}
	D_h^2(\theta, \tilde\theta) &:= \int_{-\frac12}^{\frac12} d\brOf{g(xh, \theta), g(xh, \tilde\theta)}^2 \dl x\eqcm\\
	b_h(\theta, \tilde\theta) &:= \sup_{x\in[-1,1]} d\brOf{g(xh, \theta), g(xh, \tilde\theta)}\eqfs
\end{align*}
\begin{assumptions}\label{ass:locgeo:app}\mbox{ }
\begin{itemize}
\item 
	\textsc{VarIneq}: There is $C_{\ms{Vlo}}\in[1,\infty)$ such that $C_{\ms{Vlo}}^{-1}d(q,m_t)^2 \leq \Ex{d(Y_t, q)^2 - d(Y_t, m_t)^2}$ for all $q\in\mc Q$ and $t\in[0,1]$.
\item
	\textsc{EntropyGeod}: 
	There are $C_{\ms{EnG}} \in [1,\infty)$ and $\alpha \in [1,2)$ such that 
	$\gamma_2(\mc B, \mf b_h) \leq C_{\ms{EnG}}\max(\diam(\mc B, \mf b_h), \diam(\mc B, \mf b_h)^\alpha)$ for all $\mc B \subset \Theta_h$. 
\item 
	\textsc{MomentA}: There is $\kappa > \frac{2}{2-\alpha}$ and $C_{\ms{MoA}}\in[1,\infty)$ such that $\Ex{\mf a(Y_t, m_t)^\kappa}^{\frac1\kappa} \leq C_{\ms{MoA}}$ for all $t\in[0,1]$.
\item \textsc{Kernel}:
	There are $C_{\ms{Kmi}}, C_{\ms{Kma}} \in [1,\infty)$ such that 
	\begin{equation*}
		C_{\ms{Kmi}}^{-1} \ind_{[-\frac12,\frac12]}(x) \leq K(x) \leq C_{\ms{Kma}} \ind_{[-1,1]}(x)
	\end{equation*}
	for all $x\in\R$.
\item
	\textsc{HölderSmoothEx}:
	Let $\beta > 0$.
	There is $C_{\ms{Smo}} \in [1, \infty)$ such that for all $t\in[0,1]$, there is $\theta_{t} \in \Theta_h$ such that $\Ex{d(Y_x, g(x-t, \theta_{t}))^2 - d(Y_x, m_x)^2} \leq C_{\ms{Smo}}^2\abs{x-t}^{2\beta}$ for all $x\in[0,1]$.
\item \textsc{Lipschitz}:
	There is $ C_{\ms{Lip}} \in [1,\infty)$ such that
	\begin{equation*}
		d(g(xh,\theta), g(yh,\theta)) \leq C_{\ms{Lip}} \abs{x-y}
	\end{equation*}
	for all $x,y\in[-\frac12, \frac12]$ and $\theta \in \Theta_h$.
\item \textsc{IntBoundsSup}:
	There is $C_{\ms{IBS}} \in [1,\infty)$ such that
	\begin{equation*}
		\mf b_h(\theta, \tilde\theta)^2 \leq C_{\ms{IBS}}^2 D^2_h(\theta, \tilde\theta)
	\end{equation*}
	for all $\theta, \tilde\theta \in \Theta_h$.
\end{itemize}
\end{assumptions}

\begin{theorem}[\texttt{LocGeo} General]\label{thm:locgeo}
	Assume \textsc{VarIneq}, \textsc{MomentA}, \textsc{HölderSmoothEx}, \textsc{Kernel}, \textsc{EntropyGeod}, \textsc{Lipschitz}, and \textsc{IntBoundsSup}.
	Then
	\begin{equation*}
		\Ex*{D_h^2(\hat\theta_{t,h}, \theta_{t})} 
		\leq
 		C_1 h^{2\beta} + C_2 (nh)^{-1} + C_3 (nh)^{-2} 
		\eqcm
	\end{equation*}
	for all $t\in[0,1]$, where 
	\begin{align*}
		C_1 &= c_\kappa C_{\ms{Kmi}}C_{\ms{Kma}} C_{\ms{Vlo}} C_{\ms{Smo}}^2\eqcm
		\\
		C_2 &= c_{\alpha,\kappa}  \br{C_{\ms{IBS}}^2 C_{\ms{Kmi}}^3C_{\ms{Kma}}^3
				 		C_{\ms{MoA}}^2 C_{\ms{EnG}}^2 C_{\ms{Vlo}}^2}^{\frac{2}{2-\alpha}}\eqcm
		\\
		C_3 &= c_{\alpha,\kappa} \br{C_{\ms{Lip}}C_{\ms{IBS}}}^{\frac{2}{2-\alpha}}\eqfs
	\end{align*}
\end{theorem}
We first find a general bound on $D_h^2(\theta, \tilde\theta)$ in which the integral is replaced by a sum (\autoref{lmm:locgeo:DhU}). Then \autoref{lmm:locgeo:Ubounds} shows how the resulting terms can further be bounded when applied to $\hat\theta_{t,h}$ and $\theta_{t}$ using the conditions on the kernel and the smoothness assumption. In particular, the error term has parts that can be described as bias and variance parts and the bias terms are bounded here. In \autoref{lmm:locgeo:var}, we use chaining to bound the variance term. Thereafter these results are put together to prove \autoref{thm:locgeo}.

\subsubsection{Proof of the General Result}
For $\theta\in\Theta$, define 
\begin{align*}
	U_{t}(\theta) &:= \sum_{i=1}^n w_{i,t} d\brOf{g(x_i-t, \theta), m_{x_i}}^2\eqfs
\end{align*}
\begin{lemma}\label{lmm:locgeo:DhU}
	Assume \textsc{Kernel} and \textsc{Lipschitz}.
	Let $\theta, \tilde\theta\in\Theta_h$. Then
	\begin{equation*}
		D_h^2(\theta, \tilde\theta) \leq  c C_{\ms{Kmi}}C_{\ms{Kma}} \br{U_t(\theta) + U_t(\tilde\theta)} +  c C_{\ms{Lip}}\mf b_h(\theta, \tilde\theta) (nh)^{-1}
		\eqfs
	\end{equation*}
\end{lemma}
\begin{proof} 
	\textsc{Kernel} implies 
	\begin{align*}
		w_{i,t} \geq \frac{C_{\ms{Kmi}}^{-1}}{C_{\ms{Kma}} \#I_{t,h}} \ind_{[-\frac12,\frac12]}\brOf{\frac{x_i-t}{h}}
		\eqcm
	\end{align*}
	where $I_{t,h} = \cb{i \in \cb{1,\dots,n} \colon t-h \leq x_i \leq t + h}$.
	We bound the difference between the Riemann sum and its corresponding integral using \autoref{lmm:locgeo:lipschitz} with \textsc{Lipschitz}, which shows that the function $x \mapsto d(g(xh, \theta), g(xh, \tilde\theta))^2$ is Lipschitz continuous on $[-\frac12,\frac12]$ with constant $L := c C_{\ms{Lip}} \mf b_h(\theta, \tilde\theta)$. Thus, we obtain
	\begin{align*}
		\abs{\frac1{\#I_{t,\frac h2}}\sum_{i\in I_{t,\frac h2}} \old{g(x_i - t, \theta)}{g(x_i - t, \tilde\theta)}^2 - \int_{-\frac12}^{\frac12} \old{g(xh, \theta)}{g(xh, \tilde\theta)}^2 \dl x} 
		\leq\frac{L}{\#I_{t,\frac h2}}
		\eqfs
	\end{align*}
	Hence,
	\begin{align*}
		\sum_{i=1}^n w_{i,t} \old{g(x_i-t, \theta)}{g(x_i-t, \tilde\theta)}^2 
		&\geq 
		\frac{C_{\ms{Kmi}}^{-1}}{C_{\ms{Kma}} \#I_{t,h}} 
		\sum_{i\in I_{t,\frac h2}} \old{g(x_i - t, \theta)}{g(x_i - t, \tilde\theta)}^2 
		\\&\geq 
		\frac{C_{\ms{Kmi}}^{-1}\#I_{t,\frac h2}}{C_{\ms{Kma}} \#I_{t,h}} 
		\br{\int_{-\frac12}^{\frac12} \old{g(xh, \theta)}{g(xh, \tilde\theta)}^2 \dl x - \frac{L}{\#I_{t,\frac h2}}}
		\eqfs
	\end{align*}
	As $h \geq \frac 2n$, we obtain
	\begin{equation*}
		\sum_{i=1}^n w_{i,t} \old{g(x_i-t, \theta)}{g(x_i-t, \tilde\theta)}^2 
		\geq 
		\frac{C_{\ms{Kmi}}^{-1}}{6C_{\ms{Kma}}} 
		\br{D_h^2(\theta, \tilde\theta) - \frac{2 L}{nh}}
		\eqfs
	\end{equation*}
	Using the triangle inequality, we can further bound
	\begin{align*}
		\sum_{i=1}^n w_{i,t} \old{g(x_i-t, \theta)}{g(x_i-t, \tilde\theta)}^2
		&\leq 
		2\sum_{i=1}^n w_{i,t} \br{d(g(x_i-t, \theta), m_{x_i})^2 + d(m_{x_i}, g(x_i-t, \tilde\theta))^2}
		\\&=
		2\br{U_t(\theta) + U_t(\tilde\theta)}
		\eqfs
	\end{align*}
	Thus, we arrive at
	\begin{equation*}
	2\br{U_t(\theta) + U_t(\tilde\theta)}
			\geq 
			\frac{C_{\ms{Kmi}}^{-1}}{6C_{\ms{Kma}}} 
			\br{D_h^2(\theta, \tilde\theta) - \frac{2 L}{nh}}
			\eqcm
	\end{equation*}
	which yields the claimed inequality after rearranging the terms.
\end{proof}
Define
\begin{align*}
	\bar F_t(\theta, \tilde \theta) &:= \sum_{i=1}^n w_{i,t} \Ex*{d(Y_{x_i}, g(x_i - t, \theta))^2 - d\brOf{Y_{x_i}, g(x_i - t, \tilde\theta)}^2}\eqcm\\
	\hat F_t(\theta, \tilde \theta) &:= \sum_{i=1}^n w_{i,t} \br{d(y_i, g(x_i - t, \theta))^2 - d\brOf{y_i, g(x_i - t, \tilde\theta)}^2}
	\eqfs
\end{align*}
\begin{lemma}\label{lmm:locgeo:Ubounds}\mbox{ }
\begin{enumerate}[label=(\roman*)]
\item 
	Assume \textsc{Kernel}, \textsc{HölderSmoothEx}, and \textsc{VarIneq}. Then
	\begin{equation*}
		U_t(\theta_t) \leq C_{\ms{Vlo}}  C_{\ms{Smo}}^2 h^{2\beta}
		\eqfs
	\end{equation*}
\item
	Assume \textsc{Kernel}, \textsc{HölderSmoothEx}, and \textsc{VarIneq}.
	Then
	\begin{equation*}
		U_t(\hat\theta_{t,h})
		\leq
		C_{\ms{Vlo}}\br{\bar F_t(\hat\theta_{t,h}, \theta_{t}) - \hat F_t(\hat\theta_{t,h}, \theta_{t})}
		+
		C_{\ms{Vlo}} C_{\ms{Smo}}^2 h^{2\beta}
		\eqfs
	\end{equation*}		
\end{enumerate}
\end{lemma}
\begin{proof}\mbox{ }
\begin{enumerate}[label=(\roman*)]
\item 
	Applying first \textsc{VarIneq} then \textsc{HölderSmoothEx} and finally \textsc{Kernel}, we obtain
	\begin{align*}
		U_t(\theta_t) 
		&= 
		\sum_{i=1}^n w_{i,t} d\brOf{g(x_i-t, \theta_t), m_{x_i}}^2 
		\\&\leq
		C_{\ms{Vlo}}\sum_{i=1}^n w_{i,t} \Ex{d\brOf{Y_{x_i}, g(x_i-t, \theta_t)}^2 - d\brOf{Y_{x_i}, m_{x_i}}^2 }
		\\&\leq 
		C_{\ms{Vlo}} C_{\ms{Smo}}^2  \sum_{i=1}^n w_{i,t} \abs{x_i - t}^{2\beta}
		\\&\leq 
		C_{\ms{Vlo}} C_{\ms{Smo}}^2 h^{2\beta}
		\eqfs
	\end{align*}
\item 
	For all $\theta\in\Theta$, by \textsc{VarIneq},
	\begin{align*}
		C_{\ms{Vlo}}^{-1} U_t(\theta)
		&\leq 
		\sum_{i=1}^n w_{i,t} \Ex{d(Y_{x_i}, g(x_i-t, \theta))^2-d(Y_{x_i}, m_{x_i})^2}
		\\&\leq 
		\bar F_t(\theta, \theta_{t})
		+
		\sum_{i=1}^n w_{i,t} \Ex{d(Y_{x_i}, g(x_i-t, \theta_{t}))^2-d(Y_{x_i}, m_{x_i})^2}
		\eqfs
	\end{align*}
	By \textsc{HölderSmoothEx} and \autoref{lmm:local_weights} with \textsc{Kernel},
	\begin{align*}
		\abs{\sum_{i=1}^n w_{i,t} \Ex{\old{Y_{x_i}}{g(x_i-t, \theta_{t})}^2-\old{Y_{x_i}}{m_{x_i}}^2}}
		&\leq 
		C_{\ms{Smo}}^2 \sum_{i=1}^n w_{i,t} \abs{x_i-t}^{2\beta}
		\\&\leq 
		C_{\ms{Smo}}^2 h^{2\beta}
		\eqfs
	\end{align*}
	By the minimizing property of $\hat\theta_{t,h}$, $\hat F_t(\hat\theta_{t,h}, \theta_{t}) < 0$. Putting all together yields
	\begin{equation*}
		C_{\ms{Vlo}}^{-1} U_t(\hat\theta_{t,h})
		\leq
		\bar F_t(\hat\theta_{t,h}, \theta_{t}) - \hat F_t(\hat\theta_{t,h}, \theta_{t})
		+
		C_{\ms{Smo}}^2 h^{2\beta}
		\eqfs
	\end{equation*}
\end{enumerate}
\end{proof}
Next, we bound a variance term using chaining. 
\begin{lemma}\label{lmm:locgeo:var}
	Let $\mc B \subset \Theta$ and $\theta_\bullet \in \mc B$.
	Assume \textsc{MomentA} and \textsc{Kernel}.
	Then, 
	\begin{equation*}
		\Ex*{\sup_{\theta\in\mc B} \abs{\bar F_t(\theta, \theta_\bullet) - \hat F_t(\theta, \theta_\bullet)}^\kappa} 
		\leq 
		 c_{\kappa} 
			 	\br{
			 		\br{C_{\ms{Kmi}}C_{\ms{Kma}}}^\frac12
			 		C_{\ms{MoA}}
			 		\gamma_2(\mc B, \mf b_h)
			 		(nh)^{-\frac12}
			 	}^\kappa
		\eqfs
	\end{equation*}
\end{lemma}
\begin{proof}
Define
\begin{align*}
	Z_i(\theta) &:=  w_{i,t} \Bigg(\old{y_i}{g(x_i-t, \theta)}^2 - \old{y_i}{g(x_i-t, \theta_\bullet)}^2 - \\&\hphantom{:=}\ \Ex*{\old{y_i}{g(x_i-t, \theta)}^2 - \old{y_i}{g(x_i-t, \theta_\bullet)}^2}\Bigg)
\end{align*}
Recall the definitions of $\lozenge$ and $\mf a$ at the beginning of the section to obtain
\begin{align*}
	&\Ex{\abs{Z_i(\theta)}} 
	\\&= \Ex*{w_{i,t} \Ex*{\abs{\lozenge(y_i,Y_{x_i},g(x_i-t, \theta),g(x_i-t, \theta_\bullet))} | y_i}}
	\\&\leq
	w_{i,t} d(g(x_i-t, \theta), g(x_i-t, \theta_\bullet)) \Ex{\mf a(y_i, Y_{x_i})}
	\eqfs
\end{align*}
By the triangle inequality for $\mf a$ (see auxiliary result \autoref{lmm:locgeo:pseudometric} below) and \textsc{MomentA}, \begin{equation*}
	\sup_{i\in\{1,\dots,n\}}\Ex{\mf a(Y_{x_i}, y_i\pr)} \leq 2C_{\ms{MoA}} < \infty
	\eqcm
\end{equation*}
such that the processes $Z_i$ are integrable. Furthermore, $Z_1, \dots, Z_n$ are independent. Moreover, $\Ex{Z_i(\theta)} = 0$ for all $\theta\in\Theta$, and $Z_i(\theta_\bullet) = 0$. They fulfill the following quadruple property: Let $Z_i\pr$ be independent copies of $Z_i$ with $y_i$ replaced by the independent copy $y_i\pr$. Then, for $\theta,\theta\pr\in\Theta$,
\begin{equation*}
	\abs{Z_i(\theta) - Z_i(\theta\pr) - Z_i\pr(\theta) + Z_i\pr(\theta\pr)} 
	\leq 
	w_{i,t} \mf a(y_i, y_i\pr) d(g(x_i-t, \theta), g(x_i-t, \theta\pr))
	\eqfs
\end{equation*}
As $w_{i,t} = 0$ for $\abs{x_i - t} > h$, we have
\begin{equation*}
	w_{i,t} d(g(x_i-t, \theta), g(x_i-t, \theta\pr)) \leq w_{i,t} \sup_{x\in[-1,1]} d(g(xh, \theta), g(xh, \tilde\theta)) = w_{i,t} \mf b_h(\theta, \theta\pr)
	\eqfs
\end{equation*}
Thus, \autoref{thm:empproc} implies 
\begin{align*}
	\Ex*{\sup_{\theta\in\mc B} \abs{\sum_{i=1}^n Z_i(\theta)}^\kappa} 
	&\leq 
	c_{\kappa} \gamma_2(\mc B, \mf b_h)^\kappa \Ex*{\br{\sum_{i=1}^n w_{i,t}^2 \mf a(y_i, y_i\pr)^2}^\frac\kappa2}
	\eqfs
\end{align*}
Define $W = \sum_{i=1}^n w_{i,t}^2$ and $v_i = w_{i,t}^2/W$. We obtain, using Jensen's inequality,
\begin{align*}
	\Ex*{\br{\sum_{i=1}^n w_{i,t}^2 \mf a(y_i, y_i\pr)^2}^\frac\kappa2}
	&=
	\Ex*{\br{W \sum_{i=1}^n v_i \mf a(y_i, y_i\pr)^2}^\frac\kappa2}
	\\&\leq
	W^{\frac\kappa2} \sum_{i=1}^n v_i \Ex*{\mf a(y_i, y_i\pr)^\kappa}
	\eqfs
\end{align*}
Thus, $\Ex*{\mf a(y_i, y_i\pr)^\kappa} \leq 2^\kappa \Ex*{\mf a(y_i, m_{x_i})^\kappa} \leq 2^\kappa C_{\ms{MoA}}^\kappa$. Furthermore, $W \leq \frac{6C_{\ms{Kmi}}C_{\ms{Kma}}}{nh}$ by \autoref{lmm:local_weights} (below). We obtain
\begin{equation*}
	\Ex*{\sup_{\theta\in\mc B} \abs{\bar F_t(\theta, \theta_\bullet) - \hat F_t(\theta, \theta_\bullet)}^\kappa} 
	\leq 
	c_{\kappa} 
	\br{
		\br{C_{\ms{Kmi}}C_{\ms{Kma}}}^\frac12
		C_{\ms{MoA}}
		\gamma_2(\mc B, \mf b_h)
		(nh)^{-\frac12}
	}^\kappa
	\eqfs
\end{equation*}
\end{proof}
A major step for obtaining a bound on the objects of interest instead of their objective function consists in using a \textit{peeling device} (also called \textit{slicing}). This is applied below: We first bound the probability $\Prof{D_h^2(\hat\theta_{t,h}, \theta_t) \in [a,b]}$, then infer a bound on $\Ex{D_h^2(\hat\theta_{t,h}, \theta_t)}$ from it.
\begin{proof}[Proof of \autoref{thm:locgeo}]
	Assume $D_h^2(\hat\theta_{t,h}, \theta_t) \in [a,b]$. Then $\mf b_h(\hat\theta_{t,h}, \theta_t) \leq C_{\ms{IBS}} b^{\frac12}$ by \textsc{IntBoundsSup}.
	Furthermore, by \autoref{lmm:locgeo:DhU} and \autoref{lmm:locgeo:Ubounds},
	\begin{align*}
		a 
		&\leq
		D_h^2(\hat\theta_{t,h}, \theta_t) 
		\\&\leq 
		c C_{\ms{Kmi}}C_{\ms{Kma}} \br{U_t(\hat\theta_{t,h}) + U_t(\theta_t)} +  
		c C_{\ms{Lip}}\mf b_h(\hat\theta_{t,h}, \theta_t) (nh)^{-1}
		\\&\leq 
		c C_{\ms{Kmi}}C_{\ms{Kma}} 
		\br{
			C_{\ms{Vlo}}\br{\bar F_t(\hat\theta_{t,h}, \theta_{t}) - \hat F_t(\hat\theta_{t,h}, \theta_{t})} + C_{\ms{Vlo}} C_{\ms{Smo}}^2 h^{2\beta} 
		} 
		+  c C_{\ms{Lip}}C_{\ms{IBS}} b^{\frac12} (nh)^{-1}
		\eqfs
	\end{align*}
	By \textsc{IntBoundsSup}, $\mf b_h(\theta, \tilde\theta)^2 \leq C_{\ms{IBS}}^2 D_h^2(\theta, \tilde\theta)$ for $\theta, \tilde\theta\in\Theta_h$. As $D_h^2(\hat\theta_{t,h}, \theta_t) \leq b$, we obtain $\hat\theta_{t,h} \in \mc B_{b}$, where 
	\begin{align*}
		\mc B_{b} &:= \cb{\theta\in\Theta \colon \mf b_h(\theta, \tilde\theta)^2 \leq C_{\ms{IBS}}^2 b}
		\eqfs
	\end{align*}
	Thus,
	\begin{align*}
		\bar F_t(\hat\theta_{t,h}, \theta_{t}) - \hat F_t(\hat\theta_{t,h}, \theta_{t})
		\leq
		\sup_{\theta\in{\mc B}_{b}} \abs{\bar F_t(\theta, \theta_{t}) - \hat F_t(\theta, \theta_{t})}
		\eqfs
	\end{align*}
	Hence,
	\begin{align*}
		a \leq A_0 + A_1 b^{\frac12} +  A_2 \sup_{\theta\in{\mc B}_{b}} \abs{\bar F_t(\theta, \theta_{t}) - \hat F_t(\theta, \theta_{t})}
		\eqcm
	\end{align*}
	where
	$A_0 = 
			c C_{\ms{Kmi}}C_{\ms{Kma}} C_{\ms{Vlo}} C_{\ms{Smo}}^2 h^{2\beta}$, 
	$A_1 = c C_{\ms{Lip}}C_{\ms{IBS}} (nh)^{-1}$, 
	and 
	$A_2 = c C_{\ms{Kmi}}C_{\ms{Kma}} C_{\ms{Vlo}}$.
	Using Markov's inequality,
	\begin{align*}
		\PrOf{D_h^2(\hat\theta_{t,h}, \theta_t) \in [a,b]}
		&\leq
		\PrOf{A_0 + A_1 b^{\frac12} + A_2 \sup_{\theta\in{\mc B}_{b}} \abs{\bar F_t(\theta, \theta_{t}) - \hat F_t(\theta, \theta_{t})} \geq a}
		\\&\leq
		c_\kappa \frac{A_0^\kappa + A_1^\kappa b^{\frac\kappa2} + A_2^\kappa \Ex*{\sup_{\theta\in{\mc B}_{b}} \abs{\bar F_t(\theta, \theta_{t}) - \hat F_t(\theta, \theta_{t})}^\kappa}}{a^\kappa}
		\eqfs
	\end{align*}
	By \autoref{lmm:locgeo:var} with $\theta_\bullet = \theta_{t}$ and with \textsc{EntropyGeod},
	\begin{align*}
		&\Ex*{\sup_{\theta\in{\mc B}_{b}} \abs{\bar F_t(\theta, \theta_{t,h}) - \hat F_t(\theta, \theta_{t,h})}^\kappa} 
		\\&\leq 
		c_{\kappa} 
			 	\br{
			 		\br{C_{\ms{Kmi}}C_{\ms{Kma}}}^\frac12
			 		C_{\ms{MoA}} 
			 		\gamma_2({\mc B}_{b}, \mf b_h)
			 		(nh)^{-\frac12}
			 	}^\kappa
		\\&\leq
		c_{\kappa} 
		 	\br{
		 		\br{C_{\ms{Kmi}}C_{\ms{Kma}}}^\frac12
		 		C_{\ms{MoA}} 
		 		C_{\ms{EnG}} C_{\ms{IBS}}^\alpha \max(b^{\frac12}, b^{\frac\alpha2})
		 		(nh)^{-\frac12}
		 	}^\kappa
		\eqfs
	\end{align*}
	Thus,
	\begin{align*}
		\PrOf{D_h^2(\hat\theta_{t,h}, \theta_t) \in [a,b]}
		\leq
		c_\kappa \frac{A_0^\kappa + A_3^\kappa \max(b,b^\alpha)^{\frac\kappa2}}{a^\kappa}
		\eqcm
	\end{align*}
	where 
	\begin{align*}
		A_3 = A_1 + \br{C_{\ms{Kmi}}C_{\ms{Kma}}}^\frac12
				 		C_{\ms{MoA}} 
				 		C_{\ms{EnG}} C_{\ms{IBS}}
				 		(nh)^{-\frac12} A_2	
		\eqfs
	\end{align*}
 	By \autoref{lmm:locgeo:peelhelper} below and with $h\geq \frac cn$, $\frac{2}{2-\alpha} \geq 1$, this yields
	\begin{align*}
		\Ex{D_h^2(\hat\theta_{t,h}, \theta_t)} 
		&\leq 
		c_\kappa \br{A_0 + A_3^2 + A_3^{\frac{2}{2-\alpha}}}
		\\&\leq
 		C_1 h^{2\beta} + C_2 (nh)^{-1} + C_3 (nh)^{-2} 
		\eqcm
	\end{align*}
	where 
	$C_1 = c_\kappa C_{\ms{Kmi}}C_{\ms{Kma}} C_{\ms{Vlo}} C_{\ms{Smo}}^2$,
	$C_2 = c_{\alpha\kappa}  \br{C_{\ms{IBS}}^2 C_{\ms{Kmi}}^3C_{\ms{Kma}}^3
	 		C_{\ms{MoA}}^2 C_{\ms{EnG}}^2 C_{\ms{Vlo}}^2}^{\frac{2}{2-\alpha}}$,
	and
	$C_3 = c_{\alpha\kappa} \br{C_{\ms{Lip}}C_{\ms{IBS}}}^{\frac{2}{2-\alpha}}$.
\end{proof}
\subsubsection{Auxiliary Results}
A map $d \colon \mc Q\times \mc Q \to [0,\infty]$ is called \textit{pseudo-metric} on $\mc Q$, if $d$ is symmetric with $d(q,q) = 0$ for all $q\in\mc Q$ and obeys the triangle inequality.
\begin{lemma}\label{lmm:locgeo:pseudometric}
	The functions $\mf a$ and $\mf b_h$ are pseudo-metrics on $\mc Q$ and $\Theta$, respectively.
\end{lemma}
\begin{proof}
	Recall $\ol qp=d(q,p)$. All properties for $\mf a$ are straight forward. For the triangle inequality, as
	\begin{equation*}
		\frac{\ol yq^2-\ol yp^2-\ol zq^2+\ol zp^2}{\ol qp} = 
		\frac{\ol yq^2-\ol yp^2-\ol vq^2+\ol vp^2}{\ol qp} 
		+
		\frac{\ol vq^2-\ol vp^2-\ol zq^2+\ol zp^2}{\ol qp} 
		\eqcm
	\end{equation*}
	we obtain
	\begin{align*}
		&\sup_{q\neq p}\frac{\ol yq^2-\ol yp^2-\ol zq^2+\ol zp^2}{\ol qp} 
		\\&\leq
		\sup_{q\neq p} \frac{\ol yq^2-\ol yp^2-\ol vq^2+\ol vp^2}{\ol qp} 
		+
		\sup_{q\neq p}	\frac{\ol vq^2-\ol vp^2-\ol zq^2+\ol zp^2}{\ol qp} 
		\eqfs
	\end{align*}
	For $\mf b_h$ the argument is almost identical.
\end{proof}
The weights $w_{i,t}$ have following properties, see \cite[Proposition 1.13]{tsybakov08}.
\begin{lemma}\label{lmm:local_weights}
	Assume \textsc{Kernel} and $h \geq \frac2n$. Then
	\begin{align*}
	w_{i,t} \geq 0 \eqcm\qquad
	\sum_{i=1}^n w_{i,t} = 1 \eqcm\qquad
	w_{i,t} \leq \frac{6C_{\ms{Kmi}}C_{\ms{Kma}}}{nh}\eqcm\\
	w_{i,t} = 0 \text{ if }\abs{x_i-t} > h\eqcm\qquad
	\sum_{i=1}^n w_{i,t}^2 \leq \frac{6C_{\ms{Kmi}}C_{\ms{Kma}}}{nh}
	\end{align*}
	for all $t\in[0,1]$ and $h \geq \frac2{n}$.
\end{lemma}
\begin{lemma}\label{lmm:locgeo:lipschitz}
	Assume \textsc{Lipschitz}.
	Let $x, y \in [-\frac12,\frac12]$, $\theta, \tilde \theta \in \Theta_h$. Then
	\begin{equation*}
		\old{g(xh, \theta)}{g(xh, \tilde\theta)}^2 -\old{g(yh, \theta)}{g(yh, \tilde\theta)}^2 
		\leq
		c C_{\ms{Lip}} \abs{x-y} \mf b_h(\theta, \tilde\theta)
		\eqfs
	\end{equation*}
\end{lemma}
\begin{proof}
	First, we write the difference of two squared numbers as the product of their sum and their difference,
	\begin{align*}
		&\old{g(xh, \theta)}{g(xh, \tilde\theta)}^2 -\old{g(yh, \theta)}{g(yh, \tilde\theta)}^2 
		\\&=
		\br{\old{g(xh, \theta)}{g(xh, \tilde\theta)} -\old{g(yh, \theta)}{g(yh, \tilde\theta)}}
		\\&\hphantom{=}\ 
		\br{\old{g(xh, \theta)}{g(xh, \tilde\theta)} +\old{g(yh, \theta)}{g(yh, \tilde\theta)}}
		\eqfs
	\end{align*}
	The difference can be transformed noting that in general the triangle inequality yields 
	\begin{equation*}
		\ol yq - \ol zp \ =\ \ol yq - \ol yp + \ol yp - \ol zp \ \leq\  \ol qp + \ol yz
		\eqfs
	\end{equation*}
	Thus, 
	\begin{align*}
		&\old{g(xh, \theta)}{g(xh, \tilde\theta)} -\old{g(yh, \theta)}{g(yh, \tilde\theta)}
		\\&\leq
		\old{g(xh, \theta)}{g(yh, \theta)} + \old{g(xh, \tilde\theta)}{g(yh, \tilde\theta)}
		\\&\leq
		2 C_{\ms{Lip}} \abs{x-y}
		\eqcm
	\end{align*}
	where we used \textsc{Lipschitz} in the last inequality. The summands of the other factor can each be bounded by $\mf b_h$,
	\begin{align*}
		&\old{g(xh, \theta)}{g(xh, \tilde\theta)} +\old{g(yh, \theta)}{g(yh, \tilde\theta)}
		\\&\leq
		2 \mf b_h(\theta, \tilde\theta)
		\eqfs
	\end{align*}
	Putting these bounds together yields the result.
\end{proof}
\begin{lemma}\label{lmm:locgeo:peelhelper}
Let $V$ be a nonnegative random variable. Assume that for all $0 < a < b <\infty$, it holds
\begin{equation*}
	\PrOf{V \in [a,b]} \leq c \frac{u^\kappa + \br{v \max(b,b^\alpha)^\frac12}^\kappa}{a^\kappa}
	\eqfs
\end{equation*}
where $c \geq 1, u, v > 0$, $\kappa > 2$.
Then 
\begin{equation*}
	\Ex{V} \leq c_{\kappa} c^{\frac2{\kappa}} \br{u + v^2}
	\eqfs
\end{equation*}
\end{lemma}
\begin{proof}
	For $s > 0$, 
	\begin{align*}
		\PrOf{V > s}
		&\leq
		\sum_{k=0}^\infty \PrOf{V \in [s 2^k, s 2^{k+1}]}
		\\&\leq
		\sum_{k=0}^\infty c \frac{u^\kappa + c_\kappa v \max(s^{\frac12} 2^{\frac k2}, s^{\frac\alpha2} 2^{\frac{\alpha k}2})^\kappa}{s^\kappa 2^{k\kappa}}
		\\&\leq
		c_\kappa \br{u^\kappa s^{-\kappa} \sum_{k=0}^\infty 2^{-k\kappa} 
		+ 
		v^\kappa \max\brOf{ s^{-\frac\kappa2} \sum_{k=0}^\infty 2^{-\frac{k\kappa}2}, s^{-\kappa\frac{2-\alpha}{2}} \sum_{k=0}^\infty 2^{-k\kappa\frac{2-\alpha}{2}}}}
		\\&\leq
		c_{\kappa,\alpha} \br{u^\kappa s^{-\kappa}  + v^\kappa s^{-\frac\kappa2} + v^{\kappa} s^{-\kappa\frac{2-\alpha}{2}}}
		\eqfs
	\end{align*}
	We integrate the tail to bound the expectation, 
	\begin{align*}
		\Ex{V} 
		\leq
		\int_{0}^\infty \Prof{V > s} \dl s
		\eqfs
	\end{align*}
	For $A \geq 0$, $\tau > 1$,
	\begin{align*}
		\int_{0}^\infty \min(1, A s^{-\tau}) \dl s &\leq \frac{\tau}{\tau-1} {A}^{\frac1\tau}\eqfs
	\end{align*}
	Applying this inequalities to the tail bound above, we obtain
	\begin{align*}
		\Ex{V} 
		\leq
		c_{\kappa, \alpha}\br{u + v^2 + v^{\frac{2}{2-\alpha}}}
		\eqfs
	\end{align*}
\end{proof}
\subsubsection{Main Theorems}
We use \autoref{thm:locgeo} to prove the two main theorems concerning \texttt{LocGeo}.

Instead of a general link function $g\colon \R\times\Theta \to \mc Q$, we use an exponential map $\Exp\colon \mc Q \times \R^k \to \mc Q$ with $g(x, \theta) = \Exp(p, xv)$ for $\theta = (p, v)$. The set parameterizing geodesics is $\Theta \subset \mc Q \times \R^k$. For a chosen bandwidth $h\geq \frac2n$ and a constant $R>0$, we minimize over the subset $\Theta_h:= \Theta \cap (\mc Q \times \ball(0, \norm, Rh^{-1}))$ to obtain $\hat \theta_{t,h} = (\hat m_t, \hat{\dot m}_t)$ as an estimator of $\theta_t = (m_t, \dot m_t)$. In this setting, some conditions and bounds can be replaced:

\begin{lemma}\label{lmm:locgeo:expmap}\mbox{ }
\begin{enumerate}[label=(\roman*)]
	\item 
	\textsc{ExpMap} implies \textsc{Lipschitz} with $C_{\ms{Lip}} = 2 C_{\ms{Mup}} R$ and \textsc{IntBoundsSup} with $C_{\ms{IBS}} = 2 C_{\ms{Mup}} C_{\ms{Mlo}}$.
	\item 	
	\textsc{Entropy} and \textsc{ExpMap} imply \textsc{EntropyGeod} with $C_{\ms{EnG}} = c C_{\ms{Mlo}}^\alpha C_{\ms{Mup}} C_{\ms{Ent}} \sqrt{k}$.
	\item 
	Assume \textsc{ExpMap}. Then
	\begin{equation*}
		\Ex*{d(\hat m_t, m_t)^2} + h^2 \normof{\hat{\dot m}_t - \dot m_t}^2 \leq C_{\ms{Mlo}}^2 \Ex*{D_h^2(\hat\theta_{t,h}, \theta_{t})} 
		\eqfs
	\end{equation*}
\end{enumerate}
\end{lemma}
\begin{proof}\mbox{ }
\begin{enumerate}[label=(\roman*)]
\item Trivial.
\item Let $\mc B \subset \Theta_h$. 
Define 
\begin{align*}
	\mc B_{\mc Q} &:= \cb{q\in\mc Q \ \vert\  \exists v \in \R^k \colon (q,v)\in \mc B}\eqcm\\
	\mc B_{\R^k} &:= \cb{v\in\R^k \ \vert\  \exists q\in\mc Q \colon (q,v)\in \mc B}\eqfs
\end{align*}
By \textsc{ExpMap}
\begin{align*}
	\diam(\mc B, \mf b_h) 
	&\geq 
	\diam(\mc B, D_h) 
	\\&\geq 
	C_{\ms{Mlo}}^{-1}  \max\brOf{\diam(\mc B_{\mc Q}, d), h \diam(\mc B_{\R^k}, \norm)}
	\\&\geq
	c C_{\ms{Mlo}}^{-1}  \br{\diam(\mc B_{\mc Q}, d) + h \diam(\mc B_{\R^k}, \norm)}
	\eqfs
\end{align*}
Similarly, by \autoref{lmm:chain:sublinear},
\begin{align*}
	\gamma_2(\mc B, \mf b_h) 
	&\leq 
	c C_{\ms{Mup}} (\gamma_2(\mc B_{\mc Q}, d) + h \gamma_2(\mc B_{\R^k}, \norm)) 
	\eqfs
\end{align*}
By \textsc{Entropy}, $\gamma_2(\mc B_{\mc Q}, d) \leq C_{\ms{Ent}}\max(\diam(\mc B_{\mc Q}, d), \diam(\mc B_{\mc Q}, d)^\alpha)$. Furthermore, by \autoref{lmm:chain:euclid}, $\gamma_2(\mc B_{\R^k}, \norm) \leq c \sqrt{k} \diam(\mc B_{\R^k}, \norm)$.
Thus, 
\begin{align*}
	\gamma_2(\mc B, \mf b_h) 
	&\leq 
	c C_{\ms{Mup}} C_{\ms{Ent}} \sqrt{k} \br{\max(\diam(\mc B_{\mc Q}, d), \diam(\mc B_{\mc Q}, d)^\alpha) + h \diam(\mc B_{\R^k}, \norm)}
	\\&\leq 
	c C_{\ms{Mlo}}^\alpha C_{\ms{Mup}} C_{\ms{Ent}} \sqrt{k} \max(\diam(\mc B, \mf b_h), \diam(\mc B, \mf b_h)^\alpha)
	\eqfs
\end{align*}
\item Trivial.
\end{enumerate}
\end{proof}
Thus, we can use \autoref{thm:locgeo} to show bounds on $\Ex{d(\hat m_t, m_t)^2}$, which is our main goal. Note that the bound on $\Ex{D_h^2(\hat\theta_{t,h}, \theta_{t})}$ also entails a bound on the derivatives of $\hat m$ and $m_t$.
\begin{proof}[Proof of \autoref{cor:locgeo:bounded}]
	We want to apply \autoref{thm:locgeo}. 
	\textsc{VarIneq}, \textsc{HölderSmoothEx}, and \textsc{Kernel} are assumed.	
	\textsc{ExpMap} and \textsc{Entropy}  imply \textsc{Lipschitz}, \textsc{IntBoundsSup}, and \textsc{EntropyGeod}, see \autoref{lmm:locgeo:expmap}.
	As $\diam(\mc Q, d) < \infty$, $\ol yq^2 - \ol yp^2 - \ol zq^2 + \ol zp^2 \leq 4 \ol qp \diam(\mc Q, d)$. Thus, $\mf a(y, z) \leq 4  \diam(\mc Q, d)$ and we can choose $C_{\ms{MoA}} = 4  \diam(\mc Q, d)$ to fulfill \textsc{Moment}.
	Thus, \autoref{thm:locgeo} with \autoref{lmm:locgeo:expmap} and $h\geq \frac2n$ show
	\begin{equation*}
		\Ex*{d(\hat m_t, m_t)^2}
		\leq
		C_1 h^{2\beta} + (C_2+C_3) (nh)^{-1}
		\eqfs
	\end{equation*}
	Integrating the inequality finishes the proof.
\end{proof}
\begin{proof}[Proof of \autoref{cor:locgeo:hadamard}]
	We want to apply \autoref{thm:locgeo}. 
	\textsc{HölderSmoothEx}, and \textsc{Kernel} are assumed.	
	\textsc{ExpMap} and \textsc{Entropy}  imply \textsc{Lipschitz}, \textsc{IntBoundsSup}, and \textsc{EntropyGeod}, see \autoref{lmm:locgeo:expmap}.
	Due to the quadruple inequality in Hadamard spaces, $\mf a(q,p) \leq 2d(q,p)$ and \textsc{Moment} implies \textsc{MomentA} with $C_{\ms{MoA}} = 2C_{\ms{Mom}}$. Furthermore, \textsc{VarIneq} is always true in Hadamard spaces with $C_{\ms{Vlo}} = 1$. 
	Thus, \autoref{thm:locgeo} with \autoref{lmm:locgeo:expmap} and $h\geq \frac2n$ show
	\begin{equation*}
		\Ex*{d(\hat m_t, m_t)^2}
		\leq
		C_1 h^{2\beta} + (C_2+C_3) (nh)^{-1}
		\eqfs
	\end{equation*}
	Integrating the inequality finishes the proof.
\end{proof}
\subsection{Corollaries on the Hypersphere}
In this section, we apply the main theorems concerning \texttt{LocFre}, \texttt{OrtFre}, and \texttt{LocGeo} on bounded spaces to prove the corollaries on the hypersphere. 

To this end, we need to show \textsc{Entropy}: 
There is $C_{\ms{Ent}} \in [1,\infty)$ such that $\gamma_2(\mc B, d_{\mb S^k}) \leq C_{\ms{Ent}}\diam(\mc B, d_{\mb S^k})$ for all $\mc B \subset \mb S^k$. As $\mb S^k \subset \R^{k+1}$, $\normof{q-p} \leq d_{\mb S^k}(q,p) \leq \frac\pi2 \normof{q-p}$, and \autoref{lmm:chain:euclid}, we can choose $C_{\ms{Ent}} = c \sqrt{k+1}$.
\subsubsection{\autoref{cor:locfre:sphere} -- \texttt{LocFre}}
\textsc{Kernel} is fulfilled by using the Epanechnikov kernel. \textsc{VarIneq} is assumed.
\textsc{Entropy} was shown above with $C_{\ms{Ent}} = 2 \sqrt{k+1}$.
\textsc{HölderSmoothDensity} is fulfilled by the smoothness condition in the corollary and noting that $\diam(\mb S^k) = \pi$ so that we can set $C_{\ms{Len}} = \pi$ and $C_{\ms{Int}} = \pi^2$.
\subsubsection{\autoref{cor:trifre:sphere} -- \texttt{OrtFre}}
This corollary is shown exactly the same way as the one for \texttt{LocFre}.
\subsubsection{\autoref{cor:locgeo:sphere} -- \texttt{LocGeo}}
To apply the theorem for \texttt{LocGeo} on bounded spaces to the hypersphere, we have to show \textsc{ExpMap}, i.e., we have to find constants $C_{\ms{Mup}},  C_{\ms{Mlo}} \in [1,\infty)$ such that
\begin{align*}
	d\brOf{\Exp(q, v), \Exp(p, u)} &\leq C_{\ms{Mup}} \brOf{d(q,p) + \normof{v-u}}\eqcm \\
	\int_{-\frac12}^{\frac12} d\brOf{\Exp(q, xv), \Exp(p, xu)}^2 \dl x &\geq C_{\ms{Mlo}}^{-2} \brOf{d(q,p)^2 + \normof{v-u}^2}
\end{align*}
for all $(q,v), (p,u) \in \Theta$ with $\normof{u}, \normof{v} \leq R$. We set $R = \pi$.
The auxiliary results \autoref{lmm:sphere:distance_upper} and \autoref{lmm:integral_metric_sphere_lower_bound} below show that we can choose $C_{\ms{Mup}} = 2\pi$ and $C_{\ms{Mlo}} = \sqrt{2} \pi$, respectively.

\textsc{Kernel} (with $C_{\ms{Kmi}} = C_{\ms{Kma}} = C_{\ms{Ker}}$), and \textsc{VarIneq} are assumed. \textsc{Entropy} was shown above with $C_{\ms{Ent}} = 2 \sqrt{k+1}$.

In proper Alexandrov spaces of nonnegative curvature, like (hyper-)spheres, a reverse variance inequality holds, \cite[Theorem 5.2]{ohta12}, 
\begin{equation*}
	\Ex{d(Y_t, q)^2-d(Y_t, m_t)^2} \leq d(q, m_t)^2
	\eqfs
\end{equation*}
This and the smoothness condition stated in the corollary imply \textsc{HölderSmoothEx}.
\subsubsection{Auxiliary Results}
\begin{lemma}\label{lmm:sphere:distance_upper}
	Let $(p,u), (q,v) \in \ms T\mb S^k$. Then
	\begin{equation*}
		d(\Exp(q, v), \Exp(p, u)) \leq \frac\pi2 \abs{q-p} + 2\pi \abs{v-u}
		\eqfs
	\end{equation*}
\end{lemma}
\begin{proof}
We can bound the intrinsic metric on the sphere by the extrinsic one,
\begin{align*}
	d(\Exp(q,v), \Exp(p,u))
	&\leq
	\frac\pi2 \abs{\Exp(q,v)- \Exp(p,u)}
	\\&\leq
	\frac\pi2 \br{\abs{\cos(\abs{v})q - \cos(\abs{u})p} +	\abs{\frac{\sin(\abs v)}{\abs v}v - \frac{\sin(\abs u)}{\abs u} u}}
	\eqfs
\end{align*}
For the $\cos$-terms, it holds
\begin{align*}
\abs{\cos(\abs{v})q - \cos(\abs{u})p}
&\leq
\abs{\cos(\abs{v})} \abs{q-p} + \abs{p}\abs{\cos(\abs{v}) - \cos(\abs{u})}
\\&\leq
\abs{q-p} + \abs{\abs{v} - \abs{u}}
\eqfs
\end{align*}
For the $\sin$-terms, let $J(x)$ be the Jacobi matrix of the function $\R^k \to \R^k,\, x \mapsto \frac{\sin(\abs x)}{\abs x} x$.
Then
\begin{align*}
	\abs{\frac{\sin(\abs v)}{\abs v}v - \frac{\sin(\abs u)}{\abs u} u} \leq \sup_{x\in\R^k} \normof{J(x)}_{\ms{op}} \abs{u-v}
	\eqfs
\end{align*}
As
\begin{align*}
J(x) = \br{\cos(\abs{x}) - \frac{\sin(|x|)}{|x|}} \abs{x}^{-2} xx\tr + \frac{\sin(|x|)}{|x|} I_k
\eqcm
\end{align*}
it holds 
\begin{align*}
\normof{J(x)}_{\ms{op}} 
\leq 
\br{ \abs{\cos(\abs{x})} + \abs{\frac{\sin(|x|)}{|x|}} } \normof{\abs{x}^{-2} xx\tr}_{\ms{op}} + \abs{\frac{\sin(|x|)}{|x|}} \normof{I_k}_{\ms{op}} 
\leq 3
\eqfs
\end{align*}
Thus, $d(\Exp(q,v), \Exp(p,u)) \leq \frac\pi2 \br{\abs{q-p} + \abs{\abs{v} - \abs{u}} + 3\abs{u-v}}$.
\end{proof}
\begin{lemma}\label{lmm:integral_metric_sphere_lower_bound}
Let $(p,u), (q,v) \in \ms T\mb S^k$ with $\abs{u}, \abs{v} \leq \pi$. Then
\begin{equation*}
	\int_{-\frac12}^\frac12 d_{\mb S^k}(\Exp(p, xu), \Exp(q, xv))^2 \dl x \geq \frac{1}{\pi} \abs{p-q}^2 +  \frac1{2\pi^2} \abs{v-u}^2
	\eqfs
\end{equation*}
\end{lemma}
\begin{proof}
	First we lower bound the intrinsic distance $d_{\mb S^k}$ by the euclidean one and use the explicit representation of the $\Exp$-function,
	\begin{equation*}
		d_{\mb S^k}(\Exp(p, xu), \Exp(q, xv))^2 \geq  \abs{\cos(x \abs{u})p + \sin(x\abs{u}) \frac{u}{\abs{u}} - \cos(x \abs{v})q - \sin(x\abs{v}) \frac{v}{\abs{v}}}^2
		\eqfs
	\end{equation*}
	When integrating after calculating the squared norm, all summands with a $\cos()\sin()$-factor disappear, because of symmetry. Thus, we obtain
	\begin{align*}
		&	\int_{-\frac12}^{\frac12} d_{\mb S^k}(\Exp(p, xu), \Exp(q, xv))^2 \dl x 
		\\&\geq
		\int_{-\frac12}^{\frac12} \cos(x\abs{u})^2 p\tr p - 2 \cos(x\abs{u})\cos(x\abs{v})p\tr q + \cos(x\abs{v})^2 q\tr 	q  \,\dl x
		\\&\quad+ \int_{-\frac12}^{\frac12} \sin(x\abs{u})^2 \frac{u\tr u}{\abs{u}^2} - 2\sin(x\abs{u})\sin(x\abs{v}) 	\frac{u\tr v}{\abs{u}\abs{u}} + \sin(x\abs{v})^2 \frac{v\tr v}{\abs{v}^2}\,\dl x 
		\eqfs
	\end{align*}
	As $\abs{p}=\abs{q}=1$, $\cos(x)^2+\sin(x)^2=1$, $2\cos(\alpha)\cos(\beta) = \cos(\alpha - \beta) + \cos(\alpha + \beta)$, and $2\sin(\alpha)\sin(\beta) = \cos(\alpha - \beta) - \cos(\alpha + \beta)$, the right hand side reduces to
	\begin{equation*}
		\int_{-\frac12}^{\frac12} 2 - \br{\cos(xa)+\cos(xb)} p\tr q - \br{\cos(xa)-\cos(xb)} z\, \dl x \eqcm
	\end{equation*}
	where we set $a = \abs{u}-\abs{v}$, $b = \abs{u}+\abs{v}$, and $z = \frac{u\tr v}{\abs{u}\abs{v}}$.
	Integrating yields
	\begin{equation*}
		2 - 2 \br{\frac{\sin(\frac12 a)}{a} + \frac{\sin(\frac12 b)}{b}} q\tr p - 2 \br{\frac{\sin(\frac12a)}{a} - \frac{\sin(\frac12b)}{b}} z
		\eqfs
	\end{equation*}
	As $q\tr p = 1 - \frac12\abs{q-p}^2$, we can split the sum into two parts $A+B$, where
	\begin{align*}
		A &:= \br{\frac{\sin(\frac12 a)}{a} + \frac{\sin(\frac12 b)}{b}}\abs{q-p}^2\eqcm\\
		B &:= 2 - 2 \br{\frac{\sin(\frac12 a)}{a} + \frac{\sin(\frac12 b)}{b}} - 2 \br{\frac{\sin(\frac12a)}{a} - \frac{\sin(\frac12b)}{b}} z\eqfs
	\end{align*}
	The function $x\mapsto \sin(x)/x$ decreases on the interval $(0, \pi)$. Thus, 
	\begin{equation*}
		\frac{\sin(\frac12 a)}{a} + \frac{\sin(\frac12 b)}{b} 
		\geq
		\frac{\sin(\frac12 \pi)}{\pi} + \frac{\sin(\pi)}{2\pi}  
		= 
		\frac1\pi
	\end{equation*}
	as $\abs{v},\abs{u} \leq \pi$. In particular, $A \geq \frac1\pi \abs{q-p}^2$.
	To bound $B$, we will show $f(a,b,z) \geq 0$ for all $a \in [-\pi, \pi]$, $b \in [0, 2\pi]$, and $z\in[-1,1]$, where
	\begin{align*}
		&f(a,b,z) := \\&\ 2 - 2 \br{\frac{\sin(a/2)}{a} + \frac{\sin(b/2)}{b}}- 2 \br{\frac{\sin(a/2)}{a} - \frac{\sin(b/2)}{b}} z - \frac12 c \br{a^2+b^2 + (a^2-b^2)z} 
	\end{align*}
	with $c > 0$. This suffices as $a^2+b^2 + (a^2-b^2)z = 2\abs{v-u}^2$.
	As $f$ is linear in $z$, it is minimized either at $z=1$ or at $z=-1$. It holds
	\begin{align*}
		f(a,b,1) = 2 - \frac{4\sin(\frac12a)}{a} - c a^2\eqcm
		&& 
		f(a,b,-1) = 2 - \frac{4\sin(\frac12 b)}{b} - c b^2 \eqfs
	\end{align*}
	Thus, $f(a,b,z) \geq 0$ is true if and only if
	\begin{equation*}
		c \leq \inf_{x\in[-\pi, 2\pi]} \frac{2 - \frac{4\sin(x/2)}{x}}{x^2} = \frac1{2\pi^2}\eqfs
	\end{equation*}
	By setting $c = \frac1{2\pi^2}$, we obtain
	\begin{equation*}
		B \geq \frac1{2\pi^2} \abs{v-u}^2
		\eqfs
	\end{equation*}
\end{proof}
\section{Chaining}\label{sec:chaining}
\begin{theorem}[Empirical process bound]\label{thm:empproc}
	Let $(\mc Q, d)$ be a separable pseudo-metric space and $\mc B \subset \mc Q$. Let $Z_1, \dots, Z_n$ be centered, independent, and integrable stochastic processes indexed by $\mc Q$ with a $q_0 \in \mc B$ such that $Z_i(q_0) = 0$ for $i=1,\dots, n$.
	Let $(Z_1\pr, \dots, Z_n\pr)$ be an independent copy of $(Z_1, \dots, Z_n)$.
	Assume the following Lipschitz-property: There is a random vector $A$ with values in $\R^n$ such that
	\begin{equation*}
		\abs{Z_i(q)-Z_i(p)-Z_i\pr(q)+ Z_i\pr(p)} \leq A_i d(q,p)
	\end{equation*} 
	for $i = 1,\dots, n$ and all $q,p\in\mc B$.
	Let $\kappa \geq 1$.
	Then
	\begin{equation*}
		\Ex*{\sup_{q\in\mc B} \abs{\sum_{i=1}^nZ_i(q)}^\kappa} \leq c_{\kappa}\, \Ex*{\normof{A}_2^\kappa} \, 
		\gamma_2(\mc B, d)^\kappa
		\eqcm
	\end{equation*}
	where $c_{\kappa} \in (0,\infty)$ depends only on $\kappa$.
\end{theorem}
\begin{proof}
See \cite[Theorem  6]{schoetz19}.
\end{proof}
\begin{lemma}\label{lmm:chain:euclid}
	In the Euclidean space $\R^{k}$ with the metric induced by the Euclidean norm $\abs{\cdot}$, it holds
	$\gamma_2(\ball(x, r, \abs{\cdot}), \abs{\cdot}) \leq 2 r \sqrt{k}$ for any point $x\in\R^{k}$ and radius $r > 0$.
\end{lemma}
\begin{proof}
See \cite[section 4]{pollard90}.
\end{proof}
\begin{lemma}\label{lmm:chain:sublinear}
	Let $d$ and $d\pr$ be metrics on a set $\mc Q$. 
	\begin{enumerate}[label=(\roman*)]
	\item 
		Assume $d\leq B d\pr$ for a $B > 0$. Then 
		\begin{equation*}
			\gamma_2(\mc Q, d) \leq B \gamma_2(\mc Q, d\pr)
			\eqfs
		\end{equation*}
	\item 
		There is a universal constant $c > 0$ such that 
		\begin{equation*}
			\gamma_2(\mc Q, d+d\pr) \leq c \br{\gamma_2(\mc Q, d+d\pr) + \gamma_2(\mc Q, d+d\pr)}
			\eqfs
		\end{equation*}
	\end{enumerate}
\end{lemma}
\begin{proof}
	See \cite[Exercise 2.2.20 and Exercise 2.2.24]{talagrand14}
\end{proof}
\section{Geometry}\label{sec:geometry}
We introduce some terms from (metric) geometry, which are used in this article. See \cite{burago01} for a in depth introduction.

A metric space is called \textbf{proper} if every closed ball is compact.
Let $(\mc Q, d)$ be a metric space. For a continuous map $\gamma\colon[a,b]\to\mc Q$ define its \textbf{length} as 
\begin{equation*}
	L(\gamma) = \sup\cbOf{\sum_{i=1}^n d(\gamma(x_{i-1}), \gamma(x_{i})) \ \bigg\vert \ a=x_0  < x_1 < \dots < x_n = b, n\in\N}
	\eqfs
\end{equation*}
Define the \textbf{inner metric} of $(\mc Q, d)$ as $d_{\ms i}(q,p) = \inf L(\gamma)$, where the infimum is taken over all continuous maps $\gamma\colon[a,b]\to\mc Q$ with $\gamma(a) = q$ and $\gamma(b) = p$.
A \textbf{length space} is a metric space $(\mc Q, d)$ with $d  = d_{\ms i}$.
Now, let $(\mc Q, d)$ be a length space.
A continuous map $\gamma \colon [a, b] \to \mc Q$ is called \textbf{shortest path} if $L(\gamma) \leq L(\tilde\gamma)$ for all continuous maps $\tilde\gamma \colon [\tilde a, \tilde b] \to \mc Q$ with $\gamma(a) = \tilde\gamma(\tilde a)$ and $\gamma(b) = \tilde\gamma(\tilde b)$.
A continuous map $\gamma \colon [a, b] \to \mc Q$ is \textbf{locally minimizing} if for every $t\in[a,b]$ there is $\epsilon>0$ such that $\gamma_{\vert[t-\epsilon, t+\epsilon]}$ is a shortest path. 
A continuous map $\gamma \colon [a, b] \to \mc Q$ has \textbf{constant speed} if there is $v \geq 0$ such that for every $t\in[a,b]$ there is $\epsilon>0$ such that $L(\gamma_{\vert[t-\epsilon, t+\epsilon]}) = 2v\epsilon$. 
A \textbf{geodesic} is a locally minimizing continuous map with constant speed.
A \textbf{minimizing geodesic} between two points $q,p\in\mc Q$ is a geodesic $\gamma \colon [a, b] \to \mc Q$ with $L(\gamma) = d(\gamma(a), \gamma(b))$ and $\gamma(a) = q$, $\gamma(b) = p$.
A geodesic $\gamma \colon [a, b] \to \mc Q$ is \textbf{extendible} (through both ends) if there is $\epsilon > 0$ and a geodesic $\tilde\gamma \colon [a-\epsilon, b+\epsilon] \to \mc Q$ such that $\tilde\gamma_{\vert[a,b]} = \gamma$. The tuple $(\mc Q, d)$ is a \textbf{geodesic space} if there is a connecting geodesic for every pair of points. A geodesic space $(\mc Q, d)$ is \textbf{geodesically complete}, if it is complete and all geodesics are extendible.

A \textbf{Hadamard space} is a nonempty complete metric space $(\mc Q, d)$ such that for all $q,p \in\mc Q$, there is $m \in\mc Q$ such that $d(y, m)^2 \leq \frac12 d(y,q)^2 + \frac12 d(y,p)^2 - \frac14d(q,p)^2$ for all $y\in\mc Q$. In Hadamard spaces, all geodesics are minimizing. Hilbert spaces and Riemannian manifolds of nonpositive sectional curvature are Hadamard spaces. Hadamard spaces are also called global NPC-spaces, complete $CAT(0)$ spaces or Alexandrov spaces of nonpositive curvature.

An \textbf{Alexandrov spaces of nonnegative curvature} is a geodesic space $(\mc Q, d)$ such that for all $q,p \in\mc Q$, there is $m \in\mc Q$ such that $d(y, m)^2 \geq \frac12 d(y,q)^2 + \frac12 d(y,p)^2 - \frac14d(q,p^2)$ for all $y\in\mc Q$. More generally Alexandrov spaces can be defined with an arbitrary curvature bound. They generalize Riemannian manifolds with a bound on the sectional curvature.
\end{appendix}
%
\def\arxiv#1{%
  \href{http://arxiv.org/abs/#1}{arXiv: #1}%
}
\bibliographystyle{alpha}
\bibliography{literature}
\end{document}